\documentclass[11pt]{amsart}
\usepackage{graphicx,verbatim,lineno,titletoc}
\usepackage{amssymb,mathrsfs}
\usepackage{amsmath}
\usepackage{adjustbox}
\usepackage{calc}
\usepackage{array}
\usepackage[font=small]{caption}
\usepackage[T1]{fontenc}
\usepackage{fourier}
\usepackage[english]{babel}
\usepackage{appendix}
\usepackage{enumerate}
\usepackage[english]{babel}
\usepackage{multirow}
\usepackage{float}
\usepackage{enumitem}
\usepackage[numbers]{natbib}
\usepackage[all]{xy}
\usepackage{tikz}
\usetikzlibrary{shapes.multipart,patterns,arrows}
\usepackage{hyperref}
\hypersetup{colorlinks,linkcolor={red},citecolor={olive},urlcolor={red}}
\usepackage[top=1.3in, left=1.5in, right=1.5in, bottom=1.5in]{geometry}
\usepackage{mathtools}
\mathtoolsset{showonlyrefs}
\usepackage{csquotes}
\usepackage[notcite,notref]{}
\usepackage{blkarray}
\usepackage[utf8]{inputenc}

\newcolumntype{M}[1]{>{\centering\arraybackslash}m{#1}}
\newcolumntype{N}{@{}m{0pt}@{}}

\def\.{\hskip.06cm}
\def\ts{\hskip.03cm}

\newcommand{\vertiii}[1]{{\left\vert\kern-0.25ex\left\vert\kern-0.25ex\left\vert #1 
		\right\vert\kern-0.25ex\right\vert\kern-0.25ex\right\vert}}

\def\T{\mathbf{T}}

\def\balpha{\boldsymbol{\alpha}}
\def\bbeta{\boldsymbol{\beta}}

\def\z{\mathbf{z}}
\def\y{\mathbf{y}}

\def\supp{\textup{supp}}

\def\<{\langle}
\def\>{\rangle}

\def\Z{ {\text {\rm Z} } }

\def\0{{\mathbf 0}}

\def\.{\hskip.06cm}
\def\ts{\hskip.03cm}

\def\P{{\textup{\textsf{P}}}}

\DeclareMathOperator{\Var}{\textnormal{Var}}

\def\liminf{\mathop{\rm lim\,inf}\limits}
\def\limsup{\mathop{\rm lim\,sup}\limits}

\def\Z{\mathbb{Z}}

\def\R{\mathbb{R}}

\def\E{\mathbb{E}}
\def\P{\mathbb{P}}

\def\T{\mathcal{T}}

\def\r{\mathtt{r}}

\def\eps{\varepsilon}

\def\X{\mathbf{X}}

\def\a{\mathbf{a}}
\def\b{\mathbf{b}}
\def\r{\mathbf{r}}
\def\c{\mathbf{c}}

\def\x{\mathbf{x}}

\usepackage{graphicx}

\makeatletter
\newcommand*\rel@kern[1]{\kern#1\dimexpr\macc@kerna}
\newcommand*\widebar[1]{%
	\begingroup
	\def\mathaccent##1##2{%
		\rel@kern{0.8}%
		\overline{\rel@kern{-0.8}\macc@nucleus\rel@kern{0.2}}%
		\rel@kern{-0.2}%
	}%
	\macc@depth\@ne
	\let\math@bgroup\@empty \let\math@egroup\macc@set@skewchar
	\mathsurround\z@ \frozen@everymath{\mathgroup\macc@group\relax}%
	\macc@set@skewchar\relax
	\let\mathaccentV\macc@nested@a
	\macc@nested@a\relax111{#1}%
	\endgroup
}

\newcommand{\tr}{\textup{tr}}

\DeclareMathOperator{\diag}{diag}

\DeclareMathOperator*{\argmax}{arg\,max}
\DeclareMathOperator*{\argmin}{arg\,min}
\DeclareMathOperator*{\VEC}{\mathtt{VEC}}

\usepackage{theoremref}

\newtheorem{theorem}{Theorem}
\numberwithin{theorem}{section}
\numberwithin{equation}{section}

\newtheorem{lemma}[theorem]{Lemma}
\newtheorem{prop}[theorem]{Proposition}
\newtheorem{corollary}[theorem]{Corollary}
\newtheorem{conjecture}[theorem]{Conjecture}

\newtheorem{assumption}[theorem]{Assumption}

\theoremstyle{definition}
\newtheorem{definition}[theorem]{Definition}

\newtheorem{example}[theorem]{Example}
\newtheorem{remark}[theorem]{Remark}

\allowdisplaybreaks

\definecolor{hancolor}{rgb}{0.0 0.0, 1.0}
\makeatletter
\newcommand{\addresseshere}{%
	\enddoc@text\let\enddoc@text\relax
}
\makeatother

\begin{document}
	
	\title[Large random matrices with given margins]{Large random matrices with given margins  }

	\author{Hanbaek Lyu}
	\address{Hanbaek Lyu, Department of Mathematics, University of Wisconsin - Madison, WI, 53717, USA}
	\email{\texttt{hlyu@math.wisc.edu}}

	\author{Sumit Mukherjee}
	\address{Sumit Mukherjee, Department of Statistics, Columbia University, New York, NY 10027, USA}
	\email{\texttt{sm3949@columbia.edu}}

	\keywords{Random matrices, margins, contingency tables, Schr\"{o}dinger bridge, Sinkhorn algorithm, transference, concentration, empirical singular value distribution}
	\subjclass[2010]{60B20, 60C05, 05A05}

	\begin{abstract}
		We study large random matrices with i.i.d. entries conditioned to have prescribed row and column sums (margins), a problem connected to relative entropy minimization, Schrödinger bridges,  contingency tables, and random graphs with given degree sequences.  
		Our central result is a \textit{transference principle}: the  complex margin-conditioned matrix can be closely approximated by a simpler matrix whose entries are independent and drawn from an exponential tilting of the original model. The tilt parameters are determined by the sum of two potentials. We establish phase diagrams for  \textit{tame margins}, where these potentials are uniformly bounded. This framework resolves a 2011 conjecture by Chatterjee, Diaconis, and Sly on $\delta$-tame degree sequences and generalizes a sharp phase transition in contingency tables obtained by Dittmer, Lyu, and Pak in 2020. For tame margins, we show that a generalized Sinkhorn algorithm can compute the potentials at a dimension-free exponential rate. Our limit theory further establishes that for a convergent sequence of tame margins, the potentials converge 
		as fast as the margins converge.
		
		We apply this framework and obtain several key results for the conditioned matrix: The marginal distribution of any single entry is asymptotically an exponential tilting of the base measure, resolving a 2010 conjecture by Barvinok on contingency tables. 
		The conditioned matrix concentrates in cut norm around a \textit{typical table} (the expectation of the tilted model), which acts as a static Schrödinger bridge between the margins.
		The empirical singular value distribution of the rescaled matrix converges to an explicit law determined by the variance profile of the tilted model. In particular, we confirm the universality of the Marchenko-Pastur law for constant linear margins.

	\end{abstract}

	\maketitle
	
	\let\cleardoublepage\clearpage
	\vspace{-0.5cm}
	\tableofcontents

	\section{Introduction and main results}
	\label{sec:Introduction}

	In this paper, we are interested in the structure of random matrices with i.i.d. entries conditioned to have prescribed row and column sums.  Let $\mu$ be a $\sigma$-finite Borel measure on $\R$ and let 
	\begin{align}\label{eq:supp}
		A:=\inf\{{\rm supp}(\mu)\}\le \sup\{{\rm supp}(\mu)\}=:B.
	\end{align} 
	We allow for the possibility that $A=-\infty$, and/or $B=\infty$. The \textit{base model} is the product measure $\mu^{\otimes (m\times n)}$ on the set of $m\times n$ random matrices. 
	When $\mu$ is a probability measure, then the entries in the base model are independent and identically distributed as $\mu$.
	
	Let $\x=(x_{ij})$ be an $m\times n$ matrix of real entries. We define the \textit{row margin} of $\x$ as the vector $r(\x):=(r_{1}(\x),\dots,r_{m}(\x))$ with $r_{i}(\x):=\sum_{j=1}^{n} x_{ij}$; the \textit{column margin} of $\x$ is the vector $c(\x):=(c_{1}(\x),\dots,c_{n}(\x))$ with $c_{j}(\x):=\sum_{i=1}^{m} x_{ij}$. We call the pair $(r(\x), c(\x))$ the \textit{margin} of $\x$. For each $\rho\ge 0$, we let
	\begin{align}
		\mathcal{T}_{\rho}(\r,\c) := \left\{ \x\in \R^{m\times n}\,:\,  \textup{$\lVert r(\x)-\r \rVert_{1}\le \rho$ and $\lVert c(\x)-\c \rVert_{1}\le \rho$ }  \right\}
	\end{align}
	denote the set of all $m\times n$ real matrices whose margin is within $L^{1}$ distance $\rho$ from the prescribed margin $(\r,\c)$. The  \textit{transportation polytope} with margin $(\r,\c)$ is the set  $\mathcal{T}(\r,\c):=\mathcal{T}_{0}(\r,\c)$.

	A fundamental question we investigate in this work is how a random matrix drawn from the base model behaves if we condition its margin to take prescribed values. Namely, 
	\begin{align}\label{eq:main_question}
		\textit{If we condition $X\sim \mu^{\otimes (m\times n)}$ on being in $\T_{\rho}(\r,\c)$, how does it look like?}  
	\end{align}
	Since the margin is a fundamental observable for a matrix, it is natural to ask about the most likely structure of a random matrix after we observe its margin. This natural question connects to various important problems across diverse fields. When $\mu$ is the counting measure on $\Z_{\ge 0}$, $X$ is the uniformly random contingency table with given margin $(\r,\c)$, which is a fundamental object in statistics and combinatorics \cite{good1963maximum,diaconis1995rectangular}. 
	Counting the exact number contingency tables is known to be $\#\texttt{P}$-complete 
	\cite{dyer1997sampling} even for the $2\times n$ case. There is extensive literature in combinatorics on approximately counting the number of contingency tables 
	(see, e.g., \cite{canfield2007asymptotic, canfield2010asymptotic, barvinok2009asymptotic, barvinok2010approximation, barvinok2010number, barvinok2010does,barvinok2010maximum,lyu2022number}). The number of contingency tables is also closely related to the  Littlewood-Richardson coefficients \cite{colarusso2022contingency}. 
	When $\mu$ is the counting measure on $\{0,1\}$,  $X$ is a uniformly random bipartite graph with given degree sequence \cite{barvinok2010number, wu2020properties}. Further imposing symmetry and zero diagonal entries, it represents uniformly random simple graphs with a given degree sequence \cite{chatterjee2011random, barvinok2013number}. When $\mu$ is the Lebesgue measure on $\R_{\ge 0}$, then $X$ is the uniformly random nonnegative matrix from the transportation polytope $\T(\r,\c)$, which specializes to the uniformly random doubly stochastic matrix 
	\cite{chatterjee2014properties,nguyen2014random}. When $\mu$ is the Poisson distribution with unit mean, $X$ follows the multivariate hypergeometric distribution (or Fisher-Yates) with margin $(\r,\c)$ \cite{diaconis1995random}. We will see that in this case, the structure of $X$ is closely related to the static Schr\"{o}dinger bridge and entropic optimal transport  \cite{fortet1940resolution, pavon2021data}.

	A moment's thought reveals that it is not at all easy to sample such a margin-conditioned random matrix  $X$. This task is often nontrivial, as margin conditioning can induce complex correlations between the entries. For instance, in the case where  $X$  is a uniformly random contingency table, Dittmer, Lyu, and Pak \cite{dittmer2020phase} showed that the structure of $X$ exhibits a sharp phase transition as the margin varies continuously. The problem of sampling uniformly random contingency tables has been extensively studied (see, e.g., \cite{diaconis1995random, dyer1997sampling, kannan1999simple, dyer2000polynomial, morris2002improved, dyer2003approximate, cryan2006rapidly}). 
	In this literature, the Diaconis-Gangolli Markov chain \cite{diaconis1985testing} is an important sampling algorithm, but obtaining the cutoff for it remains an open problem \cite{nestoridi2020mixing}.
	
	\vspace{0.1cm} 
	We propose to approximate $X$ by  another random matrix $Y$ with \textit{independent entries} and establish the following  \textit{transference principle}:
	\begin{align}\label{eq:main_question_ans2}
		\hspace{-1cm}	\textup{\textbf{Transference}:}	
		\quad 	\textit{Events that are sufficiently rare under $Y$ are also rare under $X$}.
	\end{align}
	Hence, as far as such sufficiently rare events are concerned, one can avoid analyzing (and even sampling) $X$ altogether and simply use the model $Y$. The point is that the independence of the entries in $Y$ makes it much easier to analyze than $X$. 
	Such a comparison model $Y$ can be characterized from dual  perspectives of maximum likelihood (parametric) and minimum relative entropy (non-parametric):
	
	\begin{description}[itemsep=0.2cm, leftmargin=0.55cm]

		\item[1.] (\textit{Maximum Likelihood Perspective}): $Y$ 
		is the random matrix with independent entries obtained by a rank-one tilting of the base model $\mu^{\otimes (m\times n)}$, where the tilt parameters are obtained by a maximum-likelihood procedure based on the observable margin $(\r,\c)$.

		\item[2.] (\textit{Minimum Relative Entropy Perspective}): 
		Among the class of random matrices with independent entries and expected margin $(\r,\c)$, $Y$ has the minimum relative entropy from the base model $\mu^{\otimes (m\times n)}$. 
	\end{description}
	
	\noindent The equivalent characterizations of $Y$ above are in \textit{Kantorovich duality} as in the static Schr\"{o}dinger bridge/entropic optimal transport and their dual formulations. We give precise statements of our main results in the following subsections.

	\vspace{0.2cm}
	\subsection{Construction of the conditional probability measures}

	We construct the law of the margin-conditioned random matrix $X$ in \eqref{eq:main_question}, which we denote $\lambda_{\r,\c,\rho}$, under either of the following two (non-exclusive) assumptions. This covers all examples that we consider in this work and also all instances in the literature that we are aware of.
	
	\begin{assumption}\label{assumption:weak_transference}
		$\mu^{\otimes(m\times n)}(\T_{\rho}(\r,\c))\in (0,\infty)$. 
	\end{assumption}
	
	\begin{assumption}\label{assumption:strong_transference}
		$\rho=0$	 or $\mu^{\otimes(m\times n)}(\T_{\rho}(\r,\c))=0$. Furthermore, let $\pi:\R^{m\times n}\rightarrow \R^{m+n}$ denote the map that sends a matrix $\x$ to its margin $(r(\x),c(\x))$. Let $\nu=\pi_{\#}(\mu^{\otimes(m\times n)})$ denote the pushforward $\mu^{\otimes(m\times n)}$ under $\pi$. Then $\nu$ is $\sigma$-finite. 
	\end{assumption}
	
	Under Assumption \ref{assumption:weak_transference}, we simply define $\lambda_{\r,\c,\rho}$  by normalizing the product measure $\mu^{\otimes (m\times n)}$ restricted on $\T_{\rho}(\r,\c)$. Under Assumption \ref{assumption:strong_transference}, we need to condition on an event of measure zero $\mu^{\otimes(m\times n)}$. To handle this, we use the `disintegration approach' for constructing regular conditional probabilities, which is nicely described by Chang and Pollard \cite{chang1997conditioning}. Namely, we set $\lambda_{\r,\c,\rho}:=\lambda_{\r,\c,0}$, where the conditional probability measures $\lambda_{\r,\c}:=\lambda_{\r,\c,0}$ are constructed for $\nu$-almost all margins $(\r,\c)$ 
	via disintegrating the product measure $\mu^{\otimes (m\times n)}$ using the margin map $\x\mapsto (r(\x), c(\x))$. (See Sec. \ref{sec:strong_transference_pf} for more details.) Frequently we will regard the pushforward measure $\nu$ on the margins as a measure on $\R^{m+n-1}$ by identifying a margin $(\r,\c)$ with the $m+n-1$ dimensional vector $(\r(1),\cdots,\r(m-1),\c(1),\cdots,\c(n))$ omitting $\r(m)$.

	\subsection{Exponential tilting and the typical table} 
	\label{sec:exp_tilting}

	We first introduce some notations on exponential tilting and the parameterized comparison model. Throughout this paper, we use `increasing' (resp., `decreasing') and `non-decreasing' (resp., `non-increasing') interchangeably.

	\begin{definition}[Exponential tilting]\label{def:tilting}
		Define the set $\Theta$  of all allowed tilt parameters for $\mu$:
		\begin{align}\label{eq:def_Theta}
			\Theta:=\Big\{\theta\in \R: \int e^{\theta x}d\mu(x)<\infty\Big\}. 
		\end{align}
		Let $\Theta^\circ$ be the interior of $\Theta$. For any $\theta\in \Theta^\circ$, let $\mu_\theta$ denote the tilted  probability measure given by
		\begin{align}\label{eq:def_tilt}
			\frac{d\mu_\theta}{d\mu}(x)=e^{\theta x-\psi(\theta)},\quad \psi(\theta):=\log\int e^{\theta x}d\mu(x).
		\end{align}
		Then we have $	\E_{\mu_\theta}[X]=\psi'(\theta)$ and $ \Var_{\mu_\theta}(X)=\psi''(\theta)>0$, so  the function $\psi'(.):\Theta^\circ\mapsto (A,B)$ is strictly increasing, and has a strictly increasing inverse $\phi(.):(A,B)\mapsto \Theta^\circ$ satisfying  $\phi(\psi'(\theta))=\theta$ for all $\theta\in \Theta^\circ$. Note that $\Theta^{\circ}=(\phi(A),\phi(B))$.	Throughout this paper, we assume the measure $\mu$ on $\R$ is so that the set $\Theta$ of tilt parameters in \eqref{eq:def_Theta} is 
		nonempty and not a singleton. Then it follows from H\"{o}lder's inequality that $\Theta$ is a non-empty interval. 
	\end{definition}

	Next, we introduce an exponential tilting of the base i.i.d. model where the tilt parameters are parameterized by the direct sum of two vectors. These parameters will be tuned to achieve the prescribed expected margin. For two vectors $\a,\b$, let $\a\oplus \b$ denote the  matrix whose $(i,j)$ coordinate is $\a(i)+\b(j)$. For a univariate function $\varphi:\R\rightarrow \R$, we denote by $\varphi(\a\oplus \b)$ the matrix whose $(i,j)$ entry is $\varphi(\a(i)+\b(j))$. 
	
	\begin{definition}[The $(\balpha,\bbeta)$-model]\label{def:alpha_beta_model}
		Let  $\{\mu_\theta\}_{\theta\in \Theta^\circ}$ be probability measures on $\R$, as introduced in \eqref{eq:def_tilt}, and let vectors $\balpha=(\alpha_{1},\dots,\alpha_{m})\in \R^{m}$ and $\bbeta=(\beta_{1},\dots,\beta_{n})\in \R^{m}$ be such that $\alpha_i+\beta_j\in \Theta^\circ$ for all $i\in [m], j\in [n]$. The \textit{$(\balpha,\bbeta)$-model} is a $m\times n$ random matrix $Y=(Y_{ij})$ where the entries are independent and $Y_{ij}\sim \mu_{\alpha_i+\beta_j}$. In this case, we write $Y\sim \mu_{\balpha\oplus\bbeta}$. 
	\end{definition}
	
	We begin by observing that the likelihood of observing a matrix $\x\in \R^{m\times n}$ under the $(\balpha,\bbeta)$-model depends only on the margin of $\x$. Indeed, suppose we observed a matrix $\x=(x_{ij})\in \R^{m\times n}$ from the model $\mu_{\balpha\oplus\bbeta}$. Let $(\r,\c)$ be the margins of $\x$. 
	Note that the  log-likelihood of observing $\x$ under $Y\sim \mu_{\balpha, \bbeta}$, with respect to base measure $\mu^{\otimes (m\times n)}$ is 
	\begin{align}
		\sum_{i=1}^m \sum_{j=1}^n \Big[\x_{ij}(\balpha(i)+\bbeta(j)) - \psi(\balpha(i)+\bbeta(j)) \Big]  
		= \langle \r, \balpha \rangle +  \langle \c, \bbeta \rangle  -  \sum_{i,j}\psi(\balpha(i)+\bbeta(j)). 
	\end{align}
	Hence, given an observed table $\x$ from the $(\balpha,\bbeta)$-model, the row and column sums of $\x$ give a sufficient statistic for the parameters $\balpha,\bbeta$. This is analogous to the fact that the degree sequence of an observed graph under the $\bbeta$-model is a sufficient statistic for $\bbeta$, see \cite{chatterjee2011random}.
	Consequently, conditional on the margin, the distribution of the $(\balpha,\bbeta)$-model is free of $\x$. 
	Given a single observation $\x\in \T(\r,\c)$, the maximum likelihood estimate (MLE) of $(\balpha,\bbeta)$ is obtained by maximizing the log-likelihood function above. 
	We formulate this discussion below.

	\begin{definition}[MLE for margin $(\r,\c)$]\label{def:MLE_typical}
		Fix an $m\times n$ margin $(\r,\c)$. The \textit{MLE} of  $(\balpha,\bbeta)$ for margin $(\r,\c)$, 
		is a solution to the following concave maximization problem:
		\begin{align}\label{eq:typical_Lagrangian}
			\sup_{\balpha, \bbeta} \left( g^{\r,\c}(\balpha, \bbeta) := 
			\langle \r,\balpha \rangle  + \langle \c, \bbeta \rangle 
			-\sum_{i,j}\psi(\balpha(i)+\bbeta(j) )  \right),
		\end{align}
		where we optimize over the open set where $\balpha(i)+\bbeta(j)\in \Theta^{\circ}$ for all $i\in [m],j\in [n]$.
		An MLE $(\balpha,\bbeta)$ for $(\r,\c)$ is said to be a  \textit{standard MLE} if $\langle \balpha,\mathbf{1} \rangle=0$ and is denoted as $(\balpha^{\r,\c}, \bbeta^{\r,\c})$. 
	\end{definition}

	A careful reader may wonder if the $(\balpha,\bbeta)$-model is a bit too restrictive for describing the structure of margin-conditioned random matrices. For instance, what if we use all $mn$ independent tilt parameters $\theta_{ij}$ for each entry instead of the $m+n$ ones in the $(\balpha,\bbeta)$-model? Parameterizing $\theta_{ij}$ as the mean $z_{ij}$ after tilt through the relation $\theta_{ij}=\phi(z_{ij})$, the optimal such $mn$ tilt parameters are given by solving the following relative entropy minimization problem.

	\begin{definition}[Typical table]\label{def:typical_table}
		Fix a $m\times n$ margin $(\r,\c)$. The \textit{typical table} $Z^{\r,\c}$ for margin $(\r,\c)$ with respect to the base measure $\mu$ is defined by 
		\begin{align}\label{eq:typical_table_opt}
			Z^{\r,\c} \. := \. \argmin_{Z=(z_{ij})\in \mathcal{T}(\r,\c) } \,  \left[H(Z):=\sum_{i,j}\, D(\mu_{\phi(z_{ij})} \Vert  \mu) \right],
		\end{align}
		where 
		$D(\mu_{\theta} \lVert \mu)$ is the \textit{relative entropy} from $\mu$ to the tilted probability measure  $\mu_{\theta}$ defined as   
		\begin{align}\label{eq:def_rel_entropy}
			D(\mu_{\theta} \lVert \mu) := \begin{cases} \int_{x\in \R} \log\left(  \frac{d\mu_{\theta}}{d\mu}(x) \right)\,d\mu_{\theta}(x) =
				\theta \psi'(\theta)  - \psi(\theta) &  \textup{if $\theta \in (\phi(A),\phi(B))$} \\
				\infty & \textup{otherwise}.
			\end{cases}
		\end{align}
	\end{definition}

	Note that when $\mu$ is a probability measure, the relative entropy above agrees with the Kullback-Leibler divergence from $\mu$ to $\mu_{\theta}$ so  $D(\mu_{\theta}\Vert \mu)\ge 0$. However, this quantity need not be nonnegative in general when $\mu$ is not a probability measure. For instance, if $\mu$ is the counting measure on nonnegative integers (see Ex. \ref{ex:counting_base_measure}), then $D(\mu_{\theta}\lVert \mu)$ equals the negative entropy of the geometric distribution $\mu_{\theta}$, so it is nonpositive and is not bounded from below.

	On the one hand, the MLE problem in \eqref{eq:typical_Lagrangian} seeks to estimate the unknown parameters $(\balpha,\bbeta)$ that best describe the random matrix with a given margin through the $(\balpha,\bbeta)$-model. On the other hand, the typical table problem in \eqref{eq:typical_table_opt} seeks to find the best $m\times n$ mean matrix in the transportation polytope that achieves the smallest possible relative entropy  when the law of each entry is exponentially tiled. A key observation in this work is that these two problems are strongly dual to each other, analogously to the Kantorovich duality in the Schr\"{o}dinger bridge theory.
	
	\begin{theorem}[Strong duality between typical table and MLE]\label{thm:strong_duality_simple}
		Let $(\r,\c)$ be an $m\times n$ margin. 
		Then an MLE $(\balpha,\bbeta)$ for $(\r,\c)$ exists if and only if the typical table $Z^{\r,\c}$ exists if and only if $\T(\r,\c)\cap (A,B)^{m\times n}\ne \emptyset$. If they exist, the following implications hold:
		\begin{align}\label{eq:lem_MLE_typical_duality1}
			\textup{$(\balpha,\bbeta)$ is an MLE} \quad \Longleftrightarrow \quad \E[\mu_{\balpha\oplus \bbeta}]=\psi'(\balpha\oplus \bbeta) \in \T(\r,\c) 	 \quad \Longleftrightarrow \quad \textup{$Z^{\r,\c}= \psi'(\balpha\oplus \bbeta)$. }
		\end{align}
		Furthermore, the typical table and the MLE problems are in strong duality: 
		\begin{align}\label{eq:strong_duality}
			\inf_{Z\in \mathcal{T}(\r,\c)} \, H(Z) = 
			\sup_{\balpha, \bbeta} \, g^{\r,\c}(\balpha, \bbeta). 
		\end{align}
	\end{theorem}

	An implication of Thm. \ref{thm:strong_duality_simple} above is that the $(\balpha,\bbeta)$-model is the best possible among all possible entry-wise exponential tilting of the base model. In fact, when $\mu$ is a probability measure, it is the best possible among all random matrix ensembles, as it is the \textit{information projection} 
	of the base measure $\mu^{\otimes (m\times n)}$ onto the set of all probability measures on $m\times n$ real matrices constrained to have expected margin $(\r,\c)$ (see Sec. \ref{sec:min_rel_entropy_discussion} for more discussion).

	The behavior of an $(\balpha,\bbeta)$-model depends crucially on how far the entries of $\balpha\oplus \bbeta$ are away from the boundary values $\phi(A)$ and $\phi(B)$. 
	This leads to the following notion of `tameness' of a margin. With a slight abuse of notations, we say $a\le M \le b$ for a matrix $M$ and scalars $a,b$ if every entry of $M$ lies in $[a,b]$. 
	
	\begin{definition}[Tame margins]\label{def:ab}
		Fix $\delta>0$ and let $\mathcal{M}^{\delta}=\mathcal{M}^{\delta}(\mu,m,n)$ denote the set of all $m\times n$ \textit{$\delta$-tame} margins $(\r,\c)$, that is, the MLE $(\balpha,\bbeta)$ exists and its  entries satisfy
		\begin{align}\label{eq:ab}
			\quad 	A_\delta:=\max\Big(A+\delta,-\frac{1}{\delta}\Big)\le \psi'(\balpha\oplus \bbeta ) \le \min\Big(B-\delta,\frac{1}{\delta}\Big)=:B_\delta. 
		\end{align}
	\end{definition}
	
	According to Thm. \ref{thm:strong_duality_simple}, $\delta$-tameness of a margin $(\r,\c)$  can be equivalently defined as the typical table $Z^{\r,\c}$ taking all entries from $[A_{\delta},B_{\delta}]$. 
	Also we remark  that, since $\balpha\oplus \bbeta$ belongs to $(\phi(A),\phi(B))^{m\times n}$ by definition, any margin $(\r,\c)$ with an MLE (or typical table)  is always $\delta$-tame for some $\delta>0$ that may depend on $m$ and $n$.
	The important question is whether a sequence of margins is uniformly $\delta$-tame for a fixed $\delta>0$. `Cloning' a given margin provides a simple way to generate a family of $\delta$-tame margins.

	\begin{example}[Cloned margins]\label{ex:cloned_margin}
		Let $(\r_{0},\c_{0})$  be a $a \times b$ margin with an MLE $(\balpha_{0},\bbeta_{0})$. The  \textit{$k$-cloning} of $(\r_{0},\c_{0})$ is the $k a\times k b$ margin $(\r,\c)$ with $\r=k\r_{0}\otimes \mathbf{1}_{k}$ and $\c =  k\c_{0}\otimes \mathbf{1}_{k}$, where $\otimes$ denotes the Kronecker product, i.e., $\r$ and $\c$ repeat $k\r_{0}$ and $k\c_{0}$ $k$ times, respectively. Then $(\r,\c)$ has MLE $(\balpha_{0}\otimes \mathbf{1}_{k}, \bbeta_{0}\otimes \mathbf{1}_{k})$. Hence if $(\r_{0},\c_{0})$ is $\delta$-tame, then its $k$-clonings for all $k\ge 1$ are all $\delta$-tame. For instance, the constant linear margin $(\r,\r)$ with $\r=n a \mathbf{1}_{n}$ for any $a\in (A,B)$ is the $n$-cloning of the $1\times 1$ margin $(a,a)$ and is $\delta$-tame for any $\delta>0$ such that $A_{\delta}\le \psi'(a)\le B_{\delta}$. 
	\end{example}
	
	Establishing $\delta$-tameness for a general margin $(\r,\c)$ beyond cloned ones turns out to be quite delicate. In Section \ref{sec:suff_cond_tame}, we establish phase diagrams and characterizations of the  $\delta$-tame margins.

	\vspace{0.2cm}
	\subsection{Transference principles} 
	\label{sec:transference_main}
	
	In this section, we give precise statements of the transference principle \eqref{eq:main_question_ans2}. The following `weak transference principle' concerns the case when the measure of the conditioning set $\T_{\rho}(\r,\c)$ is nonzero (Assumption \ref{assumption:weak_transference}).

	\begin{theorem}[Weak transference]\label{thm:transference}
		Let $(\r,\c)$ be an $m\times n$ margin with an MLE $(\balpha,\bbeta)$
		and let $X\sim \lambda_{\r,\c,\rho}$ under Assumption \ref{assumption:weak_transference}. Let $\delta>0$ be small enough so that $(\r,\c)$ is $\delta$-tame and let  $Y\sim \mu_{\balpha\oplus \bbeta}$. Then there exists  a  constant $C_{1}=C_{1}(\mu,\delta)>0$ such that    for each Borel set $\mathcal{E}\subseteq \R^{m\times n}$, 
		\begin{align}\label{eq:transference_1}
			\P(X\in \mathcal{E}) \le \exp\left(  C_{1} \rho  \right) \P\left( Y\in \T_{\rho}(\r,\c) \right)^{-1}  \P\left( Y\in \mathcal{E} \right).
		\end{align}
		Furthermore, if $\rho\ge C_{2} \sqrt{mn(m+n)}$ and $m,n\ge C_{2}$ for some constant $C_{2}=C_{2}(\mu,\delta)>0$, then
		\begin{align}\label{eq:transference_11}
			\P(X\in \mathcal{E}) \le 2 \exp\left(  C_{1} \rho  \right)  \P\left( Y\in \mathcal{E} \right).
		\end{align}
	\end{theorem}

	The upper bound in \eqref{eq:transference_1} is useful only if the product of the first two terms is not too large.  For this, there must be a sufficient probability $\P(Y\in \T_{\rho}(\r,\c))$ that the maximum likelihood tilted model $Y$ satisfying margin $(\r,\c)$ with $L^{1}$-error $\rho$. For the second part, we show $\P(Y\in \T_{\rho}(\r,\c))\ge 1/2$ if  $\rho\approx \sqrt{mn(m+n)}$. 
	The reason is that the margin of $Y$ has expectation $(\r,\c)$ and the entries of $Y$ are independent with comparably tiled sub-exponential distributions $\mu_{\balpha(i)+\bbeta(j)}$. Hence by Bernstein's inequality, any fixed signed sum of its row/column sum fluctuates by  $O(\sqrt{mn})$ with an exponential tail. Since there are at most $2^{m+n}$ such signed sums, the $L^{1}$ fluctuation of the margin of $Y$ is of order $O(\sqrt{mn(m+n)})$. 
	Hence $Y\in \T_{\rho}(\r,\c)$ with a large enough probability. However, the overal transference cost $2\exp(C\rho)$ is still large, which comes from approximating the margin of $Y$ by $(\r,\c)$, where we only know it is within $L^{1}$-distance $\rho$ away from $(\r,\c)$.

	\vspace{0.1cm}
	
	Next, we establish several `strong transference principles' for exact margin conditioning $\rho=0$ or $\T_{\rho}(\r,\c)$ having measure zero (Assumption \ref{assumption:strong_transference}). In order to state our results, we need to introduce some notations. For each matrix $\z\in \R^{m\times n}$, let $\overline{\z}$ denote its $(m-1)\times (n-1)$ submatrix obtained by deleting the last row and column from $\z$ and let 
	$\check{\z}$ denote the $m+n-1$ entries of $\z$ in its last row and column. 
	Note that $X$ is completely determined by $\overline{X}$ due to the exact margin constraint. More precisely, 
	define the `completion map' $\Gamma_{\r,\c}:\R^{(m-1) \times (n-1)}\rightarrow \R^{m\times n}$ by 
	\begin{align}\label{eq:def_completion_map}
		\Gamma_{\r,\c}(\overline{\x})_{ij}:= \begin{cases}
			\overline{\x}_{ij} & \textup{if $1\le i < m$ and $1\le j < n$}\\
			\r(i) - \overline{\x}_{i\bullet} & \textup{for $1\le i < m$ and $j=n$} \\
			\c(j) - \overline{\x}_{\bullet j} & \textup{for $i=m$ and $1\le j < n$}  \\
			\overline{\x}_{\bullet\bullet} - \left(\sum_{i=1}^{m-1} \r(i)\right) +\c(n) 
			& \textup{if $(i,j)=(m,n)$}. 
		\end{cases}
	\end{align}
	Here, $\overline{\x}_{i\bullet}$, $\overline{\x}_{\bullet j}$, and $\overline{\x}_{\bullet\bullet}$ denote the $i$th row sum,  the $j$th column sum, and the total sum of $\overline{\x}$, respectively. 
	Then given an $(m-1)\times (n-1)$ matrix $\overline{\x}$, $\Gamma_{\r,\c}(\overline{\x})$ is the unique $m\times n$ matrix with margin $(\r,\c)$ such that $\overline{\Gamma_{\r,\c}(\overline{\x})}=\overline{\x}$.
	In particular, $X=\Gamma_{\r,\c}(\overline{X})$. 
	We let $\check{\Gamma}_{\r,\c}(\cdot)$ denote the entries in the last row and column of $\Gamma_{\r,\c}(\cdot)$. 
	
	The result below is our general transference principle for exactly margin-conditioned matrices.

	\begin{theorem}[Strong transference]\label{thm:second_transference} 
		Suppose Assumption \ref{assumption:strong_transference} holds.
		For $\nu$-almost all margins $(\r,\c)$ with an MLE $(\balpha,\bbeta)$, the following hold: 
		\begin{description}[itemsep=0.1cm] 
			\item[(i)]  
			Let $X\sim \lambda_{\r,\c}$ and let $Y\sim \mu_{\balpha\oplus \bbeta}$. Then $X$ has the same distribution as $Y$ conditional on $Y\in \T(\r,\c)$. 
			
			\item[(ii)] 
			Let $\nu_{\overline{\y}}(\cdot)$
			denote the law of $(r(Y),c(Y))$ given $\overline{Y}=\overline{\y}$ and let $p_{\overline{\y}}(\cdot)$ denote the Radon-Nikodym derivative of $\nu_{\overline{\y}}$ w.r.t. the unconditional law $\nu_{Y}$ of $(r(Y), c(Y))$. Then for each bounded measurable function $h:\R^{(m-1)\times(n-1)}\rightarrow \R$, \begin{align}\label{eq:strong_transference_exact}
				\E[ h(\overline{X})] = \E\left[ p_{\overline{Y}}(\r,\c) h(\overline{Y})  \right]\le \big(\sup_{\overline{\y}} p_{\overline{\y}}(\r,\c)\big) \,\,   \E[ h(\overline{Y})]. 
			\end{align} 
			Furthermore, if $\overline{\y}\mapsto p_{\overline{\y}}(\r,\c)$ is proportional to some function $q_{\overline{\y}}(\r,\c)$, then  \begin{align}\label{eq:strong_transference_exact2}
				\sup_{\overline{\y}} p_{\overline{\y}}(\r,\c) =  \big( \sup_{\overline{\y}} q_{\overline{\y}}(\r,\c) \big)\, \E[q_{\overline{Y}}(\r,\c)]^{-1}.
			\end{align}

		\end{description}

	\end{theorem}
	
	Following our high-level transference principle in \eqref{eq:main_question_ans2}, we expect to establish results describing the structure of $X$ based on the approximation scheme $X\approx Y$, where $Y$ is the (unconditional) maximum-likelihood $(\balpha,\bbeta)$-model for the margin $(\r,\c)$. 
	However,  one should not expect this to give a useful approximation for arbitrary base measure and margin. For instance, consider  $\mu=\textup{Uniform}(\{0,1,\sqrt{2}\})$ and symmetric constant margin $(\r,\r)$ with  $\r=n\mathbf{1}_{n}$. Then $(\r,\r)$ has a positive mass under $\nu$,  
	$\overline{X}=\mathbf{1}\mathbf{1}^{\top}$ almost surely, but $\overline{Y}=\mathbf{1}\mathbf{1}^{\top}$ with probability exponentially small in $n^{2}$. For this example, the supremum of $p_{\overline{\y}}(\r,\c)$ over $\overline{\y}$ is $\P(Y\in \T(\r,\c))^{-1}$, which is exponentially large in $n^{2}$. 
	(See Ex. \ref{ex:counterexample} for details.) 
	
	Next, deduce a useful corollary of Thm. \ref{thm:second_transference} by obtaining a computable form of the proportionality $q_{\overline{\y}}(\r,\c)$. Note that the law $\nu_{\overline{\y}}$ of the margin $(r(Y),c(Y))$ given $\overline{Y}=\overline{\y}$ is the pullback of the law of $\check{Y}$ under the one-to-one map $(\r,\c)\mapsto \check{\Gamma}_{\r,\c}(\overline{\y})$: 
	\begin{align}\label{eq:nu_ybar_pullback}
		\nu_{\overline{\y}}(d(\r,\c)) = \bigotimes_{\textup{$i=m$ or $j=n$}} \mu_{\balpha(i)+\bbeta(j)}\left( d\check{\Gamma}_{\r,\c}(\overline{\y})_{ij} \right). 
	\end{align}
	Hence locally $ \nu_{\overline{\y}}$ is the product of shifts of the tilted measures  $\mu_{\balpha(i)+\bbeta(j)}$ for $i=m$ or $j=n$. Thus, if these shfted and tilted measures have density w.r.t. a common measure, say $\zeta$, on $\R$, then we can easily compute the Radon-Nikodym derivative  $d\nu_{\overline{\y}}/d\zeta^{\otimes(m+n-1)}$. In this case, it would be natural to suspect that $\overline{\y}\mapsto p_{\overline{\y}}(\r,\c)$ is proportional to this Radon-Nikodym derivative. This conclusion indeed holds if $d\zeta^{\otimes(m+n-1)}/d\nu$ makes sense at least locally near each margin $(\r,\c)$ in the support of $\nu$ (see Cor. \ref{cor:second_transference_density}). This latter condition holds for a wide range of discrete and continuous measures as stated in Cor.  \ref{cor:transfer_discrete_Leb} below.

	\begin{corollary}\label{cor:transfer_discrete_Leb}
		Keep the same setting in Thm. \ref{thm:second_transference}. 
		Suppose $\mu$ has density $p\ge 0$ w.r.t. $\zeta$, which is either the counting measure on $\Z$ or the Lebesgues measure on $\R$. In the continuous case, suppose  $\{p>0\}$ is open. Then \eqref{eq:strong_transference_exact2} holds with $q_{\overline{\y}}(\r,\c)=d\nu_{\overline{\y}}/d\zeta^{\otimes(m+n-1)}$. Furthermore, suppose the Radon-Nikodym derivatives of $\mu_{\theta}$ w.r.t. $\zeta$ for $\theta\in [\phi(A_{\delta}),\phi(B_{\delta})]$ is uniformly upper bounded by some constant $M>0$. 
		Then for $\nu$-almost all margins $(\r,\c)$ with an MLE $(\balpha,\bbeta)$, \eqref{eq:strong_transference_exact} holds with \begin{align}\label{eq:p_rc_upper_bd_expectation_density}
			\sup_{\overline{\y}} p_{\overline{\y}}(\r,\c)  \le M^{m+n-1}  \exp(-g^{\r,\c}(\balpha,\bbeta))	 \left(  	 \int  \prod_{ i,j  } p( \Gamma_{\r,\c}(\overline{\x})_{ij} )  \, \, \zeta^{\otimes (m-1)\times (n-1)}(d\overline{\x})  \right)^{-1}.
		\end{align}
		Furthermore, we can take $M=1$  in the discrete case.
	\end{corollary}

	When $\mu$ is the Lebesgue measure on $\R_{\ge 0}$, notice that the integral in \eqref{eq:p_rc_upper_bd_expectation_density} is the volume of the transportation polytope $\T(\r,\c)\cap \R_{\ge 0}^{m\times n}$ of nonnegative matrices with margin $(\r,\c)$. Also when $\mu$ is the counting measure on $\Z_{\ge 0}$, then the same integral is the number of contingency tables with margin $(\r,\c)$. Various upper and lower bounds on these quantities have been extensively studied in the combinatorics literature in the last three decades. For our purpose, we can use the lower bounds on the number of contingency tables \cite{barvinok2009asymptotic, branden2020lower} and of the volume of the transportation polytope \cite{canfield2007asymptotic,barvinok2024quick}. 
	This gives the following strong transference result.  
	
	\begin{theorem}\label{thm:transfer_counting_Leb}
		Suppose Assumption \ref{assumption:strong_transference} holds.
		If $\mu$ is the counting measure on $[0,b)\cap \Z$ for some $b\in [0,\infty]$ or the Lebesgue measure on $\R_{\ge 0}$ and if $(\r,\c)$ is $\delta$-tame, then there exists a  constant $C=C(\mu,\delta)>0$ such that for each bounded measurable function $h:\R^{(m-1)\times(n-1)}\rightarrow \R$,
		\begin{align}\label{eq:thm_double_transfer3_uniform}
			\E[h(\overline{X})] \le   \,\, 
			\exp( C (m+n)\log (m+n) )  \,\, \E[h(\overline{Y})].
		\end{align}
	\end{theorem}
	
	For possibly non-uniform density function $p\ge 0$, the integral in \eqref{eq:p_rc_upper_bd_expectation_density} can be viewed as a weighted volume of the transportation polytope $\T(\r,\c)\cap \supp(\mu)^{m\times n}$ where the weight is proportional to the product of $p(\Gamma_{\r,\c}(\overline{\x})_{ij})$.  
	It seems that there is no known upper bound on such `weighted volume' of the transportation polytope for general density $p$. We circumvent this limitation by lower-bounding the expectation $\E[d\nu_{\overline{Y}}/d\zeta^{\otimes(m+n-1)}]$ through a probabilistic argument. This gives the following strong transference principle that holds for a wide range of discrete and continuous base measures with a looser bound on the transference cost.

	\begin{theorem}\label{thm:second_transference_density_1}
		Suppose Assumption \ref{assumption:strong_transference} holds. Assume that $\mu$ has a density $p\ge 0$ w.r.t. $\zeta$, which is either the counting measure on $\Z$ or the Lebesgue measure on $\R$. Further, assume:
		\begin{description}
			\item{\textup{(1)}} There exists a constant $\gamma>0$ such that $(m\land n)\ge (m\lor n)^{\gamma}$ and $m,n$ are sufficiently large so that  $m\land n \ge 12^{\gamma}$. 
			
			\item{\textup{(2)}} $\supp(\mu) = [A,B] \cap  \supp (\zeta)$ in both cases and 
			the Radon-Nikodym derivatives of $\mu_{\theta}$ w.r.t. $\zeta$ for $\theta\in [\phi(A_{\delta}),\phi(B_{\delta})]$ is uniformly upper bounded by some constant.

		\end{description}
		Then there exists a constant $C=C(\mu,\delta,\gamma)>0$ such that for $\nu$-almost all $\delta$-tame margins $(\r,\c)$ and for each bounded measurable function $h:\R^{(m-1)\times(n-1)}\rightarrow \R$,
		\begin{align}\label{eq:thm_double_transfer3}
			\E[h(\overline{X})] \le   \,\, 
			\exp( C(m \sqrt{n}\log n+ n \sqrt{m}\log m )\log mn ) \,\, \E[h(\overline{Y})].
		\end{align}
	\end{theorem}
	
	\begin{remark}
		In case $(m\land n) \gg \log (m\lor n)$, a minor modification of our argument for Thm. \ref{thm:second_transference_density_1} shows that the similar strong transference holds with transference cost bound $\exp(o(mn))$. 
	\end{remark}

	\begin{remark}
		For enumerating contingency tables, the term $\exp(-g^{\r,\c}(\balpha,\bbeta))$ agrees (see Thm. \ref{thm:strong_duality_simple}) with a well-known upper bound on the number of contingency tables with given margin $(\r,\c)$ \cite{barvinok2010maximum, barvinok2010number}. For the volume 
		of the transportation polytope $\T(\r,\c)\cap \R_{\ge 0}^{m\times n}$, the same term (up to an exponential term in $m\land n$) serves as an upper bound \cite{barvinok2024quick}. Therefore, the upper bound on the transference cost in  \eqref{eq:p_rc_upper_bd_expectation_density} is essentially the gap between these upper bounds and the actual volume of the polytope. For these problems, the corresponding gaps of order $O(n\log n)$ (for $m=n$) give the transference cost in Thm. \ref{thm:transfer_counting_Leb}. 
		Our results in Thm. \ref{thm:second_transference_density_1} imply similar lower bounds on the volume of the polytope with general densities $p$ other than $p(x)=\mathbf{1}(x\ge 0)$.  See Remark \ref{rmk:volume_application} for details. 
	\end{remark}

	\subsection{Limit theory for the comparison model}   
	
	Our next goal is to establish a limit theory for the margin-conditioned random matrices as the sequence of $m\times n$ margins $(\r_{m}, \c_{n})$ converges to a limiting `continuum margin' $(\r,\c)$ in a suitable sense, as $m,n$ both tend to infinity. 
	By transference, it should essentially be enough to establish such a result for the corresponding maximum-likelihood tilted models. This is precisely what we establish in this section with an explicit bound on the rate of convergence of the corresponding typical tables and rescaled MLEs. 
	
	First, let us make the convergence of margins precise. 
	\begin{definition}[Continuum margin]
		A \textit{continuum margin} is a pair $(\r,\c)$ of  integrable functions $\r,\c:(0,1]\rightarrow \R$ such that $\int_{0}^{1} \r(x)\,dx=\int_{0}^{1} \c(y)\,dy$. For a $m\times n$ discrete margin $(\r_{m}, \c_{n})$, define the corresponding continuum step margin $(\bar{\r}_{m}, \bar{\c}_{n})$ as 
		\begin{align}\label{eq:margin_function}
			\bar{\r}_{m}(t)  := n^{-1} \r_{m}( \lceil m t \rceil),\qquad 	\bar{\c}_{n}(t)  := m^{-1} \c_{n}( \lceil n t \rceil).
		\end{align}
		We say a sequence of $m\times n$ margins $(\r_{m}, \c_{n})$ \textit{converges in $L^{1}$} to a continuum margin $(\r,\c)$ if 
		\begin{align}
			\lim_{m,n\rightarrow\infty }	\lVert \r - \bar{\r}_{m} \rVert_{1} + 		\lVert \c - \bar{\c}_{n} \rVert_{1}   =0.
		\end{align}
	\end{definition}

	It will be convenient to compare matrices of different dimensions in the space of kernels. 
	
	\begin{definition}[Kernels]\label{def:kernel}
		A \textit{kernel} is an integrable function $W:[0,1]^{2}\rightarrow \R$. 
		Given an $m\times n$ matrix $A$, define a function $W_A$ on the unit square as follows: Partition the unit square $(0,1]^2$ into $m\times n$ rectangles of the form $R_{ij}:=\Big(\frac{i-1}{m}, \frac{i}{m}\Big]\times \Big(\frac{j-1}{n}, \frac{j}{n}\Big]$, for $i\in [m]$ and $j\in [n]$. Set
		\begin{align}\label{eq:def_block_kernel}
			W_A(x,y):=A_{ij}\text{ if }(x,y)\in R_{ij}.
		\end{align}
		The \textit{$p$-norm} of a kernel $W$ is defined as 
		\begin{align}\label{eq:def_p_norm}
			\lVert W \rVert_{p} :=  \left( \int_{S\times T} |W(x,y)|^{p}\, dx\,dy \right)^{1/p}. 
		\end{align}
	\end{definition}

	In Theorem \ref{thm:main_convergence_conti} below, we show that the typical tables and the rescaled MLEs 
	converge in $L^{2}$ to a limiting kernel characterized by a continuum dual variable, provided the discrete margins are uniformly $\delta$-tame for a fixed $\delta>0$. Furthermore,  the rate of convergence
	(measured in the squared $L^{2}$-distance) is at least the rate of convergence of the margins (measured in  $L^{1}$-distance). 
	
	\begin{theorem}[Scaling limit of typical tables and MLEs]\label{thm:main_convergence_conti}
		Fix $\delta>0$ and let $(\r_{m},\c_{n})$ be a sequence of $m\times n$ $\delta$-tame margins  converging to a continuum margin $(\r,\c)$  in $L^{1}$ as $m,n\to\infty$. 
		\begin{description}[itemsep=0.1cm]
			\item[(i)] There exists bounded measurable functions $\balpha,\bbeta:[0,1]\rightarrow \R$  s.t. $\int \balpha(x) dx =0$, $\phi(A_{\delta})\le \balpha\oplus \bbeta\le  \phi(B_{\delta})$, 
			and  
			the kernel $W^{\r,\c}:= \psi'(\balpha \oplus \bbeta)$
			has continuum margin $(\r,\c)$. 
			
			\item[(ii)] 	Let $C_{\delta}:=2 \max\{ | \phi(A_{\delta})|,\, | \phi(B_{\delta})| \}$, where $A_\delta,B_\delta$ are as in \eqref{eq:ab}. Then 
			\begin{align}\label{eq:main_thm_typical_rate}
				\lVert W^{\r,\c}- W_{Z^{\r_{m},\c_{n}}} \rVert_{2}^{2} \le C_{\delta} \left( \sup_{ \phi(A_{\delta})\le w \le \phi(B_{\delta}) } \psi''(w)  \right) \lVert (\r,\c) - (\bar{\r}_{m},\bar{\c}_{n}) \rVert_{1}.
			\end{align}
			Furthermore, if $(\balpha_{m},\bbeta_{n})$ is the standard MLE for margin $(\r_{m},\c_{n})$ 
			and if we define functions $\bar{\balpha}_{m}, \bar{\bbeta}_{n}:(0,1]\rightarrow \R$ as $\bar{\balpha}_{m}(x):= \balpha_{m}(  \lceil m x \rceil )$ and $\bar{\bbeta}_{n}(y):=\bbeta_{n}( \lceil n y \rceil)$,  then 
			\begin{align}\label{eq:main_thm_typical_rate_MLE}
				\lVert \balpha-\bar{\balpha}_{m} \rVert_{2}^{2} + \lVert \bbeta-\bar{\bbeta}_{n} \rVert_{2}^{2} \le C_{\delta} \left(  \frac{\sup_{ \phi(A_{\delta})\le w \le \phi(B_{\delta}) } \psi''(w) }{ \inf_{ \phi(A_{\delta})\le w \le \phi(B_{\delta}) } \psi''(w) }  \right)  \lVert (\r,\c) - (\bar{\r}_{m},\bar{\c}_{n}) \rVert_{1}. 
			\end{align}
			In particular, $\lVert W^{\r,\c}- W_{Z^{\r_{m},\c_{n}}} \rVert_{2}\rightarrow 0$ and $\lVert (\balpha,\bbeta)-(\bar{\balpha}_{m}, \bar{\bbeta}_{n})  \rVert_{2}\rightarrow 0$ as $m,n\rightarrow\infty$.

		\end{description}
	\end{theorem}

	\subsection{Phase diagram for tame margins}
	\label{sec:suff_cond_tame}
	
	The assumption of $\delta$-tameness gives good control on the typical table $Z=Z^{\r,\c}$ and on the corresponding standard MLE $(\balpha,\bbeta)$. But since this assumption is implicit, it will only be useful if we can provide a more explicit characterization on the set $\mathcal{M}^{\delta}$ of all $\delta$-tame margins for each $\delta>0$. 
	We will seek for such conditions depending only on the extreme values of the margin. Namely, for each point $(s,t)\in (A,B)^{2}$, $s\le t$ in a two-dimensional phase diagram, we ask if an arbitrary $m\times n$ margin $(\r,\c)$ satisfying  
	\begin{align}\label{eq:r_c_tame_s_t}
		\T(\r,\c)\cap(A,B)^{m\times n}\ne \emptyset \quad \textup{and}\quad 	s \le \r(i)/n, \c(j)/m \le t \quad \textup{for all $(i,j)\in [m]\times [n]$}
	\end{align}
	is $\delta$-tame for some $\delta=\delta(\mu,s,t)>0$ depending only on $\mu$, $s$, and $t$. Let $\Omega(\mu)\subseteq (A,B)^{2}$ denote the set of all such points $(s,t)$ that guarantee such uniform tameness for the base measure $\mu$. 
	Can we obtain the full phase diagram $\Omega(\mu)$ for each base measure $\mu$? See Fig. \ref{fig:PhaseDiagram} for the identified shapes of $\Omega(\mu)$ for various base measures.

	\begin{figure}[hbt]
		\begin{center}
			\includegraphics[width = 1 \textwidth]{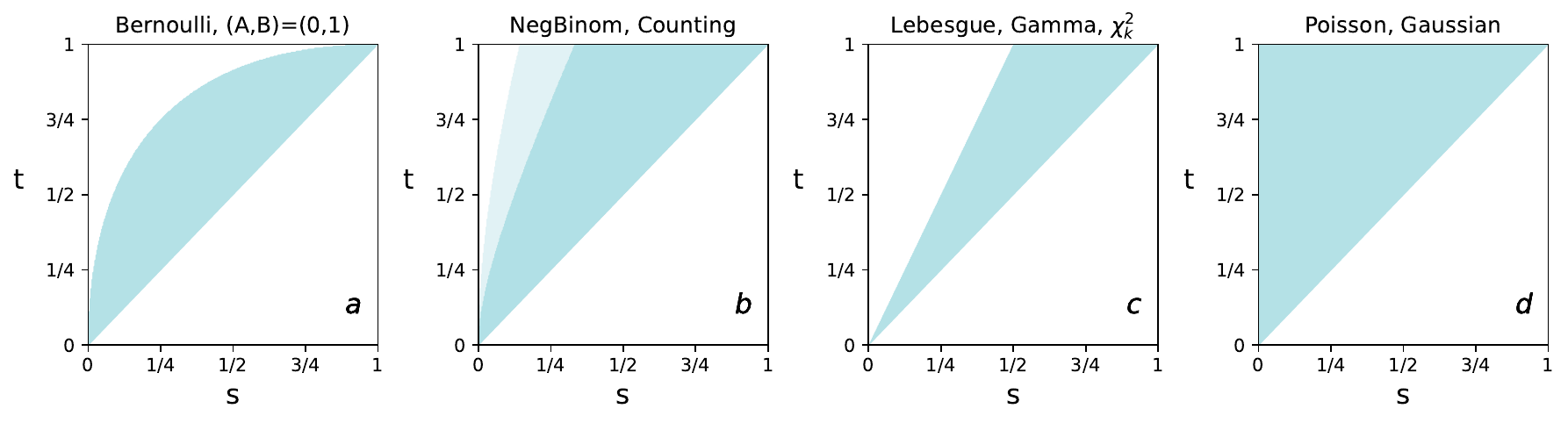}
		\end{center}
		\vspace{-0.4cm} 
		\caption{ Phase diagrams for tame margins for various base measures $\mu$. The upper contours are given by (\textbf{a}) $(s+t)^{2}<2s$, (\textbf{b}) $t\le 1+\sqrt{1+rs^{-1}}$ with $r=5$ for NegBinom and $r=1$ for Counting (so $\Omega(\mu)$ for the former measure is the union of the two shaded regions), (\textbf{c}) 
			$t\le s/2$, and (\textbf{d}) $t=\infty$. 
		}
		\label{fig:PhaseDiagram}
	\end{figure}

	First, when $\mu$ has bounded support (i.e., $|A|,|B|<\infty$), we show that  $\Omega(\mu)$ is determined by a simple quadratic inequality involving only $A,B$. Furthermore,  the shape of $\Omega(\mu)$ is universal among the class of all measures of bounded support (see Fig. \ref{fig:PhaseDiagram}\textbf{a}) up to translation and scaling. This universality result is stated below. 
	
	\begin{theorem}\label{thm:z_uniform_bd_tame}
		Suppose $-\infty<A\le B<\infty$. Then each  $(s,t)\in (A,B)^{2}$ with $s\le t$ belongs to $\Omega(\mu)$ if
		\begin{align}\label{eq:EG_condition_gen}
			(s+t-2A)^{2} < 4 (B-A)(s-A).
		\end{align} 
		Furthermore, $(s,t)$ does not belong to $\Omega(\mu)$ if the reverse inequality in \eqref{eq:EG_condition_gen} holds. 
	\end{theorem}

	When $\mu$ has unbounded support, the situation changes drastically and the depends of $\Omega(\mu)$ on $\mu$ is more sophisticated. We first investigate the phase diagram restricted to the particular class of $2\times 2$ block margins that we call the `Barvinok margins'. These are symmetric linear margins with two distinct values where a vanishing fraction of the row sums are assigned with the larger value (see \cite[Sec. 1.6]{barvinok2010does} and \cite{dittmer2020phase, lyu2022number}). In this case, the phase diagram is determined by the inequality $	2\phi(t)-\phi(s)<\phi(B)$ as stated below. 
	
	\begin{corollary}[of Prop. \ref{prop:z_tame_unbounded_gen}]\label{cor:barvinok_margin}
		Suppose $\mu$ is a such that $A=0$ and $B=\infty$. Consider a sequence of  symmetric $n\times n$ margins $(\r_{n},\r_{n})$ such that 
		\begin{align}\label{eq:def_barvinok_margin_0}
			\r_{n} :=  \bigl( \underbrace{t_{n }n  ,\ldots, t_{n} n}_{\lfloor n^{\rho} \rfloor} ,   \underbrace{s_{n}n  , \ldots,  s_{n}n }_{n-\lfloor n^{\rho} \rfloor}  \bigr)
			\, \in \, \R_{\ge 0}^{n}
		\end{align}
		where $s_{n}\rightarrow s$ and $t_{n}\rightarrow t$ as $n\rightarrow \infty$ for $0\le s\le t$ and $\rho\in [0,1)$ is fixed. If $2\phi(t)-\phi(s)<\phi(B)$, then for all $n\ge 1$, $\r_{n}$ is $\delta$-tame for some $\delta>0$.  If $2\phi(t)-\phi(s)>\phi(B)$, then the $(1,1)$ entry of the typical table for $(\r_{n},\r_{n})$ diverges to infinity as $n\rightarrow\infty$. 
	\end{corollary}

	The above result implies that for base measures $\mu$ with $(A,B)=(0,\infty)$, the interior of $\Omega(\mu)$ is contained in the region defined by $2\phi(t)-\phi(s)<\phi(B)$ and $s\le t$. 
	While there is a possibility that $\Omega(\mu)$ can be much smaller than what Cor. \ref{cor:barvinok_margin} suggests, below in 
	Theorem \ref{thm:tameness_non_compact}, we establish that these two regions in fact coincide whenever $\psi''$ is increasing and log-convex. Recall that $\psi''(\theta)$ is the variance of the tilted measure $\mu_{\theta}$. 
	Hence  $\mu$ must have right-infinite support ($B=\infty$) if $\psi''$ is increasing. Since the more exponential tilting the more mass on the larger values, it is natural to expect that $\psi''$ would be increasing. Log-convexity is an additional condition we crucially use to reduce the analysis for general symmetric margins to a Barvinok margins (see Lem. \ref{lem:reduction_barvinok}). These conditions hold for a wide range of measures with unbounded support such as Gaussian, Poisson, counting on $\Z_{\ge 0}$, Lebesgue on $\R_{\ge 0}$, and Gamma
	(see 
	Sec. \ref{sec:examples}).

	\begin{theorem}\label{thm:tameness_non_compact}
		Suppose $\psi''$ is increasing and log-convex on $\Theta$ (necessarily $B=\infty$). Then each  $(s,t)\in (A,B)^{2}$ with $s\le t$ belongs to $\Omega(\mu)$ if 
		\begin{align}\label{eq:convex_mean_tame_condition}
			\phi(A)< 3 \phi(s) - 2\phi(t) \quad \textup{and}\quad  2 \phi(t)-\phi(s) < \phi(B). 
		\end{align}
		Furthermore, for measures $\mu$ s.t. $\phi(A)=-\infty$, $(s,t)$ does not blong to $\Omega(\mu)$  if $2 \phi(t)-\phi(s) > \phi(B)$. 
	\end{theorem}
	
	The first inequality in \eqref{eq:convex_mean_tame_condition} is often vacuous (e.g., when $\phi(A)=-\infty$), and only the second one matters. For a wide range of measures, we can solve the second inequality explicitly and the phase diagram $\Omega(\mu)$ is determined by the ratio $t/s$ being strictly less than some critical threshold $\lambda_{c}$ as stated in Cor. \ref{cor:sharp_tame_CT} below.

	\begin{corollary}\label{cor:sharp_tame_CT}
		An arbitrary $m\times n$ margin $(\r,\c)$ satisfying \eqref{eq:r_c_tame_s_t} with $0< s< t$ 
		is $\delta$-tame for $\delta$ depending only on $\mu, s, t$ provided $ t/ s<\lambda_{c}$, where 
		\begin{align}\label{eq:critical_ratio}
			\quad 	\lambda_{c}:=
			\begin{cases}
				1+\sqrt{1+ r s^{-1}} & \textup{if $\mu=$ $r$-fold convolution of the counting measure on $\mathbb{Z}_{\ge 0}$ for $r\ge 1$}, \\
				2 & \textup{if $\mu$ has density $e^{-ax}x^{\gamma-1}$ on $\mathbb{R}_{\ge 0}$ w.r.t. Lebesgue measure for $a\ge 0$, $\gamma \ge 1$}  \\
				\infty & \textup{if $\mu=$ Poisson or Gaussian}.
			\end{cases}
		\end{align}
		Furthermore, for $\mu$ as above, if $t/s>\lambda_{c}$, then there exists a sequence of $n\times n$ margins satisfying the hypothesis but some entry in the typical table diverging to infinity as $n\rightarrow\infty$. 
	\end{corollary}
	
	The critical threshold $\lambda_{c}$ has not been identified before for any base measure, but there are some previously known bounds on the critical ratio $\lambda_{c}$ for the counting and the Lebesgue cases.
	
	For the counting base measure, 
	Barvinok, Luria, Samorodnitsky, and Yong \cite[Thm. 3.5]{barvinok2010approximation} showed that $\lambda_{c}\le \frac{1+\sqrt{5}}{2}\approx 1.618$, the golden ratio. We learned from Barvinok that this bound was later improved to $2$ by Luria \cite{luria08}. Our Corollary \ref{cor:sharp_tame_CT} identifies the critical threshold as  $\lambda_{c}=1+\sqrt{1+s^{-1}}$, which agrees with the sharp phase transition point for Barvinok margins obtained by Dittmer, Lyu, and Pak \cite{dittmer2020phase}. Barvinok and Hartigan \cite{barvinok2012asymptotic} obtained a Gaussian formula for the number of contingency tables that is asymptotically correct for large  $\delta$-tame margins.  Corollary \ref{cor:sharp_tame_CT} implies this asymptotic Gaussian formula holds for 
	margins satisfying \eqref{eq:r_c_tame_s_t} with $t/s<1+\sqrt{1+s^{-1}}$ and this result cannot be improved. 
	
	For the Lebesgue bese measure on $\R_{\ge 0}$, Barvinok and Rudelson \cite{barvinok2024quick} recently observed that $\lambda_{c}\ge 2$. Namely, they noted that same Barvinok margin $\r=\c=(\lambda n, n,\dots,n)$, now with the Lebesgue base measure (see Ex. \ref{ex:Lebesgue_base_measure} for the corresponding typical table),  is asymptotically $\delta$-tame for $\lambda<2$ and it is not for $\lambda>2$ (see Remark \ref{rmk:Barvionok_Lebesgue}). There was no previously known upper bound on $\lambda_{c}$, and our result above shows that $\lambda_{c}=2$.

	As we have an almost complete understanding of the phase diagram of tame margins, one may ask a finer question about characterizing a necessary and sufficient condition for a given margin to be $\delta$-tame for some $\delta$ independent of the margin. We provide such a result for symmetric margins and general base measures with compact support. The equivalent condition for $\delta$-tameness turns out to be the celebrated Erd\H{o}s-Gallai condition \cite{erdos1960graphen} with a quadratic gap.

	\begin{theorem}[$\delta$-tameness and  Erd\H{o}s-Gallai condition for symmetric margins]\label{thm:tame_EG}
		Suppose $\mu$ has compact support with $A=0$ and $B<\infty$. If an   $n\times n$ symmetric margin $(\r,\r)$ with an MLE $(\balpha,\balpha)$ is $\delta$-tame for some $\delta>0$, then  there exists constants $c_{1},c_{2}\in (0,B)$ and $c_{3}>0$ depending only on $\mu$ and $\delta$ such that  $c_{1} \le \r(i)/n \le c_{2}$ for all $i$ and 
		\begin{align}\label{eq:matrix_EG}
			\min_{I\subseteq  [n],\, |I|\ge c_{1}^{2}n} \left( B|I|^2 + \sum_{i\notin I} B|I|\land \r(i) -\sum_{i\in I}\r(i)  \right) - c_{3} |I|^{2}  \ge 0. 
		\end{align}
		Conversely, if the above condition holds for constants $c_{1},c_{2}\in (0,B)$ and $c_{3}>0$, then $(\r,\r)$ is $\delta$-tame for some $\delta=\delta(\mu,c_{1},c_{2},c_{3})>0$. 
	\end{theorem}      
	
	In 2011, Chatterjee, Diaconis, and Sly showed that the condition \eqref{eq:matrix_EG} implies uniform boundedness of the MLE for the $\bbeta$-model with degree sequence $\r$ \cite{chatterjee2011random}. They conjectured that this condition is equivalent to the $\delta$-tameness of Barvinok and Hartigan \cite{barvinok2013number} defined in terms of the typical table. The special case of our Thm. \ref{thm:tame_EG} for Bernoulli base measure in conjunction with the strong duality in Thm. \ref{thm:strong_duality_simple} establishes this conjecture. 
	See Sec. \ref{sec:background_tame_EG} for more discussion.

	\subsection{Generalized Sinkhorn algorithm for computing the MLEs}

	In most cases, there is no close-form expression for an MLE for a given margin $(\r,\c)$. Based on the classical alternating maximization principle applied to \eqref{eq:typical_Lagrangian}, we propose the following iterative algorithm for computing an MLE for a given margin $(\r,\c)$, which we call the \textit{generalized Sinkhorn} algorithm: 
	\begin{align}\label{eq:AM_typical_MLE2}
		\qquad 
		\begin{matrix}
			\textbf{Generalized}	\\
			\textbf{Sinkhorn}
		\end{matrix}
		\quad 
		\begin{cases}
			\textit{For $1\le j \le n$, }	\bbeta_{k}(j) \leftarrow \textit{unique $\beta\in \R$ s.t. }  \c(j)  = \sum_{i=1}^{m} \psi'(\balpha_{k-1}(i) + \beta),\\
			\textit{For $1\le i \le m$, }	\balpha_{k}(i) \leftarrow \textit{unique $\alpha\in \R$ s.t. }  \r(i) =   \sum_{j=1}^{n} \psi'(\alpha + \bbeta_{k}(j)).
		\end{cases}
	\end{align}
	This is simply the alternating maximiazation procedure to solve the MLE problem \ref{eq:typical_Lagrangian}. To explain the connection to the Sinkhorn algorithm in the Schr\"{o}dinger bridge and the optimal transport literature, we first note that the typical table $Z^{\r,\c}$ in \eqref{eq:typical_table_opt} can be viewed as the static Schr\"{o}dinger bridge between the row and column margins $\r$ and $\c$ with divergence $D(\mu_{\phi(\cdot)}\Vert \mu)$ and the uniform prior measure. Also, the MLE problem \eqref{eq:typical_Lagrangian} corresponds to its Kantorovich dual, and the algorithm \eqref{eq:AM_typical_MLE2} is exactly the Sinkhorn algorithm that computes the Schr\"{o}dinger potentials. Taking the Poisson base measure $\mu=\textup{Poisson}(1)$, our divergence $D(\mu_{\phi(\cdot)}\Vert \mu)$ reduces to the KL-divergence, the standard choice in the Schr\"{o}dinger bridge and the optimal transport literature \cite{franklin1989scaling, cuturi2013sinkhorn, leonard2013survey}. (See Sec. \ref{sec:schrodinger} and \ref{sec:Sinkhorn} for more discussions). 
	
	Even with the connection to the existing literature we mentioned above, convergence analysis of the generalized Sinkhorn \eqref{eq:AM_typical_MLE2} is limited beyond the KL-divergence case. This is because the updated coordinates $\bbeta_{k}(j)$ and $\balpha_{k}(i)$ in \eqref{eq:AM_typical_MLE2} in general do not have simple closed-forms as in the KL-divergnece case (see \eqref{eq:AM_typical_MLE3}). Despite this difficulty, we show that, under mild conditions, the procedure \eqref{eq:AM_typical_MLE2} for arbitrary base measure $\mu$ converges at a linear rate in the sense that the dual objective value gap decays exponentially fast. It also yields that the direct sum of the MLEs, $\balpha_{k}\oplus \bbeta_{k}$, converges exponentially fast to the uniquely determined matrix of maximum-likelihood tilt parameters $\balpha^{*}\oplus \bbeta^{*}$.

	\begin{theorem}[Linear convergence of the generalized Sinkhorn]\label{thm:Sinkhorn_conv}
		Let $(\balpha_{k},\bbeta_{k})$, $k\ge 0$ denote the iterates produced by the Sinkhorn algorithm \eqref{eq:AM_typical_MLE2} for some $m\times n$ margin $(\r,\c)$ such that $\T(\r,\c)\cap (A, B)^{m\times n}$ is non-empty.  Fix an MLE $(\balpha^{*}, \bbeta^{*})$ and denote $\Delta_{k}:=g^{\r,\c}(\balpha^{*},\bbeta^{*})-g^{\r,\c}(\balpha_{k},\bbeta_{k})$. For
		each $\eps>0$, let $\sigma_{1}(\eps)^{2}$ (resp., $\sigma_{2}(\eps)^{2}$) denote the infimum (resp., supremum) of $\psi''$ on $(\phi(A_{\eps}),\phi(B_{\eps}))$. 
		\begin{description}
			\item[(i)] (Asymptotic linear convergence) 
			Choose $\delta>0$ small enough so that $(\r,\c)$ is $\delta$-tame. Then there exists an integer $k_{0}=k_{0}(\mu,\r,\c)\ge 0$ such that the following holds with $\eps=\delta/2$:
			\begin{align}\label{eq:sinkhorn_linear_conv}
				\qquad \quad  \frac{\sigma_{1}(\eps)^{2}}{2} \| (\balpha^{*}\oplus \bbeta^{*}) - (\balpha_{k}\oplus \bbeta_{k}) \|_{F}^{2} \, \le\,  \Delta_{k} \,  \le \, \left(1 - \frac{\sigma_{1}(\eps)^{4}}{\sigma_{2}(\eps)^{4}} \right)^{k-k_{0}} \Delta_{k_{0}} \quad \textup{for all $k\ge k_{0}$}.
			\end{align}

			\item[(ii)] (Non-asymptotic linear convergence I) 
			Suppose  $\psi''$ is increasing and  $\balpha_{0}=\mathbf{0}$. Then \eqref{eq:sinkhorn_linear_conv} holds with $k_{0}=1$ and $\eps=\delta$.

			\item[(iii)] (Non-asymptotic linear convergence II) Suppose  there exists $\eps>0$ such that 
			\begin{align}\label{eq:def_L_sinkhorn}
				\phi(A_{\eps})+2\lVert \balpha_{0}-\balpha^{*}\|_{\infty} \le \balpha^{*}\oplus \bbeta^{*}  \le \phi(B_{\eps}) -2\lVert \balpha_{0}-\balpha^{*}\|_{\infty},
			\end{align}
			Then \eqref{eq:sinkhorn_linear_conv} holds with $k_{0}=1$.
		\end{description}

	\end{theorem}

	
	The asymptotic linear convergence in \textbf{(i)} holds the general base measure $\mu$, margin $(\r,\c)$, and initialization $\balpha_{0}$. If $\psi''$ is increasing (e.g., when $\mu$ is Gaussian, Poisson, Lebesgue measure on $\R_{\ge 0}$, Gamma, and counting measure on $\Z_{\ge 0}$; see Sec. \ref{sec:examples}), then the linear convergence holds for all iterates with zero initialization. If $\psi''$ is not necessarily increasing (e.g., $\mu=$ Binomial) but $\mu$ has thin tails so that $\Theta^{\circ}=\R$ (e.g., $\mu$ has compact support), then $\phi(B_{\eps})\rightarrow\infty$ and $\phi(A_{\eps})\rightarrow -\infty$ as $\eps\searrow 0$, so the hypothesis of \textbf{(iii)} always holds. Even when $\Theta^{\circ}$ is a proper subset of $\R$, \textbf{(iii)}  still holds if $\balpha_{0}$ is initialized within some $O(1)$-radius $L^{\infty}$-ball around $\balpha^{*}$. 

	\subsection{Applications of the general results}
	\label{sec:statement_applications}
	Now we apply our general framework above to establish results on specific aspects of $X$: (1) The (mixtures of) marginal distributions of the entries in $X$, (2) concentration of $X$  around the typical table $Z^{\r,\c}$ in the cut norm, and (3) the empirical singular value distribution of $X-Z^{\r,\c}$. All results are based on transferring properties of the comparison model $Y$ to the conditioned model $X$. Accordingly, throughout this section, we assume the transference hypotheses are satisfied -- Assumption \ref{assumption:weak_transference} for weak transference and Assumption \ref{assumption:strong_transference} for strong transference. For the latter case, we understand all statements to hold for $\nu$-almost all margins, even when not explicitly stated.

	\subsubsection{Marginal distributions}

	First, we show that a mixture of the entry-wise distributions for $X$ is very close to the corresponding mixture for $Y$. 
	We denote by $d_{TV}(\cdot,\cdot)$ the total variation distance between probability measures. 
	
	\begin{definition}[Mixture of entry-wise distribution]\label{def:mixture_dist}
		Let $X\sim \lambda_{\r,\c,\rho}$. Let $\xi_{ij}$ denote the law of $X_{ij}$ for each $i,j$. 
		Fix $I\subseteq [m]$, $J\subseteq [n]$ and define the  mixture distributions 
		\begin{align}
			\widetilde{\xi}_{I,J}:=\frac{1}{|I\times J|} \sum_{(i,j)\in I\times J} \xi_{ij} \quad \textup{ and} \quad \widetilde{\mu}_{\balpha(I)\oplus \bbeta(J)}:=\frac{1}{|I\times J|} \sum_{(i,j)\in I\times J} \mu_{\balpha(i)+\bbeta(j)}.
		\end{align} 
	\end{definition}
	
	\begin{theorem}[Mixture of entry-wise distribution]\label{thm:mixture_dist}
		Fix a $\delta$-tame margin $(\r,\c)$ with an MLE $(\balpha,\bbeta)$ and let $X\sim \lambda_{\r,\c,\rho}$ and $Y\sim \mu_{\balpha\oplus \bbeta}$.   
		Then there exists a constant $C=C(\mu,\delta)>0$ such that  each $\delta$-tame margin $(\r,\c)$, the following holds: If we denote \begin{align}\label{eq:F_r_c_def}
			\qquad F_{\mu}(\r,\c):= 	\begin{cases}
				C\rho & \textup{if $\rho\ge C\sqrt{mn(m+n)}$,} \\
				C (m\lor n)^{3/2}(\log m n) & \textup{if 
					the hypothesis of Thm. \ref{thm:second_transference_density_1} holds,} \\
				C (m+n)\log (m+ n)&\textup{if 
					the hypothesis of Thm. \ref{thm:transfer_counting_Leb} holds,}
			\end{cases}
		\end{align}
		then we have \begin{align}\label{eq:thm_mixture_distribution_2}
			d_{TV}\left( \widetilde{\xi}_{I,J}, \,  \widetilde{\mu}_{\balpha(I)\oplus \bbeta(J)}\right) \le \sqrt{\frac{F_{\mu}(\r,\c)}{|I\times J|}} + \exp(-F_{\mu}(\r,\c)). 
		\end{align} 
	\end{theorem}

	Note that if $\r(i)$ and $\c(j)$ are constant for $i\in I$ and $j\in J$, then $\xi_{ij}$ and $\mu_{\balpha(i)+\bbeta(j)}$ do not depend on $(i,j)$ over $I\times J$ so the above results bound the total variation distance between the marginal distributions of the entries in $X$ and $Y$ in the block $I\times J$.

	\begin{remark}\label{rmk:CT_supercritical}
		When $\mu$ is the counting measure on an integer interval and $\rho=0$, then Theorem \ref{thm:mixture_dist} holds for a sequence of non-tame margins as the constant $C$ in the third case in \eqref{eq:F_r_c_def} is an absolute constant (see Thm. \ref{thm:transfer_counting_Leb}). In particular, this result in conjunction with Prop. \ref{prop:z_tame_unbounded_gen} recovers \cite[Thm. 2.1]{dittmer2020phase} by Dittmer, Lyu, and Pak.
	\end{remark}

	The following result is a continuous analog of Thm. \ref{thm:mixture_dist}. 
	\begin{theorem}\label{thm:mixture_dist_exact_conti}
		Keep the same setting as in Theorems \ref{thm:main_convergence_conti} and \ref{thm:mixture_dist}. Then there exists a constant $C=C(\mu,\delta)>0$ such that the following hold for all $m,n\ge 1$:
		Fix each measurable sets $S,T\subseteq [0,1]$ of positive Lebesgue measures and let $(U,V)\sim \textup{Unif}(S\times T)$.  \begin{align}\label{eq:thm_mixture_distribution_conti1}
			d_{TV}\left(\E[\xi_{\lceil mU\rceil, \lceil n V \rceil}] , \, \E[\mu_{\balpha(U)+\bbeta(V)}]\right) \le C\sqrt{ \frac{F_{\mu}(\r,\c) }{\textup{Leb}(S\times T)}}  +  C \lVert (\r,\c) - (\bar{\r}_{m},\bar{\c}_{n}) \rVert_{1}^{1/4}. 
		\end{align}
	\end{theorem}

	In 2010, Barvinok asked whether one can obtain the asymptotic distribution of an entry in a uniformly random contingency table  (for $\mu=\textup{Counting}(\Z_{\ge 0})$) for a sequence of cloned margins (see Ex. \ref{ex:cloned_margin}). 
	Namely, he conjectured that the limiting marginal distribution is Bernoulli for binary ($b=1$) contingency tables \cite{barvinok2010number} and geometric for unbounded ($b=\infty$) contingency tables \cite{barvinok2010does} with mean given by the corresponding entry in the typical table. 
	The following corollary of Thm. \ref{thm:mixture_dist_exact_conti} establishes these conjectures. Similar results hold with greater generality for possibly non-uniformly weighted contingency tables and the limiting marginal distribution of an entry is an exponential tilt of the base measure $\mu$.

	\begin{corollary}\label{cor:clone}
		Keep the same setting in Thm. \ref{thm:mixture_dist_exact_conti} with $\rho=0$.
		Suppose $(\r,\c)$ is the $k$-cloning of some $m_{0}\times n_{0}$ margin $(\r_{0},\c_{0})$ with an MLE $(\balpha_{0},\bbeta_{0})$ (see Ex. \ref{ex:cloned_margin}). Then as $k\rightarrow\infty$, 
		\begin{align}
			d_{TV}(\xi_{11}, \mu_{\balpha_{0}(1)+\bbeta_{0}(1)})  = \begin{cases}
				O\left(  k^{-1/4}\log k \right)  
				& \textup{if the hypothesis of Cor.  \ref{thm:second_transference_density_1} holds}, \\
				O\left(  k^{-1/2}\sqrt{\log k} \right) & \textup{if the hypothesis of Thm. \ref{thm:transfer_counting_Leb} holds}.
			\end{cases}
		\end{align}
	\end{corollary}

	In 2010, Chatterjee, Diaconis, and Sly showed that when $X$ is an $n\times n$ uniformly random doubly stochastic matrix, 
	$n \xi_{11}$ is asymptotically the exponential distribution with mean 1 
	\cite{chatterjee2014properties}. Our Cor. \ref{cor:clone} with $\mu=\textup{Leb}(\R_{\ge 0})$ and  $\r=\c=n\mathbf{1}_{n}$ recovers their result and extends it to general cloned margins possibly with non-uniform weights.

	\subsubsection{Scaling limit in cut norm}
	\label{sec:convergence_cutnorm}
	
	Next, we show that $X$ concentrates around the typical table $Z^{\r,\c}$ in the cut-norm, which is commonly used in the theory of dense graph limits. (See Section \ref{sec:proof_typical_kernels} for more discussion on the cut norm. See \cite{borgs2008convergent, borgs2012convergent, lovasz2012large} for more background on the limit theory of dense graphs utilizing cut metric.) 
	
	\begin{definition}[Cut norm for kernels]\label{def:cut_norm}
		The \textit{cut-norm} of a kernel $W$ is defined as 
		\begin{align}\label{eq:def_cut_norm}
			\lVert W \rVert_{\square} := \sup_{S,T\subseteq [0,1]} \left|  \int_{S\times T} W(x,y)\, dx\,dy \right|.
		\end{align}
	\end{definition}

	\begin{theorem}[Scaling limit in cut norm]\label{thm:CT_limit_bdd_integer}
		Keep the same setting as in Theorems \ref{thm:main_convergence_conti} and \ref{thm:mixture_dist}. 
		Let $W^{\r,\c}$ be the limiting typical kernel in Theorem \ref{thm:main_convergence_conti}.
		Then there exists a constant $C=C(\mu,\delta)>0$ such that the following hold for all $m,n\ge 1$. With probability at least $1- 2\exp \left(  - F_{\mu}(\r,\c)  \right)$, 
		\begin{align}\label{eq:exact_conditioning_cor_tight}
			\lVert W_{X} - W^{\r,\c} \rVert_{\square} 
			\le   \sqrt{ \frac{F_{\mu}(\r,\c)}{mn} }  + C\sqrt{ \lVert (\r,\c) - (\bar{\r}_{m},\bar{\c}_{n}) \rVert_{1}}.
		\end{align}
		In particular, $W_{X}\rightarrow W^{\r,\c}$ in cut norm almost surely as $m,n \rightarrow\infty$. 
	\end{theorem}

	\subsubsection{Empirical Singular Value Distribution} 
	
	Our last application is in analyzing the fluctuation of $X$ around the typical table $Z^{\r_{m},\c_{n}}$. A natural way to do so is to characterize the limiting spectral measure of the `centered' random matrix $X-Z^{\r_{m},\c_{n}}$ as the margins converge in $L^{1}$ to a continuum margin $(\r,\c)$ as $m,n\rightarrow\infty$. The study of spectral measures of large random matrices is a central topic in the random matrix literature (e.g., see \cite{anderson2010introduction} and references therein). 
	However, there is a relatively scarce literature on limiting spectral measure of margin-conditioned random matrices, only for constant linear margins with  Lebesgue and counting base measures \cite{chatterjee2014properties, nguyen2014random, wu2023asymptotic}.  We aim to address this problem for the \textit{empirical singular value distribution} (ESD) for arbitrary margins and base measures under a mild assumption.

	We establish a general result on the limit of ESD of margin-conditioned random matrices by first establishing it for the maximum likelihood tilted model and then transferring it to the conditioned model. In Sec. \ref{sec:transference_main}, we have discussed the weak transference (Thm. \ref{thm:transference}) for approximate margin conditioning but for the exact margin conditioning, we need to impose some conditions on $\mu$ (e.g., Thm. \ref{thm:second_transference_density_1}). In addition, we need a strong concentration of the ESD of the tilted model so that the strictly super-exponential transference costs can be suppressed. Guionnet and Zeitouni \cite{guionnet2000concentration}  provide a sub-Gaussian concentration of ESD when the entries of the random matrix have either bounded support or satisfy a uniform 
	log-Sobolev inequality. Klochkov and Zhivotovskiy \cite[Lem. 1.4]{klochkov2020uniform} show that the entries being uniformly sub-Gaussian is enough to warrant such sub-Gaussian concentration of ESD. 
	We now state our main result on the limit of ESD of $X-Z^{\r_{m},\c_{n}}$ and  $Y-Z^{\r_{m},\c_{n}}$. Put succinctly, the limiting ESD  is completely determined by the limiting `variance profile' $\psi''(\balpha\oplus \bbeta)$ through the corresponding Dyson equation. 
	
	\begin{theorem}[Limit of ESD]\label{thm:ESD}
		Keep the same setting as in Theorems \ref{thm:main_convergence_conti} and let  $W^{\r,\c}$ and $(\balpha,\bbeta)$ be the limiting typical kernel and limiting MLE.  Suppose $m/n\rightarrow \kappa\in (0,1)$. 
		\begin{description}[itemsep=0.1cm]
			\item[(i)] 
			Let $\xi_{m,n}$ denote the empirical measure on the singular values of $\widetilde{Y}=\frac{1}{\sqrt{(m+n)s^{*}/2}}(Y-Z^{\r,\c})$ with 
			$Y\sim \mu_{\balpha_{m}\oplus \bbeta_{n}}$		and $s^{*}:=\sup \psi''(\balpha\oplus \bbeta)$. Then there exists a probability measure $\xi$ on $\R$ determined by $\psi''(\balpha\oplus \bbeta)$ with support contained in $[0, 2]$ such that $\xi_{m,n}\rightarrow \xi$ in probability as $m,n\rightarrow\infty$ in the weak topology. The Stieltjes transform of $\xi$ is determined by 
			\begin{align}\label{eq:Stieltjes_measure}
				\qquad 	\langle \tau_{\cdot}(z) \rangle	:= \int_{(0,1)} \tau_{x}(z)\, dx	=	\int_{\R}  \frac{1}{t-z} \, \xi(dt) \qquad \textup{for all $z\in \mathbb{H}:=\{ z\in \mathbb{C}\,|\, \textup{Im}(z)>0 \}$},
			\end{align}
			where $\tau:\mathbb{H}\rightarrow L^{\infty}(0,1)$ is the unique uniformly bounded solution to the Dyson equation 
			\begin{align}\label{eq:wishart_QVE}
				- \frac{1}{\tau_{\cdot}(z)} = z - S \frac{1}{1+S^{*} \tau_{\cdot}(z)} \qquad \textup{for all $z\in \mathbb{H}$}, 
			\end{align}
			where $S$ denotes the integral operator with kernel $\psi''(\balpha\oplus \bbeta)$ and $S^{*}$ is the adjoint operator of $S$.
			
			Furthermore, the measure $\xi$ is absolutely continuous w.r.t. the Lebesgue measure apart from a possible point mass at zero, i.e., there is a number $q^{*} \in [0, 1]$ and a locally Hölder-continuous function $q: (0,\infty) \rightarrow [0,\infty)$ such that $\xi(dx) = q^{*} \delta_{0}(dx) + q(x) \mathbf{1}(x>0) dx$.  
			
			\item[(ii)]  The results in \textup{\textbf{(i)}} hold for the ESD of $\tilde{X}=\frac{1}{\sqrt{(m+n)s^{*}/2}}(X-Z^{\r,\c})$ for $X\sim \lambda_{\r_{m},\c_{n};\rho}$ under the following  conditions: 
			\begin{description}[itemsep=0.1cm] 
				\item{(A1)} (Transference) Any one of the following holds:
				\begin{description}
					\item{(i)}   $\rho=o(mn)$ and  $\rho\ge C_{2}\sqrt{mn(m+n)}$ for $C_{2}$ in Thm. \ref{thm:transference};  
					\item{(ii)} $\rho=0$ and the hypothesis of Thm. \ref{thm:second_transference_density_1} is satisfied.
				\end{description}

				\item{(A2)} (Sub-Gaussian ESD concentration) $\mu$ is a sub-Gaussian probability measure. 
				
			\end{description}
		\end{description}

	\end{theorem}

	\begin{figure}[hbt]
		\begin{center}
			\includegraphics[width = 1 \textwidth]{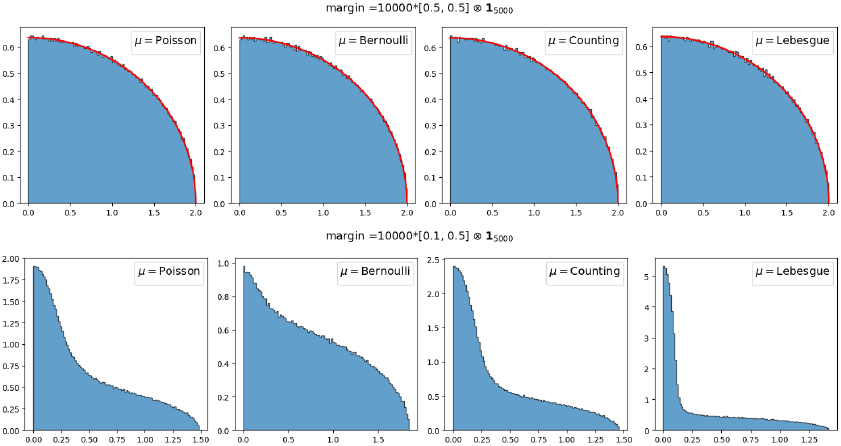}
		\end{center}
		\vspace{-0.4cm} 
		\caption{ Empirical singular value distribution of $(s^{*}n)^{-1/2}(Y-Z^{\r,\r})$ with various base measures $\mu$ and exact margin $(\r,\r)$ with $n^{-1}\r=(a, b)\otimes \mathbf{1}_{n/2}$ for $n=10^{4}$. We take $(a,b)=(0.5, 0.5)$ and $(0.1, 0.5)$. The limiting spectral distribution for the constant margin obeys the quarter-circle law (in red) universally for all $\mu$ satisfying the hypothesis of Thm. \ref{thm:ESD}. For the non-constant margin, the limit is not quarter-circle and depends on $\mu$. 
		}
		\label{fig:ESD}
	\end{figure}

	This result is a rather direct consequence of our main results and existing results in the random matrix theory literature. The ESD of $\widetilde{Y}$  is the square-root-transform of the empirical eigenvalue distribution of the inhomogeneous Wishart matrix $\widetilde{Y}\widetilde{Y}^{*}$, which is known to depend only on the variance matrix $\psi''(\balpha_{m}\oplus \bbeta_{n})$ of $\widetilde{Y}$ through the corresponding  Dyson equation \cite{alt2017local}. 	Our Thm. \ref{thm:main_convergence_conti} guarantees that the rescaled MLEs $(\bar{\balpha}_{m}, \bar{\bbeta}_{n})$ converge to the limiting MLE $(\balpha,\bbeta)$ in $L^{2}$ so a suitable continuity result on the solution of the Dyson equation (e.g., \cite{ajanki2019quadratic}) would suffice to deduce that the limiting ESD of $\widetilde{Y}$ is given by the Dyson equation with the limiting variance profile $\psi''(\balpha\oplus \bbeta)$. Then, the sub-Gaussian concentration of ESD of Wishart matrices (\cite{guionnet2000concentration} and \cite[Lem. 1.4]{klochkov2020uniform}) and our transference principles allow us to conclude the $\tilde{X}$ and $\widetilde{Y}$ have the same limiting ESD under the conditions stated in Thm. \ref{thm:ESD} \textbf{(ii)}.

	There are some important special cases where we can identify the limiting ESD explicitly. If we consider symmetric  constant linear margin, i.e., $\r_{n}=\c_{n}=an \mathbf{1}_{n}$ for some $a\in \R$, then by a symmetry argument the MLEs are given by $\balpha_{n}=\bbeta_{n}=(\phi(a)/2) \mathbf{1}_{n}$. In particular, the limiting variance profile $\psi''(\balpha\oplus \bbeta)$ is the matrix where all entries equal $\psi''(\phi(a))$. In this case, the solution of the Dyson equation \eqref{eq:wishart_QVE} is the Stieltjes transform of the celebrated Marchenko-Pastur quarter-circle law \cite{marchenko1967distribution}. This leads to the following corollary of Thm. \ref{thm:ESD} on the universality of the Marchenko-Pastur law for the ESD for constant linear margins. 	See Fig. \ref{fig:ESD} for illustration.

	\begin{corollary}[Universality of quarter-circle law for uniform margins]\label{cor:quater_circle}
		Assume uniform margins $\r_{n}=\c_{n}=a n \mathbf{1}_{n}$ for some $a\in (A,B)$ and $\rho=0$. 
		Suppose $\mu$ is a sub-Gaussian probability measure satisfying the hypothesis of Thm. \ref{thm:second_transference_density_1}. 
		Then the empirical singular value distribution of  $\widetilde{X}_{n} := (\psi''(\phi(a)) n)^{-1/2} (X- a\mathbf{1}_{n}\mathbf{1}_{n}^{\top} )$ for $X\sim \lambda_{\r_{n},\c_{n}}$ 
		converges weakly to the Marchenko-Pastur quarter-circle law $\frac{1}{\pi}\sqrt{4-x^{2}} \mathbf{1}(x>0) \,dx$ in probability. 
	\end{corollary}

	Another interesting corollary of Thm. \ref{thm:ESD} is that the limiting ESD of a margin-conditioned standard Gaussian matrix is universally Marchenko-Pastur regardless of the margin. This is essentially due to the fact that the exponential tilting of a Gaussian distribution is just a translation so that the variance profile is not affected by margin conditioning. 
	The following corollary states this in terms of the eigenvalues of the Gram matrix.

	\begin{corollary}[Universality of MP law margin-conditioned standard Gaussian matrices]\label{cor:Gaussian} 
		Assume $\mu$ is standard Gaussian and let $(\r_{n},\c_{m})$ be a sequence of $m\times n$ margins converging to a continuum margin $(\r,\c)$ in $L^{1}$ with $m/n\rightarrow \kappa\in (0,\infty)$. Let $X \sim \lambda_{\r_{m},\c_{n}}$ and let $Z:=Z^{\r_{m},\c_{n}}$ be as in \eqref{eq:Z_r_c_Gaussian}. 
		Then the empirical eigenvalue distribution of  $\frac{1}{n} (X- Z)(X- Z)^{\top}$  
		converges weakly in probability to the Marchenko-Pastur distribution $(1-\frac{1}{\kappa })\mathbf{1}(\kappa> 1)\delta_{0}(dx)+\frac{\sqrt{(\lambda_{+}-x)(\lambda_{-}+x)}}{2\pi x} \mathbf{1}(\lambda_{-}\le x \le \lambda_{+}) dx$  with $\lambda_{\pm}=(1\pm \sqrt{\kappa})^{2}$. 
	\end{corollary}
	
	Note that in the square case $\kappa=1$, the limiting ESD of $n^{-1/2}(X-Z)$  in Cor. \ref{cor:Gaussian} is the quarter-circle law stated in Cor. \ref{cor:quater_circle}.

	\subsection{Organization of the paper}

	The rest of this paper is organized as follows. We provide background on the related literature and discussion on our results in Section \ref{sec:background}. In Section \ref{sec:examples}, we discuss several examples by specializing the base measure $\mu$. In  Section \ref{sec:strong_duality_pf}, we prove the strong duality result stated in Theorem \ref{thm:strong_duality_simple}. In the following section, Section \ref{sec:transference_pf}, we prove the transference principles stated in Section \ref{sec:transference_main}. Theorem  \ref{thm:main_convergence_conti} on scaling limit of the typical tables and the MLEs is proved in Section \ref{sec:pf_scaling_limit}. The results on the phase diagram of tame margins stated in Section \ref{sec:suff_cond_tame} are proved in Section \ref{sec:tameness_pf}. In Section \ref{sec:sinkhorn}, we establish Theorem \ref{thm:Sinkhorn_conv} on the linear convergence of the generalized Sinkhorn algorithm \eqref{eq:AM_typical_MLE2} for computing the MLEs. 
	In Section \ref{sec:proof_applications}, we prove the results on the marginal distribution, scaling limit in the cut norm, and empirical singular value distribution stated in Section \ref{sec:statement_applications}. Lastly in Section \ref{sec:concluding_remarks}, we provide concluding remarks and discuss some open problems. 
	
	\section{Background, Discussions, and Conjectures}
	\label{sec:background}
	
	Our study of large random matrices with i.i.d. entries conditioned on margins has a rich connection to the static Schr\"{o}dinger bridge, matrix scaling, relative entropy minimization,  enumeration of contingency tables, random graphs with given degree sequences, and spectral distribution of inhomogeneous Wishart matrices.

	\subsection{Typical table, static Schr\"{o}dinger bridge, and matrix scaling}\label{sec:schrodinger}
	
	The central optimization problem in this work is the typical table problem described in \eqref{eq:typical_table_opt}, which can be viewed as a special instance of the discrete $f$-divergence static Schrödinger bridge. 
	For a fixed convex function $f$, the \textit{$f$-divergence} of a probability measure $\mathcal{H}$ from a measure $\mathcal{R}$ on the same probability space is
	\begin{align}\label{eq:f_divergence}
		D_{f}(\mathcal{H}\, \Vert \, \mathcal{R}) :=  \int f\left( \frac{d\mathcal{H}}{d\mathcal{R}} \right) \, d\mathcal{R} \qquad \textup{if $\mathcal{H}\ll \mathcal{R}$ and $\infty$ otherwise}, 
	\end{align}
	where $\frac{d\mathcal{H}}{d\mathcal{R}}$ denotes the Radon-Nikodym derivative of $\mathcal{H}$ w.r.t. $\mathcal{R}$, ``$\ll$'' denotes absolute continuity between measures.  Suppose we have a joint measure  $\mathcal{R}$ representing our prior knowledge on the joint behavior of two random variables, say $X_{1}$ and $X_{2}$. 
	Consider finding the joint probability distribution $\mathcal{H}$ minimizing the `$f$-divergence' from $\mathcal{R}$ with marginal distributions $\mu_{1}$ and $\mu_{2}$: 
	\begin{align}\label{eq:RM_min_SB}
		&\min_{\mathcal{H}} \,\,  D_{f}(\mathcal{H}\, \Vert \, \mathcal{R}) \qquad \textup{subject to} \quad  \mathcal{H}\in \Pi(\mu_{1},\mu_{2}), 
	\end{align}
	where $\Pi(\mu_{1},\mu_{2})$ denotes the set of all joint probability measures that have marginal densities $\mu_{1}$ and $\mu_{2}$. 
	The solution of the above problem is called the \textit{$f$-divergence static Schr\"{o}dinger bridge} (SSB) between densities $\mu_{1}$ and $\mu_{2}$ with respect to $\mathcal{R}$. 
	The $f$-divergence $D_{f}$ becomes the KL-divergence $D_{KL}$ when $f(x)=x\log x$. In this case, \eqref{eq:RM_min_SB} becomes the classical SSB (see, e.g., \cite{fortet1940resolution, pavon2021data}). SSBs with general $f$-divergence have been studied more recently 
	\cite{carlier2017convergence, lorenz2022orlicz, terjek2022optimal}.

	Now consider the following discrete setting where the reference measure $\mathcal{R}$ lives on the integer lattice $[m]\times [n]$ and the marginal distributions $\mu_{1}$ and $\mu_{2}$ live on $[m]$ and $[n]$, respectively. 	Writing \eqref{eq:RM_min_SB} as an optimization problem involving the Radon-Nikodym derivative $X =(x_{ij})= \frac{d\mathcal{H}}{d\mathcal{R}}\in \R^{m\times n}$, 
	\begin{align}\label{eq:schrodinger_discrete}
		\qquad \min_{X \in \T(\r,\c) }  \sum_{i,j} f(x_{ij}) \, \mathcal{R}(i,j)\quad \textup{s.t.} \quad \sum_{j=1}^{n} x_{ij} \mathcal{R}(i,j) = \mu_{1}(i),\quad   \sum_{i=1}^{m} x_{ij} \mathcal{R}(i,j)= \mu_{2}(j) \,\, \forall i,j.
	\end{align}
	In particular, the typical table $Z^{\r,\c}$ in \eqref{eq:typical_kernel_opt} is the SSB between the row and column margins w.r.t. the uniform prior measure $\mathcal{R}$ and divergence $f(\cdot)=D(\mu_{\phi(\cdot)}\Vert \mu)$. Also, choosing $\mu$ to be the Poisson base measure  reduces \eqref{eq:schrodinger_discrete} to the KL-divergence SSB problem between the row and column margin w.r.t. $\mathcal{R}$ (see Ex. \ref{ex:Poisson_FY} for more details). In a similar manner, if one takes $\mathcal{R}$ to be the Lebesgue measure on $[0,1]^{2}$ in   \eqref{eq:RM_min_SB} and parameterize the measure $\mathcal{H}$ by $ \frac{d\mathcal{H}}{d\mathcal{R}}$ as before, then it reduces to a continuum version of the typical table problem that we define in Section \ref{sec:proof_typical_kernels} (see \eqref{eq:typical_kernel_opt}).

	Another problem that is closely related to our typical table problem is the classical matrix scaling \cite{sinkhorn1964relationship}. Given a  matrix $R$, can one find diagonal matrices $D_{1}$ and $D_{2}$ such that $S=D_{1}RD_{2}$ has margin $(\r,\c)$? It is well-known that such a rescaled matrix $S$ minimizes the KL-divergence $D_{KL}(\cdot\Vert R)$ among all matrices in $\T(\r,\c)$. 
	It is not hard to see that the entrywise `relative density' $X=S/R$ solves  \eqref{eq:schrodinger_discrete} with $\mu$ Poisson and $\mathcal{R}=R$ (see \cite{idel2016review} for a recent review on matrix scaling).

	There is a striking resemblance of the duality between the typical table and the MLE in Thm. \ref{thm:strong_duality_simple} and the duality between KL-divergence SSBs and the Schr\"{o}dinger potentials (see, e.g., \cite{leonard2013survey, nutz2022entropic}, \cite[Thm. 2.1, 3.2]{nutz2021introduction}, and  \cite[Sec. 3 (B)]{csiszar1975divergence}). Namely, 
	for the KL-divergence and 
	under some condition on the margin $(\mu_{1},\mu_{2})$, there exists functions $\balpha_{1},\balpha_{2}: \R\rightarrow \R$ known as the \textit{Schr\"{o}dinger potentials} that characterize the SSB as 
	\begin{align}\label{eq:SB_decomposition}
		\frac{d \mathcal{H}}{d \mathcal{R}}(x,y) = \exp(\balpha_{1}(x)+\balpha_{2}(y)) \qquad \textup{$\mathcal{R}$-a.s.}
	\end{align}
	The above decomposition may be compared to the characterization of the typical table $Z^{\r,\c}=\psi'(\balpha\oplus \bbeta)$ in Thm. \ref{thm:strong_duality_simple}. Also, the MLE equations \eqref{eq:MLE_eq} correspond to the Schr\"{o}dinger equations determining the  Schr\"{o}dinger potentials and the characterization \eqref{eq:lem_MLE_typical_duality1} of the typical table corresponds to the decomposition of Schr\"{o}dinger bridge using Schr\"{o}dinger potentials above. The shift-invariance of the MLEs in \eqref{eq:params_non_identifiable} corresponds to that for the Schr\"{o}dinger potentials (see, e.g., \cite[Lem. 2.11]{nutz2021introduction}).

	\subsection{Sinkhorn algorithm and entropic optimal transport}
	\label{sec:Sinkhorn}

	Notice that the log-likelihood function $g^{\r,\c}(\balpha,\bbeta)$ in \eqref{eq:typical_Lagrangian} is concave since $\psi$ is strictly convex. It is not however strictly concave since its value is invariant under adding a constant to $\balpha$ and subtracting the same constant to $\bbeta$. A natural algorithm for maximizing a concave function in two blocks of variables is alternating maximization, which is also known as block nonlinear Gauss-Seidel  \cite{bertsekas1997nonlinear, beck2013convergence}:
	\begin{align}\label{eq:AM_typical_MLE}
		\begin{cases}
			\bbeta_{k} \leftarrow \argmax_{\bbeta \in \R^{n}} \, g^{\r,\c}(\balpha_{k-1}, \bbeta),\\
			\balpha_{k} \leftarrow \argmax_{\balpha \in \R^{m}} \, g^{\r,\c}(\balpha, \bbeta_{k}).
		\end{cases}
	\end{align}
	The log-likelihood function $g^{\r,\c}(\cdot,\cdot)$ is strictly concave in each block coordinate, so each block maximization problem in \eqref{eq:AM_typical_MLE} reduces to finding the unique stationary points. 
	Hence 	\eqref{eq:AM_typical_MLE} is equivalent to the generalized Sinkhorn algorithm \eqref{eq:AM_typical_MLE2} we stated in Section \ref{sec:Introduction}.

	A notable special case of our algorithm \eqref{eq:AM_typical_MLE2} is the celebrated Sinkhorn's algorithm for matrix scaling and computing static Schr\"{o}dinger bridges \cite{franklin1989scaling, cuturi2013sinkhorn}. 
	In Sec. \ref{sec:schrodinger}, we observed that the divergence $D(\mu_{\phi(\cdot)}\Vert \mu)$ reduces to the KL-divergence when $\mu=\textup{Poisson}(1)$. In fact, in this case $\psi'(x)=e^{x}$ so we have $\psi'(\alpha+\beta)=\psi'(\alpha)\psi'(\beta)$. This separability simplifies the algorithm \eqref{eq:AM_typical_MLE2} as 
	\begin{align}\label{eq:AM_typical_MLE3}
		\begin{cases}
			\textup{For $1\le i \le n$, }		\bbeta_{k}(j) \leftarrow \log\left( \c(j)\bigg/ \sum_{i=1}^{m} \exp(\balpha_{k-1}(i))  \right),\\
			\textup{For $1\le i \le m$, }	\balpha_{k}(i) \leftarrow \log\left( \r(i)\bigg/ \sum_{j=1}^{n} \exp(\bbeta_{k}(j))  \right).
		\end{cases}
	\end{align}
	This is Sinkhorn's algorithm for matrix scaling with input matrix $R=\mathbf{1}\mathbf{1}^{\top}$ and target margin $(\r,\c)$ and also it is Sinkhorn's algorithm for the KL-divergence Schr\"{o}dinger bridge with uniform prior (see \eqref{eq:schrodinger_discrete}). Due to the special choice of the KL-divergence and the uniform prior,  \eqref{eq:AM_typical_MLE3}  converges to the expectation of the Fisher-Yates distribution for margin $(\r,\c)$, which is also the typical table for that margin with Poisson base measure (see Ex. \ref{ex:Poisson_FY}). However, for general base measure $\mu$, the limit does not admit a closed-form expression in terms of the margin $(\r,\c)$.

	Sinkhorn's algorithm has been extensively studied for a particular instance of the KL-divergence  Schr\"{o}dinger bridge with prior $\mathcal{R}$ taken to be the Gibbs kernel $e^{-c(x,y)/\eps}\mu_{1}(dx)\otimes d\mu_{2}(dy)$, where $c$ is a cost function and $\eps>0$ is the entropic regularization parameter. 
	In this case, it reduces to the entropic optimal transport problem \cite{villani2021topics}
	\begin{align}
		\min_{\mathcal{H} \in \Pi(\mu_{1},\mu_{2}) }   \, \int  c(x,y)  \, \mathcal{H}(dx, dy) + \eps D_{KL}(\mathcal{H}\, \Vert \, \mu_{1}\otimes \mu_{2}).
	\end{align}
	The classical case with quadratic cost is when we set  $c(x,y)=(x-y)^{2}$. 
	The corresponding Sinkhorn algorithm in the discrete case is given by \eqref{eq:AM_typical_MLE3} with both $\exp(\balpha_{k-1}(i))$ and $\exp(\bbeta_{k}(j))$ multiplied by $\mathcal{R}(i,j)$. Franklin and Lorenz showed that the convergence rate of the Sinkhorn in this case is exponential (linear in the log scale) in the space of margins endowed with Hilbert's projective metric \cite{franklin1989scaling}. A similar result was obtained for the continuous setting by Chen and Pavon \cite{chen2016entropic}. R\"{u}schendorf \cite{ruschendorf1995convergence}  established asymptotic convergence of the Sinkhorn algorithm in the continuous case from the perspective of information projection. Marino and Gerolin \cite{marino2020optimal} extended a similar result to the multi-marginal case. Carlier \cite{carlier2022linear} established linear convergence of the multi-marginal Sinkhorn algorithm in the Euclidean metric. 
	
	All of the convergence results mentioned above concern the KL-divergence (i.e., Poisson base measure) and the analysis heavily relies on the explicit form of the Sinkhorn iterates, which is not available for general divergences. There are some recent results on asymptotic convergence of the Sinkhorn algorithm for entropic optimal transport with general divergences (e.g., \cite{terjek2022optimal,lorenz2022orlicz}). 
	Our Theorem \ref{thm:Sinkhorn_conv} establishes non-asymptotic linear convergence of the generalized Sinkhorn algorithm \eqref{eq:AM_typical_MLE2} for general base measures.

	\subsection{Relative entropy minimization and typical tables} 
	\label{sec:min_rel_entropy_discussion}
	
	The definition of the typical table (Def. \ref{def:typical_table}) arises naturally from the classical `least-action principle'. For our context, it roughly says that the conditioned random matrix in \eqref{eq:main_question} should have the `least-action distribution' that minimizes the relative entropy from the base measure constrained on the expected margins being $(\r,\c)$. More precisely, consider the following relative entropy minimization  problem: 
	\begin{align}\label{eq:RM_min_KL}
		\min_{\mathcal{H}\in \mathcal{P}^{m\times n} }  D_{KL}(\mathcal{H} \,\Vert \, \mathcal{R}) \quad \textup{subject to} \quad \textup{ $\E_{X\sim \mathcal{H} }[ (r(X), c(X))] = (\r,\c)$},
	\end{align}
	where $\mathcal{P}^{m\times n}$ denotes the set of all probability measures  on $\R^{m\times n}$. This is a matrix version of the standard relative entropy minimization with first moment constraints, which is also an instance of the information projection \cite{csiszar1975divergence}. 
	
	Interestingly, it turns out that the relative entropy minimization in \eqref{eq:RM_min_KL} is equivalent to our typical table problem \eqref{eq:typical_table_opt} when $\mathcal{R}=\mu^{\otimes (m\times n)}$ is a probability measure. Namely, the optimal probability measure $\mathcal{H}$ solving \eqref{eq:RM_min_KL} is the following product measure 
	\begin{align}\label{eq:rel_entropy_min_Z}
		\mathcal{H} = \bigotimes_{i,j} \mu_{\phi(z_{ij})},
	\end{align}
	where $Z^{\r,\c}=(z_{ij})$ is the typical table for margin $(\r,\c)$. 
	To see this, notice that 
	\eqref{eq:RM_min_KL} 
	can be thought of as an optimization problem involving the Radon-Nikodym derivative $h = \frac{d \mathcal{H}(\x) }{d \mathcal{R}(\x)}: \R^{m\times n}\rightarrow [0,\infty)$. 
	By setting the functional derivative of the corresponding Lagrangian 
	equal to zero, we find 
	\begin{align}\label{eq:minimum_rel_entropy_tilting}
		& \log h(\x)  \propto  \sum_{i,j}  (\alpha_{i}  + \beta_{j})  x_{ij} \qquad \textup{$\mathcal{R}$-a.s.} 
	\end{align}
	where $\alpha_{i}$ and $\beta_{j}$ are the Lagrange multipliers corresponding to the constraints on the expectation of the $i$th row and the $j$th column sums. 
	Thus, the optimal probability measure $\mathcal{H}$ solving \eqref{eq:RM_min_KL}  is a joint exponential tilting of the base distribution $\mathcal{R}$. When $\mu$ is a probability measure on $\R$ and $\mathcal{R}=\mu^{\otimes(m\times n)}$,
	then the joint tilting of $\mathcal{R}$ is the product of entrywise tilts of $\mu$. Thus we can write $\mathcal{H}=\bigotimes_{i,j}\mu_{\theta_{ij}}$, where $\theta_{ij}$ is the unknown tilting parameter for each entry $(i,j)$. Hence reparameterizing  exponential tilting by the mean after the tilt, we deduce \eqref{eq:rel_entropy_min_Z} via 
	\begin{align}
		\eqref{eq:RM_min_KL} \quad \Longleftrightarrow \quad \min_{X=(x_{ij})\in \T(\r,\c)} D_{KL}\left(  \bigotimes_{i,j} \mu_{\phi(x_{ij})} \, \bigg\Vert \,  \mu^{\otimes (m\times n)}  \right) = \eqref{eq:typical_table_opt}.
	\end{align}

	Information projection arises naturally as the rate function in the theory of large deviations (e.g., Sanov's theorem
	\cite{sanov1961probability}). Also, 
	Csisz\'{a}r's conditional limit theorem \cite[Thm. 1]{csiszar1984sanov} states that i.i.d. samples under the condition that the joint empirical distribution satisfying a convex constraint are asymptotically `quasi-independent' (see \cite[Def. 2.1]{csiszar1984sanov}) under the common distribution that solves the corresponding information projection problem. 
	
	Our main question \eqref{eq:main_question} considers a similar problem of conditioning the joint distribution of a collection of $mn$ i.i.d. random variables,  where we first arrange them to form an $m\times n$ random matrix and then condition on its row/column sums.  The variables after conditioning on the margin are not exchangeable unless $m=n$ and the row/column sums are constant. Hence, in the general case, one should not expect that the entries of the margin-constrained random matrix follow asymptotically some common law unlike in Csisz\'{a}r's conditional limit theorem. However, our transference principles (Sec. \ref{sec:transference_main}) bear some similarity in that the conditioned joint distribution of all entries is close to the random matrix ensemble with independent entries given by the information projection \eqref{eq:RM_min_KL}. Also, our dual formulation via maximum likelihood estimation is related to the Fenchel dual approach for information projection \cite{bhattacharya1995general}.

	\subsection{Contingency tables and phase transition} 
	
	Contingency tables are $m\times n$ matrices of nonnegative integer entries with prescribed margins with row sums $\r=(r_{1},\cdots,r_{m})$ and columns sums $\c=(c_{1},\cdots,c_{n})$, whereby $\textup{CT}(\r,\c)$ we denote the set of all such tables. They are fundamental objects in statistics for studying dependence structure between two or more variables and also correspond to bipartite multi-graphs with given degrees and play an important role in combinatorics and graph theory, see e.g.~see e.g.~\cite{barvinok2009asymptotic,diaconis1995rectangular,diaconis1998algebraic}. In the combinatorics literature, counting the number $|\textup{CT}(\r,\c)|$ of contingency tables has been extensively studied in the past decades \cite{barvinok2009asymptotic, barvinok2010approximation, canfield2007asymptotic, canfield2010asymptotic, lyu2022number}. In the statistics literature, sampling a random contingency table has been heavily investigated, mostly by using rejection-sampling type methods \cite{snijders1991enumeration, chen2005sequential, verhelst2008efficient, wang2020fast}  or Markov-chain Monte Carlo methods \cite{besag1989generalized, diaconis1995rectangular, wang2020fast}. These problems are closely related to each other and have many connections and applications to other fields \cite{branden2020lower} (e.g., testing hypothesis on the co-occurrence of species in ecology \cite{connor1979assembly}).

	A historic guiding principle to analyzing large contingency tables is the \textit{independent heuristic} (IH), which was introduced by I.~J.~Good as far back as in~1950 \cite{good1950probability}. The heuristic suggests that for a uniformly random table with the given total sum, the row sum and the column sum constraints should act asymptotically independently as the size of the table grows to infinity. In terms of the enumeration problem, IH yields a simple analytic formula for $|\textup{CT}(\r,\c)|$, say $|\textup{IH}(\r,\c)|$, which has been verified rigorously to be asymptotically correct when the margins are constant or have a bounded ratio close to one \cite{canfield2010asymptotic}. However, in 2009, Barvinok showed that for a sequence of cloned margins (see Ex. \ref{ex:cloned_margin}), the row and column constraints have asymptotic positive correlation and that $|\textup{CT}(\r,\c)|$ is exponentially larger than $|\textup{IH}(\r,\c)|$ 
	\cite{barvinok2009asymptotic}.

	In 2010, Barvinok introduced the concept of `typical tables' for binary \cite{barvinok2010number} and unbounded \cite{barvinok2010does} contingency tables, along with the method of maximum entropy, revolutionizing the analysis of large contingency tables. The key idea is that a uniformly random contingency table with margins \((\r,\c)\) can be approximated by the maximum entropy matrix, which maximizes a specific entropy function over the transportation polytope \(\T(\r,\c)\). Good's IH is closely related to this maximum entropy approach, except that the entropy being maximized pertains to an element of the transportation polytope, viewed as a probability distribution on the lattice $[m] \times [n]$ after rescaling \cite{good1963maximum}. Barvinok's framework treats the entries of random contingency tables as independent Bernoulli or geometric random variables for 0-1 and unbounded tables, respectively. This corresponds precisely to the optimization problem \eqref{eq:typical_table_opt} for our typical tables with Bernoulli and counting base measures.

	In \cite{barvinok2010does}, Barvinok observed that there is a drastic change in the typical tables for $n\times n$ symmetric margins $\r=\c=(\lambda n, n,\dots,n)$ between $\lambda=2$ and $\lambda=3$. Namely, as $n\rightarrow\infty$, the typical tables for $\lambda=2$ stay uniformly bounded ($\delta$-tame) but the $(1,1)$ entry blows up (non-tame) for $\lambda=3$. Based on this observation, he conjectured that there is a sharp phase transition in the behavior of the typical table somewhere between $\lambda=2$ and $3$ and that the uniformly random contingency table also exhibits a phase transition behavior. 
	This conjecture was established by Dittmer, Lyu, and Pak in \cite{dittmer2018contingency} except that the behavior of the random contingency table was established for Barvinok margins with a growing number of the larger value $\lambda n$. 
	They identified that the phase transition occurs at the critical ratio $1+\sqrt{2}$ and the asymptotic distribution of the uniformly random contingency table is geometric whose mean equals the corresponding entry in the typical table. Our Corollaries \ref{cor:barvinok_margin} and \ref{cor:sharp_tame_CT} and Theorem \ref{thm:mixture_dist} generalize these results (see also Rmk. \ref{rmk:CT_supercritical}).

	\subsection{Empirical spectral distribution of random matrices with given margin}

	In comparison to the vast literature on the empirical spectral distribution of various random matrix ensembles, there has not been much work for random matrices with i.i.d. entries conditioned on margins and existing works only consider constant margins. In 2010 (published in 2014)
	Chatterjee, Diaconis, and Sly  \cite{chatterjee2014properties} 
	established the limiting ESD of large $n\times n$ uniformly random doubly stochastic matrices (nonnegative entries with rows and columns summing to one) is the Marchenko-Pastur quarter-circle law. Adopting their approach with an additional combinatorial argument, Wu \cite{wu2023asymptotic} recently obtained the quarter-circle law for the empirical singular value distribution of the uniformly random contingency tables with constant linear margins. These works utilize an elementary form of the transference principle by comparing the corresponding random matrices 
	to the ones with i.i.d. entries from the exponential and the geometric distributions with mean 1.

	Our Cor. \ref{cor:quater_circle} establishes that the same quarter-circle law is universal for random matrices with sub-Gaussian entries conditioned to have constant linear margins. However, our result does not directly imply the two earlier results since the entry-wise distributions there after tilting (exponential and geometric, respectively) have only sub-exponential tails. The authors of the aforementioned works overcome super-exponential transference costs by first showing that the maximum of the margin-conditioned random matrices behaves as that of the corresponding comparison models with i.i.d. entries, so the support of the entries can be truncated to $[0,C \log n]$ with high probability. Then one can apply the sub-Gaussian concentration of ESD for bounded entries in \cite{guionnet2000concentration} based on Talagand's inequality (see, e.g., \cite[Thm. 6.6]{talagrand1996new}).  We conjecture that the behavior of the maximum entry should remain the same for general base measures and $\delta$-tame margins replacing constant linear margins: 
	\begin{conjecture}\label{con:X_max}
		Let $(\r,\c)$ be a $(m\times n)$ $\delta$-tame margin and let $X\sim \lambda_{\r,\c}$. Then $\max_{ij} X_{ij} \le C\log (m+n)$ with high probability for some constant $C=C(\mu,\delta)>0$. 
	\end{conjecture}

	The authors in \cite{chatterjee2014properties, wu2023asymptotic} showed Con. \ref{con:X_max} for constant linear margins for $\mu=\textup{Leb}(\R_{\ge 0})$ and $\textup{Counting}(\Z_{\ge 0})$ by appealing to an accurate volume estimate for the transportation polytope in Canfield and McCay  \cite{canfield2007asymptotic} and Barvinok \cite{barvinok2009asymptotic}, respectively. This approach does not seem to generalize well to the broader class of base measures $\mu$ and non-uniform $\delta$-tame margins. However, at least for the classical Lebesgue or counting base measure cases, we suspect that the above conjecture can be addressed by existing combinatorial techniques, especially when the margin has a small number of blocks.

	Our Theorem \ref{thm:ESD} characterizes the limiting ESD for random matrices with i.i.d. sug-Gaussian entries conditioned to have general $\delta$-tame margins. We conjecture that the sub-Gaussianity assumption in Theorem \ref{thm:ESD} \textbf{(ii)} is not necessary. 
	
	\begin{conjecture}\label{con:ESD_no_subGaussian}
		Theorem \ref{thm:ESD} \textup{\textbf{(ii)}} holds without  $\mu$ being sub-Gaussian. 
	\end{conjecture}
	
	\noindent One way to address this conjecture is to establish Con. \ref{con:X_max} and use a similar truncation argument as in \cite{chatterjee2014properties} with our Thm. \ref{thm:ESD} \textbf{(i)}.

	The authors of \cite{chatterjee2014properties} conjectured that for the uniformly random doubly stochastic matrices, the limiting empirical \textit{eigenvalue} distribution is the circular law. This conjecture was established by Ngyuen \cite{nguyen2014random}. 
	We conjecture that the same holds for constant linear margins universally for arbitrary base measure $\mu$ that satisfies the hypothesis of  Cor. \ref{cor:quater_circle}.
	
	\begin{conjecture}\label{cor:circular}
		Under the same setting as in Cor. \ref{cor:quater_circle}, 
		the empirical eigenvalue distribution of $\widetilde{X}_{n} = (2 \psi''(\phi(a)) n)^{-1/2} (X- a\mathbf{1}\mathbf{1}^{\top} )$ converges weakly to the circular law in probability. 
	\end{conjecture}

	\subsection{Random graphs with given degree sequence}

	Many scientific applications naturally prompt the exploration of graphs with specific topological constraints, such as a fixed number of edges, triangles, and so forth \cite{ying2009graph, del2010efficient, orsini2015quantifying}. As the degree sequence is one of the most fundamental observables in network science, the study of random graphs with a given degree sequence has been a major topic in the field in the last decades \cite{molloy1995critical, molloy1998size,  chung2002connected, chung2008quasi, barvinok2013number}.  
	
	There are two seminal works on this topic.  Diaconis, Chatterjee, and Sly \cite{chatterjee2011random} and Barvinok and Hartigan \cite{barvinok2013number}, taking statistical and combinatorial perspectives, respectively. Under a certain condition on the convergent sequence of degree sequences, the authors of  \cite{chatterjee2011random} showed that uniformly random graphs with a given degree sequence converge almost surely to a limiting graphon in the cut metric and also identified the limit. The limiting graphon turns out to be the limit of the expectation of an underlying maximum likelihood inhomogeneous statistical model for random graphs known as the `$\bbeta$-model'.  
	Similar results were obtained by Barvinok and Hartigan \cite{barvinok2013number} from the combinatorial perspective. Their key observation was that the adjacency matrix of a large random graph with a given degree sequence concentrates around a `maximum entropy matrix', which corresponds to the notion of typical tables for contingency tables \cite{barvinok2010does}. Our present work generalizes and leverages both of these statistical and combinatorial perspectives. In particular, our Thm. \ref{thm:strong_duality_simple} shows that these two approaches are in fact in a strong duality.
	
	\subsection{Tame margins and the Erd\H{o}s-Gallai (EG) condition}	
	\label{sec:background_tame_EG}
	
	For a sequence of nonnegative integers $d_{1} \le \dots \le d_{n}$, the necessary and sufficient condition for this sequence to represent the degree sequence of a simple graph with $n$ nodes is given by the Erd\H{o}s-Gallai (EG) condition \cite{erdos1960graphen}:
	\begin{align}\label{eq:erdos_galai}
		\min_{1 \le k \le n} \left( k(k-1) + \sum_{i=k+1}^{n} d_{i}\land k - \sum_{i=1}^{k} d_{i} \right) \ge 0.
	\end{align}
	Chatterjee, Diaconis, and Sly \cite{chatterjee2011random} assumed the limiting continuum degree distribution satisfies a continuous analog of the EG condition, ensuring the maximum likelihood estimate (MLE) of the underlying inhomogeneous random graph model remains uniformly bounded. Barvinok and Hartigan \cite{barvinok2013number} instead assumed the degree sequence is $\delta$-tame, 
	meaning typical tables stay uniformly away from boundary values 0 and 1. The same notion of $\delta$-tameness with the counting base measure is also crucial for obtaining asymptotic formulas for the number of contingency tables with unbounded nonnegative integer entries \cite{barvinok2010approximation}. These two hypotheses of boundedness of the MLEs and the typical tables are equivalent to our notion of $\delta$-tameness (Def. \ref{def:ab}) due to the strong duality relation $Z^{\r,\c}=\psi'(\balpha\oplus\bbeta)$ in Thm. \ref{thm:strong_duality_simple}.

	Our Theorems \ref{thm:z_uniform_bd_tame} and \ref{thm:tame_EG} are closely related to the EG condition \eqref{eq:erdos_galai} above. Note that if $(A,B)=(0,1)$, then the inequality \eqref{eq:EG_condition_gen} in Thm. \ref{thm:z_uniform_bd_tame} determining the phase diagram becomes $(s+t)^{2} < 4s$. Also, specializing the EG condition for $n$ sufficiently large for the linear degree sequence with two values $sn$ and $tn$, one obtains $( s+ t)^{2} \le 4s$. 
	In fact, Barvinok and Hartigan \cite[Thm. 2.1]{barvinok2013number} showed that the strict inequality here implies uniform tameness of a degree sequence, which corresponds to the special case of our Theorem \ref{thm:z_uniform_bd_tame} with $\mu=\textup{Uniform}(\{0,1\})$ restricted to symmetric margins.

	In 2011, Chatterjee, Diagonis, and Sly	\cite[Lem 4.1]{chatterjee2011random} showed that if a degree sequence is quadratically deep inside the EG-polytope in the sense of \eqref{eq:matrix_EG}, then the corresponding MLE for the $\bbeta$-model  is uniformly bounded. They furthermore conjectured that this condition is also equivalent to the $\delta$-tameness in Barvinok and Hartigan \cite{barvinok2013number}. Our Thm. \ref{thm:tame_EG} on symmetric margins (together with Thm. \ref{thm:strong_duality_simple}) establishes this conjecture not just for the Bernoulli base measure, but also for any base measures with bounded support. It would be interesting to obtain an analogous characterization of $\delta$-tameness for general asymmetric margins.

	\section{Examples}\label{sec:examples}
	
	In this section, we discuss various examples of our framework for analyzing contingency tables with general base measures. The first two examples with the Gaussian and the Poisson base measure are `exactly solvable' in that the typical table can be expressed as an explicit function of the margins. 
	
	\begin{example}[Gaussian base measure]\label{ex:typical_gaussian}
		Fix an $m\times n$ margin $(\r,\c)$. Suppose the base measure $\mu$ is the standard Gaussian distribution on $\R$, given by $\mu(dx)=\frac{e^{-x^2/2}}{\sqrt{2\pi}} dx$.  In this case, we have \begin{align}
			\Theta=\R, \qquad (A,B)=\R, \qquad \psi(\theta)=\frac{\theta^2}{2},\qquad \psi'(\theta)=\theta,\qquad \phi(x)=x.
		\end{align}
		The function $\psi''(\theta)=1$ is increasing and log-convex. 
		The titled measure $\mu_\theta$ is just a Gaussian distribution with mean $\theta$ and variance $1$. Also, the function $D(\mu_{\phi(x)}\Vert \mu)$ defined in \eqref{eq:typical_table_opt} is  
		\begin{align}
			D(\mu_{\phi(x)}\Vert \mu)=x\phi(x)-\psi(\phi(x))=\frac{x^2}{2}.
		\end{align}
		This is the  KL-divergence from the standard normal $\mathcal{N}(0,1)$ to shifted normal $\mathcal{N}(x,1)$.  
		The  typical table $Z^{\r,\c}=(z_{ij})$ that uniquely solves \eqref{eq:typical_table_opt} in this case is given by 
		\begin{align}\label{eq:Z_r_c_Gaussian}
			\hspace{4cm}	z_{ij}=\frac{\r(i)}{n}+\frac{\c(j)}{m}-\frac{N}{mn} \qquad \textup{for all $i\in [m]$ and $j\in [n]$}.
		\end{align}
		In particular, the margin $(\r,\c)$ is $\delta$-tame for a prespecified constant $\delta>0$ if and only if 
		\begin{align}
			\hspace{4cm}	-\delta^{-1} \le \frac{\r(i)}{n}+\frac{\c(j)}{m}-\frac{N}{mn} \le \delta^{-1} \qquad \textup{for all $i\in [m]$ and $j\in [n]$}.
		\end{align}
		Thus in this case the set of all $\delta$-tame margins is a convex polytope.

		Note that  $X\sim \lambda_{\r,\c}$ is distributed according to a degenerate Gaussian distribution on $\R^{m\times n}$ supported on the transportation polytope $\T(\r,\c)$ 
		satisfying  
		\begin{align*}
			\E[X \,|\, \r(X)=\r, \c(X)=\c] = Z^{\r,\c} , \quad 	\textup{Cov}(X_{ij},X_{k\ell}\, |\,  \r(X)=\r, \c(X)=\c)
			=\begin{cases}
				\frac{(m-1)(n-1)}{mn} & \text{ if }i=k, j=\ell,\\
				-\frac{n-1}{mn} & \text{ if }i\ne k, j=\ell,\\
				-\frac{m-1}{mn} & \text{ if }i=k, j\ne \ell,\\
				\frac{1}{mn} & \text{ if }i\ne j, k\ne \ell.
			\end{cases}
		\end{align*}
		One can see that the entries of $X$ after conditioning on the margin are asymptotically independent. Also, the correlation between two distinct entries is negative if they belong to the same row/column and positive otherwise. Finally, note that the maximum likelihood tilted model $Y$ in the Gaussian case has the law of $X+Z^{\r,\c}$, which has the same conditional law as $X$ given their margins as $(\r,\c)$. 
		\hfill $\blacktriangle$
	\end{example}
	
	\begin{example}[Poisson base measure]\label{ex:Poisson_FY}
		Suppose the base measure $\mu$ is the Poisson base measure on $\mathbb{N}\cup \{0\}$ given by $\mu(k)=\frac{1}{k!}$ for $k\in \Z_{\ge 0}$. In this case, we have   
		\begin{align}
			\Theta=\R, \qquad (A,B)=(0,\infty), \qquad \psi(\theta)=e^{\theta},\qquad \psi'(\theta)=e^{\theta}, \qquad \phi(x)=\log x.
		\end{align}
		Note that $\psi''(\theta)=e^{\theta}$ is increasing and log-convex.
		The function $D(\mu_{\phi(x)}\Vert \mu)$ defined in \eqref{eq:typical_table_opt} is  
		\begin{align}
			D(\mu_{\phi(x)}\Vert \mu)=x\, \phi(x)-\psi(\phi(x))=x\log x - x.	
		\end{align}
		Noting that for any matrix $X=(x_{ij})$ with margin $(\r,\c)$ the total sum $\sum_{i,j}x_{ij}$ is constant and equals to $N=\sum_{i=1}^{m} \r(i)=\sum_{j=1}^{n} \c(j)$, 
		the typical table problem in \eqref{eq:typical_table_opt} reduces to 
		\begin{align}\label{eq:poisson_shrodinger}
			\argmin_{X\in \mathcal{T}(\r,\c) \cap (0,\infty)^{m\times n}} \,  \sum_{i,j}\, x_{ij} \log x_{ij}.
		\end{align}
		Notice that \eqref{eq:poisson_shrodinger} is an instance of discrete Schr\"{o}dinger bridge problem \eqref{eq:schrodinger_discrete} with uniform reference measure $\mathcal{R}(i,j)\equiv 1/mn$.

		Since $\psi'(\theta)=e^{\theta}$, the typical table $Z^{\r,\c}=(z_{ij})$ that uniquely solves \eqref{eq:poisson_shrodinger} takes the following form 
		\begin{align}
			\hspace{3.5cm}	z_{ij} = e^{\balpha(i)+\bbeta(j)} \qquad \textup{for all $i\in [m]$ and $j\in [n]$}
		\end{align}
		for some dual variables (Schr\"{o}dinger potentials, see \eqref{eq:SB_decomposition}) $\balpha,\bbeta$. Invoking the margin condition $Z^{\r,\c}\in \T(\r,\c)$, we find that the typical table is precisely the Fisher-Yates table
		\begin{align}\label{eq:possion_typical_table}
			\hspace{3.5cm}	z_{ij}=\frac{ \r(i) \c(j)}{N} \qquad \textup{for all $i\in [m]$ and $j\in [n]$}.
		\end{align}
		In \cite{good1963maximum}, Good referred the typical table above as the `independence table' and observed that it maximizes the entropy function $\sum_{i,j} \frac{x_{ij}}{N}\log \frac{N}{x_{ij}}$ subject to the margin constraint, which is equivalent to \eqref{eq:poisson_shrodinger}.

		The Poisson case is exactly solvable in the sense that the distribution of $X$ conditional on $X\in \T(\r,\c)$ is precisely given as the  hypergeometric (Fisher-Yates) distribution
		\begin{align}\label{eq:Poisson_hyper_geom}
			\P(X=(x_{ij})_{i,j} \,|\, X\in \T(\r,\c) ) = \frac{\prod_{i=1}^m \r(i)! \prod_{j=1}^n \c(j)!}{N!\prod_{i=1}^m \prod_{j=1}^n x_{ij}!}. 
		\end{align}
		This fact holds for a more general model of independence in the statistic literature \cite{diaconis1998algebraic}, where  $X$ has independent Poisson entries $X_{ij}$ with mean $p_{i}q_{j}$ for some $p_{i},q_{j}\ge 0$. 
		Together with \eqref{eq:possion_typical_table}, 
		it follows that the conditional expectation of $X$ is exactly the typical table: 
		\begin{align}
			\E[X \,|\, X\in\T(\r,\c)] = Z^{\r,\c} = \left(  \r(i) \c(j)/N \right)_{i,j}. 
		\end{align}

		In \cite[p.5]{barvinok2010does}, Barbinok wrote that ``\textit{It looks plausible that the independence table $Z^{\r,\c}$ is close with high probability to the random contingency table $X\in \T(\r,\c)$ if it was sampled from the Fisher-Yates distribution \eqref{eq:Poisson_hyper_geom} instead of the uniform distribution.}'' Our Theorem  \ref{thm:CT_limit_bdd_integer} specialized to $\mu=\textup{Poisson}(1)$  confirms this speculation. \hfill $\blacktriangle$
	\end{example}

	\begin{example}[Binomial base measure] 
		Suppose the base measure $\mu$ is the Binomial distribution with parameters $(B,\frac{1}{2})$, given by $\mu(i)={B\choose i} 2^{-B}$ for $i\in \{0,1,\cdots,B\}$. For any $\theta\in \R$, the measure $\mu_\theta$ is the Binomial distribution with parameters $\Big(B,\frac{e^\theta}{1+e^\theta}\Big)$. We also have 
		\begin{align}
			\Theta=\R, \qquad (A,B)=(0,B), \qquad \psi(\theta)=B\log \frac{1+e^\theta}{2},\qquad \psi'(\theta)=\frac{Be^\theta}{1+e^\theta}, \qquad \phi(x)=\log \frac{x}{B-x}.
		\end{align}
		In this case the function $D(\mu_{\phi(x)}\Vert \mu)$ defined in \eqref{eq:typical_table_opt} is  
		\begin{align}
			D(\mu_{\phi(x)}\Vert \mu)=x \phi(x)  -  \psi(\phi(x))  = x\log x+(B-x)\log (B-x)+B\log \frac{2}{B}.
		\end{align}
		The typical table $Z^{\r,\c}=(z_{ij})$ that uniquely solves \eqref{eq:typical_table_opt} is given by 
		\begin{align}
			z_{ij}=\frac{B}{e^{-\balpha(i)-\bbeta(j)}+1},\qquad i\in [m], j\in [n],
		\end{align}
		where $\balpha,\bbeta$ depend on $(\r,\c)$ implicitly so as to satisfy $Z^{\r,\c}\in \T(\r,\c)$.

		In particular, taking $B=1$ we get the Bernoulli base measure,
		which was studied using the maximum entropy principle in Barvinok and Hartigan \cite{barvinok2010maximum}. 
		Note that the function $f$ above (up to a constant) is the negative entropy of the Binomial distribution with mean $x$. Hence  The typical table can be interpreted to be the mean matrix of maximum entropy within the transportation polytope $\mathcal{T}(\r,\c)$.

		Let $X\sim \lambda_{\r,\c}$. Unlike the Gaussian and the Poisson case, with this information alone it is not clear how $X$ given the margin $(\r,\c)$ would look like. Our Thm. \ref{thm:second_transference_density_1} states that $X$ can be approximated by a random matrix $Y$ with independent Binomial entries with mean $\E[Y]=Z^{\r,\c}$. \hfill $\blacktriangle$
	\end{example}

	\begin{example}[Counting base measure]\label{ex:counting_base_measure}
		Suppose the base measure $\mu$ is counting measure  on $\mathbb{Z}_+:=\{0,1,2,\ldots,\}$, given by $\mu(i)=1$ for $i\in \mathbb{Z}_+$. In this case, for any $\theta\in \Theta=\Theta^\circ$, the measure $\mu_\theta$ is a Geometric distribution with parameter $p=1-e^\theta$. Also we have \begin{align}
			\Theta=(-\infty,0),\quad (A,B)=(0,\infty),\quad \psi(\theta)=-\log(1-e^\theta),\quad \psi'(\theta)=\frac{e^\theta}{1-e^\theta},\quad \phi(x)= -\log  (1+x^{-1}).
		\end{align}
		Note that $\psi''(\theta)=e^{\theta}(1-e^{\theta})^{2}$ is increasing and log-convex on $\Theta$. The function $D(\mu_{\phi(x)}\Vert \mu)$ defined in \eqref{eq:typical_table_opt} is  
		\begin{align}
			D(\mu_{\phi(x)}\Vert \mu) = x \phi(x)  -  \psi(\phi(x)) 
			&=   x\log  x - (1+x) \log\left( 1+x  \right)
		\end{align}
		for $x\in (0,\infty)$. The typical table $Z^{\r,\c}=(z_{ij})$ that uniquely solves \eqref{eq:typical_table_opt} with this $f$ is given by 
		\begin{align}
			z_{ij}=\frac{1}{e^{-\balpha(i)-\bbeta(j)}-1},\qquad  i\in [m], j\in [n],
		\end{align}
		where $\balpha,\bbeta$ depend on $(\r,\c)$ implicitly so as to satisfy $Z^{\r,\c}\in \T(\r,\c)$.
		Recall that $f(x)$ is defined to be the relative entropy from the base measure $\mu$ to the geometric distribution  $\mu_{\phi(x)}$  on $\{0,1,\dots \}$ with mean $x$. Since the base measure $\mu$ is the counting measure, it coincides with the negative entropy of $\mu_{\phi(x)}$. Hence our definition of the typical table in this case coincides with that of Barvinok and Hartigan \cite{barvinok2010maximum} formulated through the maximum entropy principle.

		Let $X\sim \lambda_{\r,\c}$.  Clearly, such $X$  is uniformly distributed over the set $\T(\r,\c)\cap \Z_{\ge 0}^{m\times n}$ of all nonnegative inter-valued contingency tables with margin $(\r,\c)$. 
		Our Thm. \ref{thm:second_transference_density_1} states that $X$ can be approximated by a random matrix $Y$ with independent geometric entries with mean $\E[Y]=Z^{\r,\c}$. Also, our Theorem \ref{thm:CT_limit_bdd_integer} shows that $X$ is very close to the deterministic matrix $Z^{\r,\c}$ in cut norm with high probability as long as the margin is $\delta$-tame for some $\delta>0$. This is warranted if the ratios between the extreme values in the row/column sums are strictly less than the critical threshold $\lambda_{c}=1+\sqrt{1+s^{-1}}$, where $s>0$ is a constant lower bounding the linearly rescaled row and column sums (see Cor. \ref{cor:sharp_tame_CT}). 
		\hfill $\blacktriangle$ 
	\end{example}
	
	\begin{example}[Negative binomial base measure]\label{ex:neg_binom}
		Suppose the base measure $\mu$ puts the mass $\mu(i)={r+i-1\choose i}$, for $i\in \mathbb{Z}_{\ge 0}$. Here $r$ is a positive integer, which, following standard negative binomial terminology, can be thought of as the number of heads/successes $r$, and $\mu(i)$ is the number of ways one can get $i$ tails/failures before getting $r$ heads/successes. We note that this $\mu$ is the $r$-fold convolution of the counting measure on $\mathbb{Z}_{\ge 0}$. In this case, we have   
		\begin{align}
			\Theta=(-\infty,0), \quad (A,B)=(0,\infty), \quad \psi(\theta)=-r\log(1-e^\theta),\quad \psi'(\theta)=\frac{re^\theta}{1-e^\theta},\quad \phi(x)= -\log \Big(1+\frac{r}{x}\Big).
		\end{align}
		Also $\psi''(\theta)=r e^{\theta}(1-e^{\theta})^{2}$ is increasing and log-convex on $\Theta$. 
		The titled measure $\mu_\theta$ is a negative binomial distribution with parameters $(r, p=1-e^\theta)$.
		The function $f(x)=D(\mu_{\phi(x)}\Vert \mu)$ defined in \eqref{eq:typical_table_opt} is just $f(x)=-(x+r)\log(x+r)+x\log x+r\log r$, which is the entropy of the negative binomial distribution with mean $x$. The the typical table $Z^{\r,\c}=(z_{ij})$ that uniquely solves \eqref{eq:typical_table_opt} with this $f$ is given by 
		\begin{align}
			z_{ij}=\frac{r}{e^{-\balpha(i)-\bbeta(j)}-1},\qquad i\in [m], j\in [n],
		\end{align}
		where $\balpha,\bbeta$ depend on $(\r,\c)$ implicitly so as to satisfy $Z^{\r,\c}\in \T(\r,\c)$. 
		
		Our Thm. \ref{thm:second_transference_density_1} states that $X\sim\lambda_{\r,\c}$ can be approximated by a random matrix $Y$ with negative binomial  entries with mean $\E[Y]=Z^{\r,\c}$.
		Our Theorem \ref{thm:CT_limit_bdd_integer} yields that $X$ is very close to the typical table $Z^{\r,\c}$ in cut norm with high probability as long as the margin is $\delta$-tame for some $\delta>0$. Note that the threshold for $\delta$-tameness now becomes $ \lambda_{c}=1+\sqrt{1+rs^{-1}}$ (see Cor. \ref{cor:sharp_tame_CT}). 
		\hfill $\blacktriangle$
	\end{example}

	\begin{example}[Lebesgue base measure on positive reals]
		\label{ex:Lebesgue_base_measure}
		
		Suppose the base measure $\mu$ is Lebesgue measure on $(0,\infty)$, given by $\mu(dx)=dx$.  In this case, for any $\theta<0$, the measure $\mu_\theta$ is the Exponential distribution with mean $-\frac{1}{\theta}$. We  also have   
		\begin{align}
			\Theta=(-\infty,0), \qquad (A,B)=(0,\infty), \qquad \psi(\theta)=-\log(-\theta),\qquad \psi'(\theta)=-\frac{1}{\theta},\qquad \phi(x)=-\frac{1}{x}.
		\end{align}
		Note that $\psi''(\theta)=\theta^{-2}$ is increasing and log-convex on $\Theta$. The function $f(x)=D(\mu_{\phi(x)}\Vert \mu)$  defined in \eqref{eq:typical_table_opt} is just $f(x)=-1-\log x$, which is the negative entropy of the exponential distribution with mean $x$. The the typical table $Z^{\r,\c}=(z_{ij})$ that uniquely solves \eqref{eq:typical_table_opt} with this $f$ is given by 
		\begin{align}
			z_{ij}=\frac{-1}{\balpha(i)+\bbeta(j)},\qquad i\in [m], j\in [n],
		\end{align}
		where $\balpha,\bbeta$ depend on $(\r,\c)$ implicitly so as to sayisfy $Z^{\r,\c}\in \T(\r,\c)$. The above matrix is also known as the \textit{analytic center} of $\T(\r,\c)$, which is closely related to the `interior point method' in optimization \cite{renegar1988polynomial}. 
		
		Our Thm. \ref{thm:second_transference_density_1} states that $X$ can be approximated by a random matrix $Y$ with independent exponential entries with mean $\E[Y]=Z^{\r,\c}$.
		Chatterjee, Diaconis, and Sly \cite{chatterjee2014properties} analyzed large uniformly random doubly stochastic matrices, which is the special case of the Lebesgue base measure with constant margin $\r=\c=n\mathbf{1}_{n}$ and exact margin condition $\rho=0$. In this case, the maximum likelihood tilted model is simply the $n\times n$ random matrix of i.i.d. exponential entries with mean one and the typical table is $\mathbf{1}_{n}\mathbf{1}_{n}^{\top}$. 
		\hfill $\blacktriangle$
	\end{example}

	\begin{example}[Gamma distribution]\label{ex:gamma}
		Suppose the base measure $\mu$ has Gamma density $\mu(dx)=e^{-ax} x^{\gamma-1}\mathbf{1}(x>0)\,dx$ with rate parameter $a\ge 0$ and shape parameter $\gamma>0$.  This measure entails several special cases. Taking $a=1/2$ and $\gamma=k/2$ for $k\ge 1$ an integer, $\mu$ becomes $\chi^{2}$-distribution with $k$ degrees of freedom. Also, taking $a=0$ and $\gamma=k$, $\mu$ becomes the $k$-fold convolution of the Lebesgue measure on $\R_{\ge 0}$. In particular, it becomes the Lebesgue measure on $\R_{\ge 0}$ when $a=0$ and $\gamma=1$. For any $\theta<a$, the tilted measure $\mu_\theta$ is the Gamma distribution with parameters $(a-\theta,\gamma)$. 
		Also, we  have   
		\begin{align}
			\Theta=(-\infty,a), \quad (A,B)=(0,\infty), \quad \psi(\theta)=\log \Gamma(\gamma)-\gamma\log(a-\theta),\quad \psi'(\theta)=\frac{\gamma}{a-\theta},\quad \phi(x)=a-\frac{\gamma}{x}.
		\end{align}
		Note that $\psi''(\theta)=\gamma (a-\theta)^{-2}$ is increasing and log-convex on $\Theta$. 
		The function $f(x)=D(\mu_{\phi(x)}\Vert \mu)$  defined in \eqref{eq:typical_table_opt} is  $f(x)=ax - \gamma + \log \Gamma(\gamma) - \gamma \log (x/\gamma)$, which is the negative entropy of the Gamma distribution with mean $x$ and shape parameter $\gamma$. The typical table $Z^{\r,\c}=(z_{ij})$ that uniquely solves \eqref{eq:typical_table_opt} with this $f$ is given by 
		\begin{align}
			z_{ij}=\frac{\gamma}{a-(\balpha(i)+\bbeta(j))},\qquad i\in [m], j\in [n],
		\end{align}
		where $\balpha,\bbeta$ depend on $(\r,\c)$ implicitly so as to satisfy $Z^{\r,\c}\in \T(\r,\c)$.

		Our Thm. \ref{thm:second_transference_density_1} states that $X\sim \lambda_{\r,\c}$ can be approximated by a random matrix $Y$ with independent Gamma entries with mean $\E[Y]=Z^{\r,\c}$.
		Our Theorem \ref{thm:CT_limit_bdd_integer} shows that it is very close to the deterministic matrix $Z^{\r,\c}$ in cut norm with high probability as long as the margin is $\delta$-tame for some $\delta>0$. This is warranted if the ratios between the extreme values in the row/column sums are strictly less than the critical threshold $\lambda_{c}=2$ regardless of the shape parameter $\gamma$ (see Cor. \ref{cor:sharp_tame_CT}). \hfill $\blacktriangle$
	\end{example}

	\begin{example}[Laplace base measure]
		\label{ex:Laplace}
		
		Suppose the base measure is the Laplace measure with mean 0 and scale parameter $b=1$, which has a density on $\R$ with respect to the Lebesgue measure given by $\mu(dx)=\frac{1}{2}\exp(-|x|)\, dx$. In this case, we have   
		\begin{align}
			\Theta=(-1,1), \qquad (A,B)=\R, \qquad \psi(\theta)=  - \log(1-\theta^{2}) ,\qquad \psi'(\theta)= \frac{2\theta}{1-\theta^{2}},\qquad \phi(x)= 
			\frac{x}{1+\sqrt{1+x^{2}}}
		\end{align}
		In this case the function $f(x)=D(\mu_{\phi(x)}\Vert \mu)$  defined in \eqref{eq:typical_table_opt} is just $f(x)=(x+k)/2 - \frac{k}{2}\log x + \frac{k}{2}\log k$. The typical table $Z^{\r,\c}=(z_{ij})$ that uniquely solves \eqref{eq:typical_table_opt} with this $f$ is given by 
		\begin{align}
			z_{ij}=\frac{2(\balpha(i)+\bbeta(j))}{1-(\balpha(i)+\bbeta(j))^{2}},\qquad i\in [m], j\in [n],
		\end{align}
		where $\balpha,\bbeta$ depend on $(\r,\c)$ implicitly so as to sayisfy $Z^{\r,\c}\in \T(\r,\c)$. Our Theorem \ref{thm:CT_limit_bdd_integer} shows that it is very close to the deterministic matrix $Z^{\r,\c}$ in cut norm with high probability as long as the margin is $\delta$-tame for some $\delta>0$. This is warranted if the ratios between the extreme values in the row/column sums are strictly less than the critical threshold $\lambda_{c}=2$ regardless of the shape parameter $\gamma$ (see Cor. \ref{cor:sharp_tame_CT}). 
		\hfill $\blacktriangle$
	\end{example}

	\begin{example}[A counterexample to strong transference]\label{ex:counterexample}
		Suppose $\mu=\textup{Uniform}(\{0,1,\sqrt{2}\})$ and an $n\times n$ margin $(\r,\r)$ with $\r=n \mathbf{1}_{n}$. On the one hand,  note that $U\sim \mu^{\otimes (n\times n)}$ satisfies $U\in \T(\r,\r)$ if and only if $U_{ij}\equiv 1$ for all $i,j$. Thus, the conditioned random matrix $X\sim \lambda_{\r,\r}$  is almost surely the all ones matrix  $\mathbf{1}_{n}\mathbf{1}_{n}^{\top}$.  On the other hand, note that the symmetric MLE is $(\balpha,\balpha)$ for $(\r,\r)$ is given by $\balpha = \alpha^{*} \mathbf{1}_{n}$ with $\alpha^{*}=\phi(1)$, where $\phi=(\psi')^{-1}$ with $\psi(x)=\log((1+\exp(x)+ \exp(x\sqrt{2}) )/3)$.
		A direct computation shows $\alpha^{*}=\log (1+2^{-1/2})\approx 0.6232$. Thus 
		the deterministic matrix $X=\mathbf{1}_{n}\mathbf{1}_{n}^{\top}$ cannot be approximated by the random matrix $Y$ with i.i.d. entries from the non-degenerate probability distribution $\mu_{\alpha^{*}}$. For instance, the event $X=\mathbf{1}_{n}\mathbf{1}_{n}^{\top}$ occurs almost surely, but the transferred event $Y=\mathbf{1}_{n}\mathbf{1}_{n}^{\top}$ occurs with probability $\mu_{\alpha^{*}}(\{1\})^{n^{2}}=\exp(-O(n^{2}))$, where $\mu_{\alpha^{*}}(\{1\})= \exp(\alpha^{*} - \psi'(\alpha^{*}))/3 \approx 0.3532$. Consequently, subsequent results about the marginal distribution of an entry and limiting ESD based on strong transference must also not hold for $X$. For instance, the law of $X_{11}$ has point mass $\delta_{1}$  but the law of $Y_{11}$ is $\mu_{\alpha^{*}}$ with support $\{0,1,\sqrt{2}\}$. Also,  $X=\mathbf{1}_{n}\mathbf{1}_{n}^{\top}$ is rank one so its limiting ESD is the point mass $\delta_{0}$, whereas $Y\sim \mu_{\balpha\oplus \balpha}$ has i.i.d. entries from $\mu_{\alpha^{*}}$, which has limiting ESD the quarter-circle law. 
		\hfill $\blacktriangle$
	\end{example}

	\section{Proof of the strong duality} 
	\label{sec:strong_duality_pf}

	In this section, we prove Theorem \ref{thm:strong_duality_simple} on the strong duality between the MLEs and the typical tables. 
	
	First we remark that the parameters $\balpha\in \R^{m}$ and $\bbeta\in \R^{n}$ in the $(\balpha,\bbeta)$-model is not identifiable, since $(\balpha,\bbeta)$ and $(\balpha+\delta \mathbf{1}_{m}, \bbeta-\delta \mathbf{1}_{n})$ generates the same distribution on $\R^{m\times n}$. More precisely, 
	\begin{align}\label{eq:params_non_identifiable}
		\textup{$(\balpha,\bbeta)$-model}  \,\, \overset{d}{=} \,\,  \textup{$(\balpha',\bbeta')$-model} \qquad \Longleftrightarrow \qquad \textup{$\exists \lambda \in \R$ s.t. $\balpha'-\balpha \equiv \lambda \equiv \bbeta-\bbeta'$},
	\end{align}
	where the condition on the right states that the coordinate-wise difference vectors $\balpha'-\balpha$ and $\bbeta-\bbeta'$ have a constant value of $\lambda$. The above claim is easy to verify. Since we have the total sum condition $\sum_{i=1}^m \r(i) = \sum_{j=1}^n \c(j)$, one of the $m+n$ linear equations that define the polytope $\mathcal{T}(\r,\c)$ is redundant. Consequently, an MLE, even if it exists, is not unique. 
	Also, because of the non-compactness of the parameter space, the MLE may not even exist. 
	
	In the lemma below, we deduce a set of equations characterizing MLEs for a given margin $(\r,\c)$. 
	
	\begin{lemma}[The MLE equation] \label{lem:typical_MLE_equation}
		$(\hat{\balpha},\hat{\bbeta})$ is a solution of \eqref{eq:typical_Lagrangian} if and only if 
		\begin{align}\label{eq:MLE_eq}
			\sum_{j=1}^{n} \psi'(\hat{\balpha}(i)+\hat{\bbeta}(j)) = \r(i) \quad \text{for $i\in [m]$}, \quad 	\sum_{i=1}^{m} \psi'(\hat{\balpha}(i))+\hat{\bbeta}(j)))  =\c(j) \quad  \text{for $j\in [n]$},
		\end{align}
		that is, the expected table $\psi'(\balpha\oplus \bbeta)=\E_{Y\sim \mu_{\balpha+\bbeta}}[Y]$  satisfies the margin $(\r,\c)$ when $(\balpha,\bbeta)=(\hat{\balpha},\hat{\bbeta})$.
		We call \eqref{eq:MLE_eq} the MLE equation for the $(\balpha,\bbeta)$-model. 
	\end{lemma}
	\begin{proof}
		The set of admissible dual variables $(\balpha,\bbeta)$ for the MLE problem \eqref{eq:typical_Lagrangian} is open and non-empty and the objective function $g^{\r,\c}$ is smooth and concave. It follows that $(\hat{\balpha},\hat{\bbeta})$ is a global maximizer of $g^{\r,\c}$ if and only if it is a critical point of $g^{\r,\c}$, i.e., $\nabla g^{\r,\c}(\hat{\balpha},\hat{\bbeta})=\mathbf{0}$. This holds precisely if the system of equations  \eqref{eq:MLE_eq} hold. 
	\end{proof}
	
	\begin{remark} 
		When $\mu=\textup{Poisson}(1)$, the MLE equations above correspond to the Schr\"{o}dinger equations with reference measure being the uniform measure on the lattice $[m]\times [n]$ characterizing Schr\"{o}dinger potentials (see Sec. \ref{sec:schrodinger}). For this reason, we can view an MLE for a margin $(\r,\c)$ as potentials characterizing the typical table. 
	\end{remark}

	Instead of arguing the existence of MLEs directly by solving the MLE equations, we will first show that the typical table is uniquely defined under a mild condition and then establish equivalence between typical tables and MLEs. The key advantage in analyzing the typical table first is that the objective function $g$ defining the usual table in \eqref{eq:typical_table_opt} is strictly convex. 
	To see this, denote $f(x)=D(\mu_{\phi(x)}\Vert \mu)$ defined in \eqref{eq:def_rel_entropy}. 
	Note that  $f$ is strictly convex since 
	\begin{align}\label{eq:f_derivatives}
		\qquad 	f'(x) =\phi(x) + x\phi'(x) -  \psi'(\phi(x)) \phi'(x)  =  \phi(x), \,\,\,\,
		f''(x) =  \phi'(x)=\frac{1}{\psi''(\phi(x))}  =\frac{1}{\Var(\mu_{\phi(x)})}> 0.
	\end{align}
	Thus the objective function $g$  in \eqref{eq:typical_table_opt} is strictly convex on the  set $\mathcal{T}(\r,\c) \cap (A,B)^{m\times n}$.  As long as this set it is nonempty, there is a unique minimizer of $g$,  which is exactly the typical table $Z^{\r,\c}$. 
	
	\begin{lemma}[Existence and uniqueness of typical table]\label{lem:typical}
		
		For an $m\times n$ margin $(\r, \c)$, the typical table  $Z^{\r,\c}$  in \eqref{eq:typical_table_opt} exists if and only if the set $\T(\r,\c)\cap (A,B)^{m\times n}$ is non-empty. Furthermore, the typical table is unique if it exists. 
		
	\end{lemma}

	\begin{proof}
		If a typical table $Z^{\r,\c}$ exists, it belongs to the intersection $\mathcal{T}(\r,\c)\cap (A,B)^{m\times n}$. To show the other direction, suppose  there exists some $X=(x_{ij})\in \mathcal{T}(\r,\c)\cap (A,B)^{m\times n}$. Then there exists $\delta>0$ such that $X\in [A_{\delta},B_{\delta}]^{m\times n}$ (see Def. \ref{def:ab}). 
		Let $Z^{(k)}$ be a sequence of matrices in $(A,B)^{m\times n}\cap \mathcal{T}(\r,\c)$ such that
		\begin{align}
			H(Z^{(k)})\le \inf_{Z\in \mathcal{T}(\r,\c)}H(Z)+\frac{1}{k}
		\end{align}
		if the infimum is finite, and require $	H(Z^{(k)})\ge k$
		otherwise. We begin by showing that for any $i,j$,
		\begin{align}\label{eq:not_boundary}
			A<\liminf_{k\to\infty}z_{ij}^{(k)}\le \limsup_{k\to\infty}z_{ij}^{(k)}<B.
		\end{align}
		By passing to a subsequence, assume that $z^{(k)}_{ij}\to z^{(\infty)}_{ij}\in [A,B]$ for all $ij$, where $Z^{(\infty)}=(z^{(\infty)}_{ij})$ is a (possibly) extended real-valued matrix. Let 
		$$\mathcal{I}_A:=\{(i,j):z^{(\infty)}_{ij}=A\},\quad \mathcal{I}_B=\{(i,j):z^{(\infty)}_{ij}=B\},\quad \mathcal{I}_{A,B}:=\{(i,j):z^{(\infty)}_{ij}\in (A,B)\}.$$

		For any $\lambda\in [0,1]$  set $Z^{(k,\lambda)}:=(1-\lambda)Z^{(k)}+\lambda X$, and note that $Z^{(k,\lambda)}\in \mathcal{T}(\r,\c)\cap [A_{\lambda \delta},B_{\lambda\delta}]^{m\times n}$. Hence 
		\begin{align}
			H(Z^{(k)}) \le \inf_{Z\in \mathcal{T}(\r,\c)\cap [A_{\lambda\delta}, B_{\lambda\delta}]^{m\times n}	} H(Z) + \frac{1}{k} \le H(Z^{(k,\lambda)})	 
		\end{align}
		for all sufficiently large $k\ge 1$. 
		By convexity of $g$, 
		\begin{align}\label{eq:mvt}
			\frac{1}{k\lambda}\ge \frac{H(Z^{(k)})-H(Z^{(k,\lambda)})}{\lambda}\ge  \langle \nabla H(Z^{(k,\lambda)}) ,\, Z^{(k)}-X  \rangle =  \sum_{i,j} \phi\Big(z_{ij}^{(k,\lambda)}\Big) (z^{(k)}_{ij} - x_{ij}),
		\end{align}
		where we used \eqref{eq:f_derivatives}. Letting $k\to\infty$ followed by $\lambda\searrow 0$ in \eqref{eq:mvt} we get
		\begin{align}
			0 &\ge -\sum_{ij}\phi(z_{ij}^{(\infty)})(z_{ij}^{(\infty)}-x_{ij}) \\
			&= \sum_{(i,j)\in \mathcal{I}_A}\phi(A)(A-x_{ij})+\sum_{(i,j)\in \mathcal{I}_B}\phi(B)(B-x_{ij})+\sum_{(i,j)\in \mathcal{I}_{A,B}}\phi(z_{ij}^{(\infty)})(z_{ij}^{(\infty)}-x_{ij}).
		\end{align}
		Note that $\mathcal{I}_{A}=\emptyset$ if $A=-\infty$ and $\phi(A)=-\infty$ if $A$ is finite. Similarly, $\mathcal{I}_{B}=\emptyset$ if $B=\infty$ and $\phi(B)=\infty$ if $B$ is finite. Since the third term above is finite, the above equality holds only if $\mathcal{I}_A=\mathcal{I}_B=\emptyset$, which gives \eqref{eq:not_boundary}.

		Given \eqref{eq:not_boundary}, we have the existence of $\delta>0$ such that
		\begin{align}
			\inf_{Z\in \mathcal{T}(\r,\c)}H(Z)=\inf_{Z\in \T(\r,\c)\cap [A_\delta,B_\delta]^{m\times n}}H(Z).
		\end{align}
		But the RHS above focuses on a compact set, and minimizes a strictly convex function, 
		as \begin{align}\label{eq:f_derivatives2}
			f''(x) =  \phi'(x)= \frac{1}{\psi''(\phi(x))}  =\frac{1}{\Var(\mu_{\phi(x)})}> 0.
		\end{align}	
		Hence the existence of a unique optimizer follows.
	\end{proof}

	The following lemma establishes most of the strong duality stated in Theorem \ref{thm:strong_duality_simple}.

	\begin{lemma}\label{lem:strong_dual_MLE_typical}
		Fix  an $m\times n$ margin $(\r,\c)$ 
		such that the set $\mathcal{T}(\r,\c)\cap (A,B)^{m\times n}$ is non-empty. Then the following hold:
		\begin{description}[itemsep=0.1cm]
			\item[(i)] There exists a unique standard MLE $(\balpha,\bbeta)$ for margin $(\r,\c)$ which satisfies $Z^{\r,\c}=\psi'(\balpha\oplus \bbeta)$. 
			
			\item[(ii)] The typical table and the MLE problems are in strong duality: 
			\begin{align}\label{eq:strong_duality}
				\inf_{Z\in \mathcal{T}(\r,\c)} \, H(Z) = 
				\sup_{\balpha, \bbeta} \, g^{\r,\c}(\balpha, \bbeta). 
			\end{align}
			Furthermore, if there is a matrix $Z\in \T(\r,\c)\cap (A,B)^{m\times n}$ such that $Z=\psi'(\balpha\oplus \bbeta)$, 
			then $(\balpha,\bbeta)$ is an MLE for $(\r,\c)$ and $Z$ is the typical table $Z^{\r,\c}$. 
			
			\item[(iii)]
			If further $(\r,\c)$ is $\delta$-tame for some $\delta>0$,  the corresponding standard MLE $(\balpha,\bbeta)$ for $(\r,\c)$ satisfies 
			$\lVert \balpha \rVert_{\infty} \le 2C$ and $\lVert \bbeta \rVert_{\infty}\le C$ where $C:=\max\{ | \phi(A_{\delta})|,\, | \phi(B_{\delta})| \} $.
		\end{description}
	\end{lemma}

	\begin{proof}
		We first show \textbf{(i)}. Since $\T(\r,\c)\cap (A,B)^{m\times n}$ is non-empty, the typical table $Z=Z^{\r,\c}$ for $(\r,\c)$ uniquely exists by Lemma \ref{lem:typical}. 
		From \eqref{eq:f_derivatives}, we have $	\nabla  g (Z)  = \left( \phi(z_{ij})\right)_{ij}$. 
		Since $\phi$ is differentiable,  we can apply the multivariate Lagrange multiplier method, to conclude the existence of dual variables $\balpha^{*}\in \R^{m}$ and $\bbeta^{*}\in \R^{n}$ such that $\phi(z_{ij}) = \balpha^{*}_{i} + \bbeta^{*}_{j}$, 
		or equivalently, 
		\begin{align}\label{eq:typical_lagrange_multiplier_formula_pf}
			z_{ij} = \psi'(  \balpha^{*}_{i} + \bbeta^{*}_{j}).
		\end{align}
		Since $Z$ satisfies the margins $(\r,\c)$, Lemma \ref{lem:typical_MLE_equation} yields that $(\balpha^{*}, \bbeta^{*})$ is an MLE. Then by using the shift equivalence of dual variables in \eqref{eq:params_non_identifiable}, it follows that a standard MLE for $(\r,\c)$ exists.

		For the uniqueness of the standard MLE, suppose $(\balpha',\bbeta')$ is another standard MLE for the margin $(\r,\c)$. Again using Lemma \ref{lem:typical_MLE_equation} 
		we conclude
		\begin{align}
			\balpha(i)+\bbeta(j) = \phi(z_{ij}) = 	\balpha'(i)+\bbeta'(j)
		\end{align}
		for all $i,j$. Since $\balpha(1)=\balpha'(1)=0$, the above relation yields $\bbeta=\bbeta'$. Consequently, the above relation yields $\balpha=\balpha'$ 
		and we have completed the proof of part \textbf{(i)}.
		
		Next, we show \textbf{(ii)}. 	For each $(\balpha,\bbeta)\in \R^{m}\times \R^{n}$, define $X^{\balpha,\bbeta}=(x_{ij})_{i,j} \in \R^{m\times n}$ by 
		\begin{align}\label{eq:lem_MLE_typical_duality3}
			x_{ij} := \psi'(\balpha(i)+\bbeta(j)) \quad \textup{for all $1\le i \le m$ and $1\le j \le n$}. 
		\end{align}
		Fix dual variables $(\balpha,\bbeta)$ and let  $X^{\balpha, \bbeta}=(x_{ij})_{i,j}$ be as in \eqref{eq:lem_MLE_typical_duality3}. Then
		\begin{align}\label{eq:entropy_likelihood_computation}
			-	H(X^{\balpha,\bbeta}) 
			&=  \sum_{i,j} \psi(\phi(x_{ij})) -  \sum_{i=1}^{m}\sum_{j=1}^{n}  (\balpha(i)+\bbeta(j))x_{ij}  \\
			&= \sum_{i,j} \psi( \balpha(i)+\bbeta(j) )  -  \sum_{i=1}^{m} \balpha(i) \sum_{j=1}^{m} x_{ij} - \sum_{j=1}^{n} \bbeta(j) \sum_{i=1}^{n} x_{ij}  \\
			&= -g^{\r,\c}(\balpha, \bbeta) + \sum_{i=1}^{m} \balpha(i) \left(\r(i)-\sum_{j=1}^{m} x_{ij}\right) + \sum_{j=1}^{n} \bbeta(j) \left(\c(j) - \sum_{i=1}^{n} x_{ij} \right),
		\end{align}
		where $g^{\r,\c}(\balpha, \bbeta)$ is defined in \eqref{eq:typical_Lagrangian}.

		Now suppose $(\hat{\balpha},\hat{\bbeta})$ is any MLE for margin $(\r,\c)$, which exists by the previous part. Then by Lemma \ref{lem:typical_MLE_equation}, $X^{\hat{\balpha},\hat{\bbeta}}\in \mathcal{T}(\r,\c)\cap (A,B)^{m\times n}$. Thus the above yields 
		\begin{align}\label{eq:strong_duality_pf1}
			H(X^{\hat{\balpha},\hat{\bbeta}})  = g^{\r,\c}(\hat{\balpha},\hat{\bbeta}) =  \sup_{\balpha,\bbeta} g^{\r,\c}(\balpha, \bbeta). 
		\end{align}

		It remains to show that $X^{\hat{\balpha},\hat{\bbeta}}$ is the typical table for margin $(\r,\c)$. Denote the typical table for margin $(\r,\c)$ as $Z^{\r,\c}=(z_{ij})\in (A,B)^{m\times n}$, which exists by Lemma \ref{lem:typical}. 
		By part (a), there exists a (standard) MLE $(\balpha^{*}, \bbeta^{*})$ for margin $(\r,\c)$ such that $Z^{\r,\c}=X^{\balpha^{*},\bbeta^{*}}$. Then by using \eqref{eq:strong_duality_pf1} with $(\balpha^{*}, \bbeta^{*})$ and $(\hat{\balpha}, \hat{\bbeta})$, it follows that  
		\begin{align}
			\sup_{\balpha,\bbeta} g^{\r,\c}(\balpha, \bbeta) =H(X^{\balpha^{*}, \bbeta^{*} }) 	=H(Z^{\r,\c}) \le H(X^{\hat{\balpha},\hat{\bbeta}})  = g^{\r,\c}(\hat{\balpha},\hat{\bbeta}) =  \sup_{\balpha,\bbeta} g^{\r,\c}(\balpha, \bbeta).
		\end{align}
		Thus all terms that appear above must equal, verifying \eqref{eq:strong_duality}.

		Lastly, suppose there exists $X\in \T(\r,\c)\cap (A,B)^{n\times m}$ that admits the decomposition $X=X^{\balpha,\bbeta}$ for some $(\balpha,\bbeta)\in \R^{m}\times \R^{n}$.  Since $X$ has margin $(\r,\c)$, it follows that $(\balpha,\bbeta)$ is an MLE for $(\r,\c)$ due to Lemma \ref{lem:typical}. Furthermore, by \eqref{eq:strong_duality_pf1} and  \eqref{eq:strong_duality}, 
		\begin{align}\label{eq:strong_duality2}
			H(X) =  -\sup_{\balpha, \bbeta} \, g^{\r,\c}(\balpha, \bbeta) = 	\sup_{X\in \mathcal{T}(\r,\c)\cap (A,B)^{m\times n}} \, H(X),
		\end{align}
		so $X$ is a typical table for $(\r,\c)$. By the uniqueness, $X=Z^{\r,\c}$. 
		
		Lastly, we show \textbf{(iii)}. 	By part \textbf{(i)}, there exists a unique standard MLE $(\balpha,\bbeta)$ for margin $(\r,\c)$ such that for all $i,j$, $	\phi(z_{ij}) = \balpha(i)+\bbeta(j)$ 
		and $\balpha(1)=0$. Since $Z$ is $\delta$-tame, we get $\phi(A_{\delta})	\le 	\balpha(i) + \bbeta(j)  \le \phi(B_{\delta})$ for all $i,j$. 
		Setting $i=1$ and recalling that $\balpha(1)=0$, the above gives $\phi(A_{\delta})	\le  \bbeta(j)  \le \phi(B_{\delta})$ for all $j$. 
		In turn, it follows that  
		\begin{align}
			\phi(A_{\delta}) - \phi(B_{\delta})	\le  \balpha(i)  \le \phi(B_{\delta}) - \phi(A_{\delta})  \quad \textup{for all $i$}. 
		\end{align}
	\end{proof}
	
	Theorem  \ref{thm:strong_duality_simple} can now be deduced easily from the lemmas above. 
	
	\begin{proof}[\textbf{Proof of Theorem \ref{thm:strong_duality_simple}}]
		In Lemma \ref{lem:typical}, we have already shown that $Z^{\r,\c}$ exists if and only if $\T(\r,\c)\cap (A,B)^{m\times n}$ is non-empty. The first equivalence in \eqref{eq:lem_MLE_typical_duality1} is the content of Lemma \ref{lem:typical_MLE_equation}. Now 
		if $Z^{\r,\c}$ exists, then by Lemma \ref{lem:strong_dual_MLE_typical} \textbf{(i)}, a unique standard MLE $(\balpha,\bbeta)$ for $(\r,\c)$ exists and $Z^{\r,\c}=\psi'(\balpha\oplus \bbeta)$. Conversely, if an MLE $(\balpha,\bbeta)$ exists, then by Lemma \ref{lem:typical_MLE_equation} $Z:=\E[\mu_{\balpha\oplus \bbeta}]\in \T(\r,\c)\cap (A,B)^{m\times n}$. Then Lemma \ref{lem:strong_dual_MLE_typical} \textbf{(ii)} yields that $Z$ is the typical table $Z^{\r,\c}$. This shows that $Z^{\r,\c}$ exists if and only if an MLE exists and also that the second equivalence in \eqref{eq:lem_MLE_typical_duality1} holds. 
	\end{proof}

	\section{Proof of the transference principles} 
	\label{sec:transference_pf}
	
	\subsection{The weak transference principle}

	In this section, we prove the weak transference principle in Thm. \ref{thm:transference}. To prove the second part, we need the following lemma.

	\begin{lemma}[Concentration of quadratic forms for the $(\balpha,\bbeta)$-model]\label{lem:concentration_quadratic_form}
		Let $(\r,\c)$ be an $m\times n$ $\delta$-tame margin and $Y\sim \mu_{\balpha\oplus \bbeta}$, where $(\balpha,\bbeta)$ is an MLE for the margin $(\r,\c)$. Define positive constants 
		\begin{align}
			&L^{-}:= \min\left\{  \phi(A_{\delta})-\phi(A_{\delta/2}),\,  \phi(B_{\delta/2}) - \phi(B_{\delta}) \right\}, \quad L^{+}:=\sup_{|s|\le L^{-}} \psi''(s). 
		\end{align}
		Denote $\tilde{Y}=Y-\E[Y]$. For each $t>0$, $\x \in [-1,1]^{m}$, and $\y\in [-1,1]^{n}$, 
		\begin{align}\label{eq:gen_hoeffding_pf_00}
			\P\left(   \x^{\top}\tilde{Y}\y    \ge t \lVert \x \rVert^{2} \lVert \y \rVert^{2} \right) \le 2  \exp\left( -\lVert \x \rVert^{2} \lVert \y \rVert^{2} \frac{ t (t\land L^{-}L^{+})}{2L^{+}}    \right).
		\end{align}
		Furthermore, for $s\in [0, L^{-}L^{+}]$,
		\begin{align}\label{eq:Y_margin_prob_lower_bd}
			\P(Y\in \T_{smn }(\r,\c)  ) 			&\ge 1- 3^{m+n}  \exp\left( -\frac{s^{2}mn}{2L^{+}} \right).
		\end{align}  
		In particular, for $\rho= \sqrt{8L^{+} mn(m+n)}$ and if $m,n\ge 1$ are large enough so that $\frac{1}{m} + \frac{1}{n} \le (L^{-})^{2} L^{+}/8$, 
		\begin{align}\label{eq:Y_margin_prob_lower_bd2}
			\P(Y\in \T_{\rho}(\r,\c)) \ge 1 - (1/3)^{m+n}. 
		\end{align}
	\end{lemma}
	
	\begin{proof}
		We will first show \eqref{eq:gen_hoeffding_pf_00}. Note that if $T\sim \mu_{\theta}$, then 
		\begin{align}
			&\E[\exp(s (T-\E[T]))] 
			= \exp(\psi(s+\theta)-\psi(\theta) - s\psi'(\theta)).
		\end{align}
		Write $t':= t \lVert \x \rVert^{2} \lVert \y \rVert^{2}$ and $s'_{ij}=s\, \x(i)\y(j)$. For each $s\ge 0$, we have 
		\begin{align}
			\P\left(   \x^{\top}\tilde{Y}\y    \ge  t' \right) 
			&\le   \exp\left( -s t'   + \sum_{i,j}  \psi(s'_{ij}+ \theta_{ij} ) - \psi(\theta_{ij}) - s'_{ij}\psi'(\theta_{ij}) \right) .
		\end{align}
		Denote $\varphi(s):= \psi(s+\theta) - \psi(\theta) - s\psi'(\theta)$. Then 
		\begin{align}
			& \varphi'(s)= \psi'(s+\theta) - \psi'(\theta),\quad  \varphi'(0) = 0,\quad 		 \varphi''(s)= \psi''(s+\theta)\ge 0.
		\end{align}
		Hence $\varphi$ is concave and is minimized at $0$. Let $Z=(z_{ij})$ denote the typical table for $(\r,\c)$. Since $(\r,\c)$ is $\delta$-tame, by Lemma \ref{lem:strong_dual_MLE_typical}, $\E[Y]=Z\in [A_{\delta},B_{\delta}]^{m\times n}$. So we have  
		\begin{align}\label{eq:gen_hoeffding_pf_1}
			\phi(A_{\delta/2})	\le 	s+\phi(z_{ij})  \le 	\phi(B_{\delta/2})
		\end{align}
		whenever $|s| \le L^{-}$. Hence $\sup_{|s|\le L^{-} } |\varphi''(s)| \le L^{+}$.  Then by Taylor's theorem, 
		\begin{align}
			\varphi(s) \le \varphi(0) + \varphi'(0)s + \frac{L^{+}}{2}s^{2} = \frac{L^{+}}{2}s^{2} \qquad \textup{for all $s\in [-L^{-}, L^{-}]$}.
		\end{align}
		Applying the above bound for $\theta=\phi(\E[Y_{ij}])$ and $s=s'_{ij}$ all $i,j$, from the previous inequality we get 
		\begin{align}\label{eq:gen_hoeffding_pf_2}
			\P\left(   \x^{\top}\tilde{Y}\y    \ge t \lVert \x \rVert^{2} \lVert \y \rVert^{2} \right) \le \exp\left( -\lVert \x \rVert^{2} \lVert \y  \rVert^{2} \left( s t - \frac{L^{+}}{2} s^{2}  \right)  \right) \quad \textup{for all $s\in [-L^{-}, L^{-}]$},
		\end{align}
		where we have also used $\sum_{i,j} \x(i)^{2}\y(j)^{2} = \lVert \x \rVert^{2} \lVert \y  \rVert^{2}$. 
		In order to optimize the above bound, denoting $\bar{t} = t\land L^{-}L^{+}$, write 
		\begin{align}
			s t - \frac{L^{+}}{2} s^{2}  &=  s(t-\bar{t}) +   \left( s  \bar{t} - \frac{L^{+}}{2} s^{2} \right). 
		\end{align}
		The quadratic function in $s$ in the second term of the right-hand side above is minimized at $s=\bar{t} / L^{+}  \in [-L^{-},L^{-}]$ with minimum value $\bar{t}^{2}/2L^{+}$. For this choice of $s$, and noting $t\ge \bar{t}$, the above is at least $\frac{t\bar{t}}{2L^{+}}$. Hence we deduce \eqref{eq:gen_hoeffding_pf_00} from \eqref{eq:gen_hoeffding_pf_2}.

		Next, we deduce \eqref{eq:Y_margin_prob_lower_bd}. Fix $\x\in [-1,1]^{m}$ and $\y\in [-1,1]^{n}$. Substituting $t = s \frac{mn}{\lVert \x \rVert^{2} \rVert \y \rVert^{2} }$ for $s\in [0,L^{-}L^{+} ]$ in \eqref{eq:gen_hoeffding_pf_00}, we get 
		\begin{align}\label{eq:gen_hoeffding_pf_000}
			\P\left(   \x^{\top}\tilde{Y}\y   \ge s mn \right) &\le 2 \exp\left( -\frac{smn}{2L^{+}}  \left(  \frac{smn}{\lVert \x \rVert^{2} \rVert \y \rVert^{2} } \land L^{-}L^{+}  \right)  \right) \le 2 \exp\left( -\frac{s^{2}mn}{2L^{+}}   \right),
		\end{align}
		where the last inequality uses $\lVert \x \rVert^{2} \le m$ and $\lVert \y\rVert^{2}\le n$.  Now observe that 
		\begin{align}
			\lVert r(\tilde{Y})  \rVert_{1}  &=  \lVert r(\tilde{Y})^{+}  \rVert_{1} + \lVert r(\tilde{Y})^{-}  \rVert_{1}  \le 2 \max_{\x\in \{0,1\}^{m}} | \x^{\top}\tilde{Y} \mathbf{1}_{n} | \le 2 \max_{\x\in \{-1,0,1\}^{m}, \y\in \{-1,0,1\}^{n}}  \x^{\top}\tilde{Y} \y .
		\end{align}
		The same upper bound holds for $			\lVert c(\bar{Y})  \rVert_{1}$. Thus,
		\begin{align}
			\P(Y\notin \T_{smn }(\r,\c)  ) &= \P\left( \min\left\{ \lVert r(Y)-\r  \rVert_{1} ,\, \lVert c(Y)-\c  \rVert_{1} \right\}  \ge s mn \right) \\
			&\le  \P\left( \max_{\x\in \{-1,0,1\}^{m}, \y\in \{-1,0,1\}^{n}}  \x^{\top}\bar{Y} \y   \ge s mn \right) \le 3^{m+n} \exp\left( -\frac{s^{2}mn}{2L^{+}} \right). 
		\end{align}

		Lastly, choose $s=\sqrt{8L^{+}(m^{-1}+n^{-1})}$. Then $s\le L^{-}L^{+}$ since $\frac{1}{m} + \frac{1}{n} \le (L^{-})^{2} L^{+}/8$, so we can apply \eqref{eq:Y_margin_prob_lower_bd}. The bound \eqref{eq:Y_margin_prob_lower_bd2} then follows immediately. 
	\end{proof}

	\begin{proof}[\textbf{Proof of Theorem \ref{thm:transference}}]

		Let $C=2\max\{ \phi(A_{\delta}),\phi(B_{\delta}) \}$. Then by Lemma \ref{lem:strong_dual_MLE_typical} \textbf{(iii)}, we can choose $(\balpha,\bbeta)$ so that their  $L^{\infty}$-norm is at most $C$. Now for each $\x\in \T_{\rho}(\r,\c)$, by H\"{o}lder's inequality, 
		\begin{align}
			|	g^{r(\x),c(\x)}(\balpha,\bbeta) - g^{\r,\c}(\balpha,\bbeta) | 
			&\le C \|(r(\x),c(\x))-(\r,\c)\|_{1} 
			\le C\rho =: D. 
		\end{align}
		Hence 
		\begin{align}
			&	\sup_{\x\in \T_{\rho}(\r,\c)} \exp \left( g^{r(\x),c(\x)}(\balpha,\bbeta )\right) \le \exp\left( g^{\r,\c}(\balpha,\bbeta) + D\right), \\
			&	\inf_{\x\in \T_{\rho}(\r,\c)} \exp \left( g^{r(\x),c(\x)}(\balpha,\bbeta) \right) \ge \exp\left( g^{\r,\c}(\balpha,\bbeta) -D\right). 
		\end{align}
		Recall the hypothesis $ \mu^{\otimes (m\times n)}\left( \T_{\rho}(\r,\c) \right)  \in (0,\infty)$. For each measurable set  $\mathcal{E}\subseteq \R^{m\times n}$, 
		\begin{align}\label{eq:Y_T_rho_pf}
			\P( Y\in \mathcal{E} \,|\,   Y\in \T_{\rho}(\r,\c)) &= \frac{ \int_{\T_{\rho}(\r,\c)}  \mathbf{1}(\x\in \mathcal{E}) \exp(g^{r(\x), c(\x)}(\balpha',\bbeta')) \, \mu^{\otimes (m\times n)}(d\x) }{\int_{\T_{\rho}(\r,\c)}  \exp(g^{r(\x), c(\x)}(\balpha,\bbeta)) \,\, \mu^{\otimes (m\times n)}(d\x)  } \\
			&\ge \frac{\exp\left( g^{\r,\c}(\balpha',\bbeta') - D \right) }{\exp\left( g^{\r,\c}(\balpha',\bbeta') + D \right)  } \frac{ \int_{\T_{\rho}(\r,\c)}  \mathbf{1}(\x\in \mathcal{E})  \, \, \mu^{\otimes (m\times n)}(d\x) }{\int_{\T_{\rho}(\r,\c)} \,\,  \mu^{\otimes (m\times n)}(d\x)  } \ge \exp(-2D) 	\, \P( X\in \mathcal{E} ). 
		\end{align}
		This is enough to conclude \eqref{eq:transference_1}. The second part in \eqref{eq:transference_11} follows immediately from  \eqref{eq:transference_1} and Lemma \ref{lem:concentration_quadratic_form}. 
	\end{proof}

	\subsection{The strong transference principles}
	\label{sec:strong_transference_pf}

	Next, we prove the strong transference principle stated in Theorem \ref{thm:second_transference} and prove the subsequent transference results in Lem. \ref{lem:second_transference_density}, Thm. \ref{thm:second_transference_density_1}, and \ref{thm:transfer_counting_Leb}.  Note that here we do not necessarily assume that the conditioning set $\T(\r,\c)$ has a positive measure under $\mu^{\otimes(m\times n)}$.

	We give some further details of the disintegration construction of $\lambda_{\r,\c}$ under Assumption \ref{assumption:strong_transference}. 
	Since $\mu$ is assumed to be $\sigma$-finite, according to \cite[Thm. 1 and 2]{chang1997conditioning}, there exists a family of $\sigma$-finite Borel  measures  $\{ \lambda_{\mathbf{t}} \}$ on $\R^{m+n}$ that disintegrate $\mu^{\otimes (m\times n)}$ w.r.t. $\pi$ (i.e., \textit{$\pi$-disintegration}):   
	\begin{description}[leftmargin=0.55cm, itemsep=0.1cm]
		\item[(i)] For $\eta$-almost all $\mathbf{t}\in \R^{m+ n}$, $\lambda_{\mathbf{t}}$  lives on $\{ \pi = \mathbf{t} \}$, that is, $\lambda_{\mathbf{t}}\{ \pi \ne \mathbf{t} \} =0$. 
		Also, each $\lambda_{\mathbf{t}}$ is a probability measure.
		
		\item[(ii)] For each nonnegative measurable function $h:\R^{m\times n}\rightarrow \R$, $\mathbf{t}\mapsto \int h \,d\lambda_{\mathbf{t}}$ is measurable and 
		\begin{align}\label{eq:disintegration_formula1}
			\int h(\x) \, \mu^{\otimes(m\times n)}(d\x) = \iint  h(\x) \,\lambda_{\mathbf{t}}(d\x) \, \nu(d\mathbf{t}).
		\end{align}
	\end{description}
	Furthermore, such a $\pi$-disintegration of $\mu^{\otimes (m\times n)}$ is unique up to an almost-sure equivalence: if $\{ \lambda_{\mathbf{t}}^{*} \}$ is another $\pi$-disintegration of $\mu^{\otimes (m\times n)}$, then $\nu\{  \mathbf{t}\,:\, \lambda_{\mathbf{t}}\ne \lambda_{\mathbf{t}}^{*} \}=0$.

	Since $\nu$ is assumed to be $\sigma$-finite under Assumption \ref{assumption:strong_transference}, 
	the above properties ensure that for 	$\nu$-almost all $(\r,\c)\in \R^{m\times n}$, and for each $\mu^{\otimes(m\times n)}$-integrable random variable $S$, $\int S \, d\lambda_{\mathbf{t}}$ is a version of the conditional expectation $\E[S\,|\, \pi=(\r,\c)]$. Hence $\lambda_{\r,\c}$ gives the law $\P(\cdot\,|\, \pi=(\r,\c))$ of the margin-conditioned random matrix $X$ for $\nu$-almost all margins.

	\begin{proof}[\textbf{Proof of Theorem \ref{thm:second_transference}}]

		Recall that the the law of $Y$ is absolutely continuous w.r.t. $\mu^{\otimes(m\times n)}$ with probability density $	\exp(g^{r(\cdot),c(\cdot)}(\balpha,\bbeta))$,  
		which is a positive constant over each  fiber $\pi^{-1}(\r,\c)$. 
		Let $\nu_{Y}:=\pi_{\#}(\mu_{\balpha\oplus \bbeta})$ denote the pushforward of the law $\mu_{\balpha\oplus \bbeta}$ of $Y$ on  $\R^{m+n}$. Then $\nu_{Y}$ is absolutely continuous w.r.t. $\nu=\pi_{\#}(\mu^{\otimes (m\times n)})$ with positive density $\exp(g^{\r(\cdot),\c(\cdot)}(\balpha,\bbeta))$. Since $\nu$ is assumed to be $\sigma$-finite, it follows that $\nu_{Y}$ also $\sigma$-finite. Hence  \cite[Thm. 3]{chang1997conditioning} implies that the $\pi$-disintegration  $\{ \lambda_{\mathbf{t}} \}$ for $\mu^{\otimes (m\times n)}$ is the $\pi$-disintegration for $\mu_{\balpha\oplus \bbeta}$ as well. It follows that the law of  $Y$ conditional on $Y\in \T(\r,\c)$ is given by $\lambda_{\r,\c}$, which is the law of $X$ conditional on $X\in \T(\r,\c)$ for $\nu$-almost all margins. This shows \textbf{(i)}. 
		
		Next, we show \textbf{(ii)}.  Let 	$\overline{\lambda}_{\r,\c}$ denote the law of $\overline{Y}$ conditional on $Y\in \T(\r,\c)$. Let $\overline{\mu}_{\balpha\oplus \bbeta}$ denote the law of unconditional $\overline{Y}$. Note that $\overline{\lambda}_{\r,\c}$ is the pullback measure of $\lambda_{\r,\c}$ via the completion map $\Gamma_{\r,\c}$. 
		To show $\overline{\lambda}_{\r,\c}\ll \overline{\mu}_{\balpha\oplus \bbeta}$ for $\nu$-almost all margins, fix $S\subseteq \R^{(m-1)\times (n-1)}$ write $\y=(\overline{\y},\check{\y})$. Using the disintegration property, 
		\begin{align}
			\P(\overline{Y}\in S) &= \E\left[ \P\left(\overline{Y}\in S \,\bigg|\, r(Y),c(Y) \right) \right] = \int \overline{\lambda}_{\r',\c'}(S) \exp(g^{\r',\c'}(\balpha,\bbeta)) \,\, \nu (d (\r',\c') ). 
		\end{align}
		Thus if $	\overline{\mu}_{\balpha\oplus \bbeta}(S)=0$, then $\overline{\lambda}_{\r',\c'}(S)=0$ for $\nu$-almost all margins $(\r',\c')$. This yields the relative density $p_{\r,\c}$ for which $	\overline{\lambda}_{\r,\c}(d\overline{\y}) = p_{\r,\c}(\overline{\y}) \overline{\mu}_{\balpha\oplus\bbeta}(d\overline{\y}) $ for $\nu$-almost all margins. 
		By using \textbf{(i)}, for $\nu$-almost all margins and measurable functions $h$, 
		\begin{align}
			\E[h(\overline{X})] = \int h(\overline{\x})\,\overline{\lambda}_{\r,\c}(d\overline{\x})=\int h(\overline{\y})p_{\r,\c}(\overline{\y}) \overline{\mu}_{\balpha\oplus\bbeta}(d\overline{\y}) = \E[p_{\r,\c}(\overline{Y}) h(\overline{Y})].
		\end{align}

		Now in order to conclude \eqref{eq:strong_transference_exact} from the above, it is enough to justify Bayes' theorem to deduce $p_{\r,\c}(\overline{\y})= p_{\overline{\y}}(\r,\c)$ for $\overline{\mu}_{\balpha\oplus \bbeta} \otimes \nu$-almost all $(\overline{\y}, (\r,\c))$. 
		
		Given $\overline{\y}\in \R^{(m-1)\times (n-1)}$, $\check{\y}\in \R^{m+n-1}$  is in one-to-one correspondence with the margin $(\r,\c)$ by the relationship $\y_{ij}=\Gamma_{\r,\c}(\overline{\y})$ for $(i,j)$ with $i=m$ or $j=n$. That is, $\y_{in}=\r(i)-\overline{\y}_{i\bullet}$ for $1\le i<m$, $\y_{mj}= \c(j)-\overline{\y}_{\bullet j}$ for $1\le j <n$, and $\y_{mn}=\overline{\y}_{\bullet\bullet} -( \sum_{i=1}^{m-1}\r(i) )+\c(n)$. Accordingly, we define a new measure $\zeta$ on $\R^{m\times n}$ as the pushforward of pushforward of the law of $Y$ via the one-to-one map $S:\y\mapsto (\overline{\y},(r(\y),c(\y)))$.  Let $\pi_{i}$ for $i=1,2$ denote the projection of $(\overline{\y},(r(\y),c(\y)))$ onto the $i$th coordinate. Clearly $(\pi_{1})_{\#}(\zeta) = \overline{\mu}_{\balpha\oplus \bbeta}$ and $(\pi_{2})_{\#}(\zeta)=\nu_{Y}$. Then   $\zeta(d\overline{\y}, d(\r,\c)):= \overline{\lambda}_{\r,\c}(d\overline{\y}) \, \nu_{Y}( d(\r,\c) ) $, meaning that the $\pi_{2}$-disintegration of $\zeta$ is given by $\overline{\lambda}_{\r,\c}$. Indeed, for nonnegative functions $h$, using the fact that the law of $\overline{Y}$ given its margin $(\r,\c)$ is $\overline{\lambda}_{\r,\c}$, 
		\begin{align}
			\E[h(\overline{Y}, r(Y), c(Y))] = \E\left[ (h\circ S)(Y) \right] &= \E[ \E[ (h\circ S)(Y) \,|\, r(Y),c(Y) ]  ] \\
			&=  \iint h(\overline{\y}, (\r,\c))\, \overline{\lambda}_{\r,\c}(d\overline{\y})\, \nu_{Y}(d(\r,\c)). 
		\end{align}
		
		Next, we disintegrate the measure $\zeta$ using the other projection $\pi_{1}: (\overline{\y}, (\r,\c))\mapsto \overline{\y}$. Let $\nu_{\overline{\y}}(d(\r,\c)) $ give the $\pi_{1}$-disintegration of $\zeta$, which are the laws of the margin of $Y$ given $\overline{Y}=\overline{\y}$. Since the pushforward of $\zeta$ via this projection map is $\overline{\mu}_{\balpha\oplus \bbeta}$, the two different ways of disintegrating $\zeta$ give the following Bayes' theorem: 
		\begin{align}
			\overline{\lambda}_{\r,\c}(d\overline{\y}) \, \nu_{Y}( d(\r,\c) )  =\nu_{\overline{\y}}(d(\r,\c))  \,  \overline{\mu}_{\balpha\oplus \bbeta}(d\overline{\y}).
		\end{align}

		\noindent  Using $\P((r(Y),c(Y))\in E) = \E[ \P((r(Y),c(Y))\in E) \,|\, \overline{Y} ) ]$, we see that $\nu_{\overline{\y}}\ll \nu_{Y}$ for $\overline{\mu}_{\balpha\oplus \bbeta}$-almost all $\overline{\y}$'s. Hence the relative density $p_{\overline{\y}}(\r,\c)$ exists such that $\nu_{\overline{\y}}(d(\r,\c))=p_{\overline{\y}}(\r,\c)  \nu_{Y}(d(\r,\c))$. 
		It follows that 
		\begin{align}
			p_{\r,\c}(\overline{\y}) \overline{\mu}_{\balpha\oplus \bbeta}(d\overline{\y})  \nu_{Y}(d(\r,\c))  =  p_{\overline{\y}}(\r,\c)  \nu_{Y}(d(\r,\c)) \, \overline{\mu}_{\balpha\oplus \bbeta}(d\overline{\y}).
		\end{align}
		From this we conclude $p_{\r,\c}(\overline{\y})=p_{\overline{\y}}(\r,\c)$ for $\overline{\mu}_{\balpha\oplus \bbeta} \otimes \nu_{Y}$-almost all $(\overline{\y}, (\r,\c))$. Since $\nu_{Y}$ has a positive density w.r.t. $\nu$, this finishes the proof. 
	\end{proof}

	Next, we prove Corollary \ref{cor:transfer_discrete_Leb}. We extract the key mechanism behind this result in the following corollary of Thm. \ref{thm:second_transference}. 
	
	\begin{corollary}[Sufficient condition for strong transference]\label{cor:second_transference_density}  
		Keep the same setting as in Thm. \ref{thm:second_transference}.
		Further assume that there exists a Borel measure $\zeta$ on $\R$ and a constant $K>0$ such that 
		\begin{description}[leftmargin=0.9cm]
			\item{(a)} \,\, For almost all $\overline{\y}$ under the law of $\overline{Y}$, $  \nu_{\overline{\y}}(\cdot)\ll \zeta^{\otimes (m+n-1)}$ with bounded Radon-Nikodym derivative $q_{\overline{\y}}(\r,\c)\le K^{m+n-1}$; and 
			
			\item{(b)} \,\, $\zeta^{\otimes (m+n-1)}\ll \nu$ locally $\nu$-a.s. That is, there exists a $\sigma$-finite Borel set  $V$ such that $\nu_{Y}(V)=1$ and for each $(\r',\c')\in V$, there exists an open neighborhood $U$ of $(\r',\c')$ such that $\zeta^{\otimes (m+n-1)}|_{U}\ll \nu|_{U}$.

		\end{description}
		Then  \eqref{eq:strong_transference_exact} in Thm. \ref{thm:second_transference} holds with 
		\begin{align}\label{eq:thm_double_transfer2}
			\sup_{\overline{\y}} p_{\overline{\y}}(\r,\c)  =	\big( \sup_{\overline{\y}} q_{\overline{\y}}(\r,\c) \big) \, \E[q_{\overline{Y}}(\r,\c)]^{-1}.
		\end{align}
	\end{corollary}

	\begin{proof}
		
		We claim that for each $(\r,\c)\in V$, the map $\overline{\y}\mapsto p_{\overline{\y}}(\r,\c)$ is proportional to $q_{\overline{\y}}(\r,\c)$. This will be enough to conclude. Indeed, taking $h\equiv 1$ in \eqref{eq:strong_transference_exact} in Thm. \ref{thm:second_transference}, 
		we have $\E[p_{\overline{Y}}(\r,\c) ]=1$. Hence from the claim and since $\nu_{Y}(V)=1$, the normalizing constant for $p_{\overline{\y}}(\r,\c)$ must be $\E[q_{\overline{Y}}(\r,\c)]^{-1}$. This and \eqref{eq:strong_transference_exact} yield 
		\begin{align}
			\E[h(\overline{X})] = \E\left[ \frac{q_{\overline{Y}}(\r,\c)}{\E[q_{\overline{Y}}(\r,\c)]} h(\overline{Y})\right] \le 
			K^{m+n-1} \, \E[q_{\overline{Y}}(\r,\c)]^{-1}  \E[h(\overline{Y})],
		\end{align}
		as desired. 
		
		It remains to justify the claim. For simplicity denote $\tilde{\zeta}:=\zeta^{\otimes (m+n-1)}$. 
		First note that the hypothesis, for $\overline{\mu}_{\balpha\oplus \bbeta}\otimes \nu_{Y}$-almost all $(\overline{\y}, (\r,\c))$,
		\begin{align}\label{eq:p_q_nu_zeta}
			p_{\overline{\y}}(\r,\c) \, \nu_{Y}(d(\r,\c)) =    \nu_{\overline{\y}}(d(\r,\c))  = q_{\overline{\y}}(\r,\c) \,\, \tilde{\zeta}(d(\r,\c)).
		\end{align}
		Fix $(\r,\c)\in V$. By the hypothesis, there exists an open neighborhood $U\in \R^{m+n-1}$ of $(\r,\c)$ such that $\tilde{\zeta}|_{U}\ll \nu|_{U}$. Hence the local Radon-Nikodym derivative $\frac{d\tilde{\zeta}|_{U}}{d \nu_{Y}|_{U}}$ exists. From \eqref{eq:p_q_nu_zeta}. it follows that, for $\overline{\mu}_{\balpha\oplus \bbeta}$-almost all $\overline{\y}$,
		\begin{align}   
			p_{\overline{\y}}(\r,\c) = q_{\overline{\y}}(\r,\c) \frac{d\tilde{\zeta}|_{U} }{d\nu_{Y}|_{U}}(\r,\c). 
		\end{align}
		This shows the claim, as desired. 
	\end{proof}

	\begin{prop}\label{prop:pi_open_map}
		The map $\pi:\R^{m\times n}\rightarrow \R^{m+n-1}$ defined by  $\pi(\x)=(r_{1}(\x),\dots,r_{m-1}(\x),c_{1}(\x),\dots,c_{n}(\x) )$ is an open map. 
	\end{prop}
	
	\begin{proof}
		The proof is straightforward and we omit the details.
	\end{proof}

	\begin{proof}[\textbf{Proof of Corollary \ref{cor:transfer_discrete_Leb}}]
		We will verify the hypothesis of Cor. \ref{cor:second_transference_density} holds.  Given that, let  $p^{ij}$ denote the Radon-Nikodym derivative of  $\mu_{\balpha(i)+\bbeta(j)}$ w.r.t. $\zeta$. 
		Then using the expression in \eqref{eq:nu_ybar_pullback} for $\nu_{\overline{\y}}$, 
		\begin{align}\label{eq:q_y_density}
			q_{\overline{\y}}(\r,\c)=\frac{d\nu_{\overline{\y}}}{d\widetilde{\zeta}} = \prod_{\textup{$i=m$ or $j=n$}} 
			p^{ij}(\Gamma_{\r,\c}(\overline{\y})_{ij}). 
		\end{align}
		Specifically, under the hypothesis, 
		\begin{align}\label{eq:p_rc_density_formula}
			p^{ij}(x) = \exp(x(\balpha(i)+\bbeta(j))- \psi(\balpha(i)+\bbeta(j)) ) p(x). 
		\end{align}
		It follows that 
		\begin{align}\label{eq:p_rc_density_formula2}
			\E[q_{\overline{Y}}(\r,\c)]  = 	\exp(g^{\r,\c}(\balpha,\bbeta))	 \int  \prod_{ i,j  } p( \Gamma_{\r,\c}(\overline{\x})_{ij} )  \, \, \zeta^{\otimes (m-1)\times (n-1)}(d\overline{\x}).
		\end{align}

		Now we justify Cor. \ref{cor:second_transference_density} and give the remaining details for the discrete and the continuous cases. First, when $\zeta$ is the counting measure on $\Z$, then for each $(\r,\c)$ in the support of $\nu$, $\nu_{Y}\{(\r,\c)\}=\P(r(Y)=\r,c(Y)=\c)>0$ so using the expression in \eqref{eq:nu_ybar_pullback} for $\nu_{\overline{\y}}$, 
		\begin{align}
			p_{\overline{\y}}(\r,\c) = \frac{d\nu_{\overline{\y}}}{d\nu_{Y}}(\r,\c) = \frac{\P(\check{Y}= \check{\Gamma}_{\r,\c}(\overline{\y}))}{\P(r(Y)=\r, c(Y)=\c)}.
		\end{align} 
		Then we can choose $q_{\overline{\y}}(\r,\c)=\P(\check{Y}= \check{\Gamma}_{\r,\c}(\overline{\y}))$ and the counting measure $\zeta^{\otimes (m+n-1)}$ is trivially locally absolutely continuous w.r.t. 
		$\nu_{Y}$ on the support of $\nu_{Y}$. This justifies the hypothesis of  Cor. \ref{cor:second_transference_density}. (In fact, the above identity is trivially true in the discrete case without appealing to Cor. \ref{cor:second_transference_density}.) Noting that the supremum of $\P(\check{Y}= \check{\Gamma}_{\r,\c}(\overline{\y}))$ over $\overline{\y}$ is one and using \eqref{eq:p_rc_density_formula2}, 
		\begin{align}\label{eq:discrete_transference_cost}
			\quad \sup_{\overline{\y}} \, p_{\overline{\y}}(\r,\c) =  \frac{1}{\P(Y\in \T(\r,\c))} = 	\exp(-g^{\r,\c}(\balpha,\bbeta))\left(	 \int  \prod_{ i,j  } p( \Gamma_{\r,\c}(\overline{\x})_{ij} )  \, \, \zeta^{\otimes (m-1)\times (n-1)}(d\overline{\x}) \right)^{-1}, 
		\end{align}
		as desired.

		Next, assume $\zeta$ is the Lebesgue measure on $\R$. 
		Note that the support of the law $\overline{\mu}_{\balpha\oplus \bbeta}$ of $\overline{Y}$ is $\supp(p)^{m\times n}$, where $\supp(p)$ is the closure of $\mathcal{C}:=\{p>0\}$. Since $\mathcal{C}$ is an open subset in $\R$, it follows that $\supp(p)\setminus \mathcal{C}$ has Lebesgue measure zero. Hence letting $V=\pi(\{p>0\}^{m\times n})$ where $\pi$ is the margin map, we have $\nu_{Y}(V)=1$. Also, by Prop. \ref{prop:pi_open_map}, $V$ is an open subset of $\R^{m+n-1}$. We will verify that Cor. \ref{cor:second_transference_density}(b) holds.  Indeed, fix $(\r,\c)\in V$. Since $V$ is open in $\R^{m+n-1}$, we can choose an open neighborhood $U_{\r,\c}$ of this margin that is contained in $V$. Now it must be that    $\zeta^{\otimes(m+n-1)} \ll \nu_{Y}$ on $U_{\r,\c}$. If this is not the case,  then there must be some open ball $\mathcal{B}\subseteq U_{\r,\c}$ such that $\nu_{Y}(\mathcal{B})=0$, which must have positive  Lebesgue measure $\zeta^{\otimes(m+n-1)}(\mathcal{B})>0$. However,  $\mathcal{B}\subseteq V$ and by construction $V\subseteq \supp(\nu_{Y})$. Hence the open ball $\mathcal{B}$ contains some margin $(\r',\c')$ in $V$, so  $\nu_{Y}(\mathcal{B})>0$, which is a contradiction. 
	\end{proof}

	We will also use the lower bound on the number of integer-valued contingency tables due to Br\"{a}nd\'{e}n, Leake, and Pak \cite{branden2020lower}, which were obtained by using Lorentzian polynomials by Br\"{a}nd\'{e}n and Huh \cite{branden2020lorentzian}.

	\begin{lemma}\label{lem:counting_measure_Y_lowerbd}
		Let $(\r,\c)$ be a $(m\times n)$ integer-valued margin such that $0<\r(i)/n < \lfloor D \rfloor$ and $0< \c(j)/m < \lfloor D \rfloor$ for all $i,j$ for some $D>1$. Then the following hold: 
		\begin{description}[itemsep=0.1cm]
			\item[(i)] $\T(\r,\c) \cap (\{0,1,\dots,\lceil D\rceil )^{m\times n} $ is non-empty. 
			\item[(ii)] 
			Let $\mu$ be the counting measure on $[0,B]\cap \Z_{\ge 0}$ for any $B\in \{\lceil D \rceil,\lceil D \rceil+1,\dots  \}\cup \{\infty\}$. Let $Y\sim \mu_{\balpha\oplus \bbeta}$ where $(\balpha,\bbeta)$ is an MLE for margin $(\r,\c)$ and let $N:=\sum_{i}\r(i)=\sum_{j}\c(j)$. Then 
			\begin{align}\label{eq:BLP_lower_bd}
				\P\bigl(Y\in \T(\r,\c)\bigr)  \ge N^{-(11/2)(m+n)}. 
			\end{align}
		\end{description}
	\end{lemma}
	
	\begin{proof}
		Note that \textbf{(i)} follows immediately from \textbf{(ii)} with $B=\lceil D \rceil$, so it suffices to show \textbf{(ii)}.
		Existence of MLE $(\balpha,\bbeta)$ for the margin $(\r,\c)$ follows since  the Fisher-Yates table $\left( \r(i)\c(j)/N\right)_{i,j}$ takes entries from the open interval $(0,B)$ and Lemmas \ref{lem:typical} and \ref{lem:strong_dual_MLE_typical}. 
		Recall that the log-likelihood of $Y$ at any matrix in $\T(\r,\c)$ with entries from $[0,B]\cap \Z_{\ge 0}$ w.r.t. the counting measure $\mu^{\otimes (m\times n)}$ is $g^{\r,\c}(\balpha,\bbeta)$. So 
		\begin{align}
			\P\bigl(Y\in \T(\r,\c)\bigr)  =  \exp \left( g^{\r,\c}(\balpha, \bbeta ) \right) \textup{CT}(\r,\c), 
		\end{align}
		where $\textup{CT}(\r,\c)=|\mathcal{T}(\r,\c)\cap \{0,1,\dots,B\}^{m\times n}|$ denotes the number of contingency tables with margin $(\r,\c)$ and entries from $\{0,1,\dots,B\}$. 
		Then the result will follow once the following inequality is verified: 
		\begin{align}\label{eq:counting_lower_bd1}
			\textup{CT}(\r, \c) \ge N^{-(11/2)(m+n)}  \exp \left( -g^{\r,\c}(\balpha, \bbeta) \right). 
		\end{align}

		To deduce \eqref{eq:counting_lower_bd1}, we use the lower bound on $\textup{CT}(\r,\c)$ in \cite[Thm. 2.1]{branden2020lower} in conjunction with the discussion in \cite[Sec. 7.1]{branden2020lower} to get 
		\begin{align}
			\textup{CT}(\r, \c) \ge N^{-(11/2)(m+n)} \textup{Cap}(\r,\c),
		\end{align}
		where 
		\begin{align}
			\textup{Cap}(\r,\c) := \inf_{\x\in (0,1)^{m},\y\in (0,1)^{n}} \,\left( \prod_{i=1}^{m} x_{i}^{-r_{i}} \prod_{j=1}^{n} y_{j}^{-c_{j}}	\prod_{i,j} \frac{1-(x_{i}y_{j})^{B+1}}{1-x_{i}y_{j}} \right).
		\end{align}
		By taking the infimum over the larger set $(\x,\y)\in (0,\infty)^{m}\times (0,\infty)^{n}$ and making change of variables $x_{i}\mapsto e^{\balpha(i)}$ and $y_{j}\mapsto e^{\bbeta(j)}$, 
		\begin{align}
			\log \, \textup{Cap}(\r,\c) &\ge  \inf_{\x,\y}  \left( -\sum_{i=1}^{m} \r(i)\log x_{i} - \sum_{j=1}^{n} \c(y)\log y_{j}  + \sum_{i,j} \log  \frac{1-(x_{i}y_{j})^{B+1}}{1-x_{i}y_{j}}   \right)\\
			&=\inf_{\balpha, \bbeta} \left(  	-\sum_{i=1}^{m} \r(i) \balpha(i)  -\sum_{j=1}^{n} \c(j)\bbeta(j)  + \sum_{i,j} \psi(\balpha(i)+\bbeta(j)) \right)
			=  -g^{\r,\c}(\balpha,\bbeta).
		\end{align}
		This is enough to conclude \eqref{eq:counting_lower_bd1}.
	\end{proof}

	By Cor. \ref{cor:transfer_discrete_Leb} and known results in the combinatorics literature, we can easily deduce Thm. \ref{thm:transfer_counting_Leb}.

	\begin{proof}[\textbf{Proof of Theorem \ref{thm:transfer_counting_Leb} }]

		First assume $\mu$ is the counting measure on $[0,b]\cap Z$ for some $b\in [0,\infty]$. By Lemma \ref{lem:counting_measure_Y_lowerbd}, $\P(Y\in \T(\r,\c))\ge N^{-(11/2)(m+n)}$ with $N=\sum_{i} \r(i)$. Now notice that the upper bound on the transference cost in Cor. \ref{cor:transfer_discrete_Leb} is simply $\P(Y\in \T(\r,\c))$ (see the proof of Cor. \ref{cor:transfer_discrete_Leb}; one can also use Thm. \ref{thm:transference} with $\rho=0$ instead of Cor. \ref{cor:transfer_discrete_Leb}.) 
		
		Next, assume $\zeta$ is the Lebesgue measure on $\R$ and let $h(x)=\mathbf{1}(x> b 0)$. Recently, Barvinok and Rudelson \cite{barvinok2024quick} obtained the following lower bound on the volume of the transportation polytope: 
		\begin{align}
			&  \int  \prod_{ i,j  } h( \Gamma_{\r,\c}(\overline{\x})_{ij} )  \, \, \zeta^{\otimes (m-1)\times (n-1)}(d\overline{\x}) = \textup{Vol}_{\R^{(m-1)\times (n-1)}}\left(\mathcal{T}(\r,\c)\cap \R_{\ge 0}^{m\times n} \right)\\ 
			&\hspace{2cm} \ge (4/e)\sqrt{\pi}  \left( \frac{1}{e\sqrt{2\pi}} \right)^{m+n-1} \frac{1}{\sqrt{m+2}}\, \exp(-H(Z^{\r,\c})) \min_{i,j}\, Z^{\r,\c}_{ij}.  \label{eq:BR_volume}
		\end{align}
		Note that $H(Z^{\r,\c})=g^{\r,\c}(\balpha,\bbeta)$ by Thm. \ref{thm:strong_duality_simple} and $\min_{i,j}Z^{\r,\c}\ge \delta$ by $\delta$-tameness. 
		Combining these results, 
		\begin{align}
			\eqref{eq:p_rc_density_formula2} \quad \ge \quad  C  (4/e)\sqrt{\pi}  \left( \frac{1}{e\sqrt{2\pi}} \right)^{m+n-1} \frac{\delta}{\sqrt{m+2}}. 
		\end{align}
		Thus \eqref{eq:thm_double_transfer3_uniform} follows from Cor. \ref{cor:transfer_discrete_Leb}. 
	\end{proof}
	
	\begin{remark}
		Chatterjee, Diaconis, and Sly obtained a similar result in \cite[Lem. 2.1]{chatterjee2014properties} for constant margins and Lebesgue base measure by relying on (the lower bound of) the volume estimate of the Birkoff polytope (i.e., $\T(\mathbf{1},\mathbf{1})$) due to Canfield and McKay \cite{canfield2007asymptotic}. 
		In our proof of Theorem \ref{thm:transfer_counting_Leb}, we used a recent volume estimate for transportation polytopes with general margins due to Barvionok and Rudelson \cite{barvinok2024quick}.
	\end{remark}

	In the rest of this section, we prove Thm. \ref{thm:second_transference_density_1}. Recall that $\Gamma_{\r,\c}(\overline{Y})$ is the $m\times n$ matrix obtained from the interior matrix $\overline{Y}$ uniquely completing it so that the margin is $(\r,\c)$ (see \eqref{eq:def_completion_map}). In the following lemma, we show that it is enough to lower bound the probability that the entries in the last row and the last column of $\Gamma_{\r,\c}(\overline{Y})$ take values from a compact interval.

	\begin{lemma}\label{lem:second_transference_density}  
		Keep the same setting as in Cor. \ref{cor:transfer_discrete_Leb}. 
		Fix a Borel set $I\subseteq \R$ and assume that  $c_{I}:=\inf_{x\in I}p(x)>0$. Then  \eqref{eq:strong_transference_exact} in Thm. \ref{thm:second_transference} holds with 
		\begin{align}\label{eq:thm_double_transfer22}
			\sup_{\overline{\y}} p_{\overline{\y}}(\r,\c)  \le  (C	c_{I})^{m+n} \, \P\left( \Gamma_{\r,\c}(\overline{Y})_{ij}\in I \,\, \textup{$\forall (i,j)$ s.t. $i=m$ or $j=n$ }   \right)^{-1}
		\end{align}
		for some constant $C=C(\mu,\delta)>0$.
	\end{lemma}

	\begin{proof} 
		Let $T= \prod_{\textup{$i=m$ or $j=n$}} 
		p^{ij}(\Gamma_{\r,\c}(\overline{Y})_{ij})$ be a nonnegative random variable where $p^{ij}$s are the Radon-Nikodym derivatives of $\mu_{\balpha(i)+\bbeta(j)}$ w.r.t. $\zeta$. Note that $T$ equals  $q_{\overline{\y}}(\r,\c)=d\nu_{\overline{\y}}/d\zeta^{\otimes(m+n-1)}$ in \eqref{eq:q_y_density} evaluated at $\overline{Y}$. Let $\mathcal{E}$ denote the event in the probability in \eqref{eq:thm_double_transfer22}. Denote 
		\begin{align}
			C  = \inf_{\phi(A_{\delta})\le \theta \le \phi(B_{\delta})} \,  \inf_{x\in I} e^{x\theta-\psi(\theta)} \in (0,\infty).
		\end{align}
		Then on the event $\mathcal{E}$, 
		\begin{align}
			T  \ge     \prod_{\textup{$i=m$ or $j=n$}} 
			\inf_{x\in I} \, p^{ij}(x) \ge  (Cc_{I})^{m+n}. 
		\end{align}
		It follows that by Cor. \ref{cor:transfer_discrete_Leb}, 
		\begin{align}
			\sup_{\overline{\y}} p_{\overline{\y}}(\r,\c) \le M^{m+n}  \, \E[T]^{-1} \le (MC c_{I})^{m+n} \P(\mathcal{E})^{-1},
		\end{align}
		where $M=M(\mu,\delta)>0$  is a constant defined in the statement of Cor. \ref{cor:transfer_discrete_Leb}. 
	\end{proof}

		\begin{proof}[\textbf{Proof of Theorem \ref{thm:second_transference_density_1}}]
		
		Let $\mathcal{E}(I)$ denote the event in \eqref{eq:thm_double_transfer22}. By Lemma \ref{lem:second_transference_density}, it suffices to lower bound the probability of $\mathcal{E}(I)$ for some compact interval $I$ in the support of $\mu$. Our argument is based on partitioning $Y$ into a $2\times 2$ block matrix where the size of the $(2,2)$ block is about $\sqrt{m}\times \sqrt{n}$. By concentration inequalities and union bounds, the largest block $Y^{11}$ will have its row and column sums close to its expectations with probability bounded away from zero. On this event, we will show that there is a sufficient probability to assign values of the entries in the remaining three blocks \textit{except} the ones in the last row and the law column in a way that the required values to complete $\overline{Y}$  to $\Gamma_{\r,\c}(\overline{Y})$ all lie in a compact interval in the support of $\mu$. 
		
		In the rest of the proof, we will use $C=C(\mu,\delta)>0$ to denote a positive constant depending only on $\mu$ and $\delta$ whose value can change from line to line.

		Denote $m_{0}:=\lfloor \sqrt{m}\log m\rfloor $, $n_{0}:=\lfloor \sqrt{n}\log n\rfloor $, $m':=m-m_{0}$ and $n':=n-m_{0}$. 
		Partition the row indices into  sets $\{1,\dots, m'\}$ and $\{m'+1,\dots,m\}$. Make a similar partitioning for the column indices. 
		Also, we have  the following $2\times 2$ block partitioning of the $m\times n$ random matrix $Y$: 
		\begin{align}
			Y = \begin{bmatrix}
				Y^{11} & Y^{12}  \\
				Y^{21} & Y^{22}  \\
			\end{bmatrix}
		\end{align}
		where $Y^{22}$ is $m_{0}\times n_{0}$. 
		The expected margin of $Y$ is exactly $(\r,\c)$. Partition the typical table $Z$ similarly.
		Denote $Z^{k\ell}$ for $1\le k,\ell\le 2$ for the block entries of the corresponding partitioning of the $Z$. Note that $\E[Y]=Z$ so $\E[Y^{k\ell}]=Z^{k\ell}$ for all $k,\ell\in \{1,2\}$. Note that each $Y_{ij}\sim \mu_{\balpha(i)+\bbeta(j)}$ with $\phi(A_{\delta})\le \balpha(i)+\bbeta(j)\le \phi(B_{\delta})$ by $\delta$-tameness of $(\r,\c)$. Hence $\{Y_{ij}\}'$s have a uniformly bounded sub-exponential norm, so by Bernstein's inequality for sums of independent sub-exponential random variables (see, e.g., \cite[Thm. 2.8.1]{vershynin2018high}), 
		\begin{align}
			\P\left( \left| Y^{11}_{\bullet\bullet} - Z^{11}_{\bullet\bullet} \right| \le  c \sqrt{mn \log mn} \right) &\ge 1- 2 \exp\left( - 2\gamma \log m n\right), \\
			\P\left( \left| Y^{11}_{i\bullet} - Z^{11}_{i\bullet} \right| \le c \sqrt{n\log n} \right) & \ge 1- 2 \exp\left( - 2\gamma \log n \right) \quad \textup{for $1\le i \le m'$}, \\
			\P\left( \left| Y^{11}_{\bullet j} - Z^{11}_{\bullet j} \right| \le c \sqrt{m\log m} \right) &\ge 1- 2 \exp\left( - 2\gamma \log m \right) \quad \textup{for $1\le j \le n'$}
		\end{align}
		for some constant $c=c(\mu,\delta)>0$. Let $\mathcal{D}^{11}$ denote the intersection of all events that appear on the left-hand side of the above inequalities.
		By a union bound, 
		\begin{align}\label{eq:Y11_prob_bd}
			\P(\mathcal{D}^{11}) \ge 
			1 - 6(m\lor n) (m\land n)^{-2\gamma} \ge 1-6(m\land n)^{-\gamma}.
		\end{align}
		By the hypothesis  $m,n\ge 1$ are sufficiently large so that the last expression above is at least $1/2$. Then $\P(\mathcal{D}^{11})\ge 1/2$.  Hence by Lemma \ref{lem:second_transference_density}, it suffices to show 
		\begin{align}\label{eq:E_I_claim}
			\P(\mathcal{E}(I)\,|\, \mathcal{D}^{11}) \ge 	\exp( -C(m \sqrt{n}\log n+ n \sqrt{m}\log m )\log mn ), 
		\end{align}
		where  $I\subseteq \R$ is a Borel set such that $c_{I}=\inf_{x\in I}p(x)>0$.

		We first argue for the discrete case when 
		$\zeta$ is the counting measure on $\Z$ and $\mu$ has probability density $p$ w.r.t. $\zeta$. By the hypothesis, $\{p>0\}=[A,B]\cap \Z$. By translating $\mu$, without loss of generality we assume $\lfloor A_{\delta/2} \rfloor =0$. Then $p(0)>0$. We also have $1, \lceil B_{\delta/2} \rceil\in \{ p>0 \}$ since $p$ cannot be concentrated at a single value (otherwise $\nu$ is not $\sigma$-finite). Take $I:=[0, \lceil B_{\delta/2} \rceil]\cap \Z$ and write	 $\r=(\r^{1},\r^{2})$ and $\c=(\c^{1},\c^{2})$ where $\r^{2}\in \R^{m_{0}}$ and $\c^{2}\in \R^{n_{0}}$. Define events 
		\begin{align}
			\mathcal{D}^{22} &:= \left\{ Y^{22}_{\bullet\bullet}=Z^{22}_{\bullet\bullet}+Y^{11}_{\bullet\bullet}-Z^{11}_{\bullet\bullet},\, Y^{22} \in  I^{m_{0}\times n_{0}} \right\}, \\
			\mathcal{D}^{12} &:= \left\{ \| \pi(Y^{12}) - (\r^{1}-r(Y^{11}), \c^{2}-c(Y^{22}) ) \|_{\infty}=0, \, Y^{12} \in I^{m'\times n_{0}} \right\}, \\
			\mathcal{D}^{21} &:= \left\{ \| \pi(Y^{21}) - (\r^{2}-r(Y^{22}), \c^{1}-c(Y^{11}) ) \|_{\infty}=0  ,\,  Y^{21} \in I^{m_{0}\times n'}   \right\}, 
		\end{align}
		where $\pi$ is the map that sends a matrix $\x$ to its margin $(r(\x),c(\x))$. We claim 
		\begin{align}\label{eq:D22_D12_D21_claim}
			\P(\mathcal{D}^{22}\,|\, \mathcal{D}^{11}) &\ge \exp(-Cm_{0}n_{0}), \\
			\P(\mathcal{D}^{12}\,|\, \mathcal{D}^{22}\cap \mathcal{D}^{11}) &\ge \exp(-Cm n_{0}), \\ \nonumber
			\P(\mathcal{D}^{21}\,|\, \mathcal{D}^{22}\cap \mathcal{D}^{11}) &\ge \exp(-Cm_{0} n). \nonumber
		\end{align}
		Since $\mathcal{D}^{22}\cap \mathcal{D}^{12}\cap \mathcal{D}^{21}$ implies $\mathcal{E}(I)$ and since $\mathcal{D}^{12}$, $\mathcal{D}^{21}$ are conditionally independent on $\mathcal{D}^{11}\cap\mathcal{D}^{22}$, this would be enough to conclude \eqref{eq:E_I_claim}. 
		
		To justifty the first inequality in \eqref{eq:D22_D12_D21_claim}, recall that since $(\r,\c)$ is $\delta$-tame, 
		$Z^{22}_{\bullet\bullet} \in [A_{\delta} m_{0}n_{0}, B_{\delta} m_{0}n_{0}]$. Since  $|Y^{11}_{\bullet\bullet}-Z^{11}_{\bullet\bullet}|\le c \sqrt{mn\log mn}\ll  m_{0} n_{0}$ on $\mathcal{D}^{11}$, 
		\begin{align}\label{eq:Y22_range}
			A_{\delta/2} m_{0}n_{0}  \le	Z^{22}_{\bullet\bullet} + Y^{11}_{\bullet\bullet} - Z^{11}_{\bullet\bullet} < B_{\delta/2} m_{0}n_{0} \quad \textup{on $\mathcal{D}^{11}$} 
		\end{align}
		for $m,n$ large enough depending only on $\delta$. The term in the middle is an integer since $Z^{22}_{\bullet\bullet}-Z^{11}_{\bullet\bullet} = \sum_{n'<j\le n} \c(j) - \sum_{1\le i \le m'} \r(i)$. Note that since $\mu$ assigns a positive probability for all values in the support, so does its exponential tilt with finite tilting parameter. Hence, there exists a matrix $\y^{22}\in I^{m_{0}\times n_{0}}$ such that $\y^{22}_{\bullet\bullet}=Z^{22}_{\bullet\bullet}+Y^{11}_{\bullet\bullet}-Z^{11}_{\bullet\bullet}$ and $\P(Y^{22}=\y^{22})>0$. Since the $m_{0}n_{0}$ entries of $Y^{22}$ are independent, it follows that 
		\begin{align}
			\P(\mathcal{D}^{22}\,|\, \mathcal{D}^{11}) \ge \P(Y^{22}=\y^{22}\,|\, \mathcal{D}^{11}) = \P(Y^{22}=\y^{22}) \ge \exp(-Cm_{0}n_{0}). 
		\end{align}

		Next, consider the second inequality in \eqref{eq:D22_D12_D21_claim}. The event $\mathcal{D}^{12}$ 
		requires the row and column sums of $Y^{12}$ to take the prescribed values so that the first $m'$ row sums and the last $n_{0}$ columns sums of $Y$ match the corresponding values in the target margin $(\r,\c)$. Denote the required row and column sums for $Y^{12}$ as $\r^{12}:=\r^{1}-r(Y^{11})$ and $\c^{12}:=\c^{2}-c(Y^{22})$, respectively. On $\mathcal{D}^{22}$, they have the same total sum (hence $(\r^{12}, \c^{12})$ is a margin for $Y^{12}$)   since 
		\begin{align}
			\sum_{i} \r^{12}(i)	=	\sum_{1\le i \le m'} \r(i) - Y^{11}_{\bullet\bullet} = 	Z^{11}_{\bullet\bullet} + Z^{12}_{\bullet\bullet} - Y^{11}_{\bullet\bullet},  \\
			\sum_{j} \c^{12}(j)	= \sum_{n'< j \le n} \c(j) - Y^{22}_{\bullet\bullet}  = Z^{12}_{\bullet\bullet} + Z^{22}_{\bullet\bullet} - Y^{22}_{\bullet\bullet}
		\end{align}
		and the last expressions in each line above coincide on $\mathcal{D}^{22}$. Proceeding as before, it remains to show that there is at least one realization of $Y^{12}$ in $I^{m'\times n_{0}}$ with margin $(\r^{12},\c^{12})$. 
		According to Lemma \ref{lem:counting_measure_Y_lowerbd} \textbf{(i)}, we only need to verify that 
		\begin{align}
			A_{\delta/2} n_{0}	
			&< \r(i) - Y^{11}_{i\bullet } < B_{\delta/2} n_{0} \quad \textup{for all $1\le i\le m'$} \quad \textup{and}\\
			A_{\delta/2} m_{0}	 
			& < \c(j) - Y^{22}_{\bullet j} < B_{\delta/2} m_{0}  \quad \textup{for all $n'< j\le n$}.
		\end{align}
		Writing $\r(i) = Z^{11}_{i\bullet} + Z^{12}_{i\bullet}$ for $1\le i \le m'$ and $\c(j)=Z^{12}_{\bullet j}+Z^{22}_{\bullet_{j}}$ for $n'<j\le n$, and noting that $Z\in [A_{\delta}, B_{\delta}]^{m\times n}$, 
		one can see that the above inequalities hold almost surely for all sufficiently large $m,n$ (depending only on $c,\delta$) on the event $\mathcal{D}^{11}$. A symmetric argument applies for $Y^{21}$. This shows \eqref{eq:D22_D12_D21_claim}, completing the proof for the discrete case.

		Lastly, we argue for the continuous case with $I=[A_{\delta/4}, B_{\delta/4}]$. 
		By a discretization and perturbation argument, we will show that \eqref{eq:D22_D12_D21_claim} still holds with some negligible error. Assume $\zeta=\textup{Leb}(\R)$ and $\supp(p)=[A,B]\cap \R$. 
		Consider the rescaled integer grid $\eps \Z$ in $\R$ for $\eps=O(\delta/m^{2}n^{2})$. For each $i,j=1,2$,  let $[Y^{ij}]$ denote the random matrix whose entries are the nearest point in $\eps\Z$ to the corresponding entries in $Y^{ij}$. Since the support of $p$ is an interval, the support of each entry of $[Y^{ij}]$ is a set of consecutive points in $\eps\Z$. By the same argument for $Y^{22}$ in the discrete case, we have 
		\begin{align}\label{eq:pf_Y22_conti}
			\P\left( \left|  [Y^{22}]_{\bullet\bullet}- \left( Z^{22}+Y^{11}_{\bullet\bullet}-Z^{11}_{\bullet\bullet}\right) \right| \le \eps,\,  [Y^{22}]\in J_{\delta}\cap \eps\Z  \,\bigg|\, \mathcal{D}^{11} \right) \ge \exp(-Cm_{0}n_{0}),
		\end{align}
		where $J_{\delta} = [\eps \lfloor \eps^{-1} A_{\delta/2} \rfloor, \eps \lceil \eps^{-1} B_{\delta/2} \rceil]$. 
		Note that for each matrix $\y^{22}\in I^{m_{0}\times n_{0}}$, $[Y^{22}]=\y^{22}$ with probability at least $\Omega(\eps^{m_{0}n_{0}}) \ge \exp(-Cm_{0}n_{0}\log mn)$. Hence denoting $J_{\delta}\pm \eps=\{ x\pm \eps\,:\, x\in J_{\delta} \}$, 
		\begin{align}\label{eq:pf_Y22_conti2}
			\P\left( \left|  Y^{22}_{\bullet\bullet}- \left( Z^{22}+Y^{11}_{\bullet\bullet}-Z^{11}_{\bullet\bullet}\right) \right| \le (m_{0}n_{0}+1)\eps, \,\,  Y^{22} \in (J_{\delta}\pm \eps)^{m_{0}\times n_{0}} \,\bigg|\, \mathcal{D}^{11} \right) \ge  \exp(-C m_{0}n_{0}\log mn ).  
		\end{align}
		Let $\mathcal{D}^{22}_{\eps}$ denote the event in \eqref{eq:pf_Y22_conti} and define 
		\begin{align}
			\mathcal{D}^{12}_{\eps} &:= \left\{ \| \pi(Y^{12}) - (\r^{1}-r(Y^{11}), \c^{2}-c(Y^{22}) ) \|_{\infty}\le (m+n)\eps, \, Y^{12} \in (J_{\delta}\pm \eps)^{m'\times n_{0}}  \right\}, \\
			\mathcal{D}^{21}_{\eps} &:= \left\{ \| \pi(Y^{21}) - (\r^{2}-r(Y^{22}), \c^{1}-c(Y^{11}) ) \|_{\infty}\le (m+n)\eps  ,\,  Y^{21} \in (J_{\delta}\pm \eps)^{m_{0}\times n'} \right\}.
		\end{align}
		Similarly, we can show 
		\begin{align}\label{eq:pf_Y12_Y21_conti}
			\P(\mathcal{D}^{12}_{\eps}\,|\, \mathcal{D}^{11}\cap \mathcal{D}^{22}_{\eps}  ) \ge \exp(-C m n_{0} \log mn ), \\
			\P(\mathcal{D}^{21}_{\eps}\,|\, \mathcal{D}^{11}\cap \mathcal{D}^{22}_{\eps}  ) \ge \exp(-C m_{0} n \log mn ). 
		\end{align}
		Combining the above, we deduce 
		\begin{align}
			\P\left( \Gamma_{\r,\c}(\overline{Y})_{ij}\in J_{\delta} \pm (m+n+1)\eps \, \,\, \textup{$\forall (i,j)$ s.t. $i=m$ or $j=n$ }  \,\bigg|\, \mathcal{D}^{11}  \right) \ge \exp(-C (m n_{0}+m_{0}n )\log mn ),
		\end{align}
		Hence for $m,n$ large enough dependong only on $\delta$ so that $J_{\delta}\pm (m+n+1)\eps\subseteq [A_{\delta/4}, B_{\delta/4}]=I$, 
		we verify \eqref{eq:E_I_claim}. 
	\end{proof}

	\begin{remark}[Combinatorial applications]\label{rmk:volume_application}
		Our probabilistic argument above gives some interesting combinatorial results. In the proof of Thm. \ref{thm:second_transference_density_1} we have bounded above the transference cost \eqref{eq:strong_transference_exact2}  with $q_{\overline{\y}}(\r,\c)$ given by \eqref{eq:q_y_density}. Then using \eqref{eq:p_rc_density_formula2}, we obtain 
		\begin{align}
			&	\int  	\prod_{ i,j  } h( \Gamma_{\r,\c}(\overline{\x})_{ij} )  \, \, \zeta^{\otimes (m-1)\times (n-1)}(d\overline{\x}) \\ &\qquad \ge 	\frac{\exp(-g^{\r,\c}(\balpha,\bbeta)) }{ \sup_{\y}  q_{\overline{\y}}(\r,\c) } 	 \exp( -C(m \sqrt{n}\log n+ n \sqrt{m}\log m )\log mn ). 
		\end{align}
		The integral on the left-hand side above is a $h$-weighted volume of the transportation polytope $\T(\r,\c)$ in units of basic cells in $\R^{(m-1)\times (n-1)}$. When $p(x)=\mathbf{1}(x\ge 0)$ so that $\mu$ is the Lebesgue measure on $\R_{\ge 0}$, it gives a lower bound on the volume of the nonnegative transportation polytope (in this case $\sup_{\y}  q_{\overline{\y}}(\r,\c)=1$), which may be compared to the result \eqref{eq:BR_volume} of Barvinok and Rudelson \cite{barvinok2024quick}. Our lower bound is weaker than their result but ours applies for more general non-constant densities $p$. 
		
		In the discrete case, the integral above becomes the weighted count $	\sum_{\x \in \T(\r,\c)\cap \Z^{m\times n}} \prod_{ij} p(\x_{ij})$ of the inter points in the transportation polytope $\T(\r,\c)$ and $\sup_{\y}  q_{\overline{\y}}(\r,\c)=1$. 
		Setting $p(x) = \mathbf{1}(x\ge 0)$, it gives a lower bound on the number of contingency tables with margin $(\r,\c)$. This gives a weaker lower bound than the known result  \eqref{eq:counting_lower_bd1} of Br\"{a}nd\'{e}n, Leake, and Pak \cite{branden2020lower} but ours also applies for non-constant densities $p$. For the discrete case, we essentially lower bounded $\P\bigl(Y\in \T(\r,\c)\bigr)$, which equals  $\exp \left( g^{\r,\c}(\balpha, \bbeta ) \right) \mu^{\otimes(m\times n)}(\T(\r,\c))$ noting that  $g^{\r,\c}(\balpha,\bbeta)$ is the log-likelihood of $Y$ taking any matrix in $\T(\r,\c)$ w.r.t. $\mu^{\otimes (m\times n)}$. This is why our upper bound on the transference cost implies a lower bound on the corresponding weighted volume of the polytope. It would be interesting to know if the above results can be obtained by a purely combinatorial argument. 
	\end{remark}

	\section{Proof of scaling limit of the typical tables and MLEs} 
	\label{sec:pf_scaling_limit}

	Our goal in this section is to prove Theorem \ref{thm:main_convergence_conti}. This will take several steps.

	\subsection{Lipschitz continuity of typical tables and standard MLEs}

	In this section, we establish a stability result (Thm. \ref{thm:typical_lipschitz_margins} below) that bounds the perturbation on the typical tables in terms of the perturbation of the margin. This result is central to our limit theory and it will be extended to the continuum setting (Thm. \ref{thm:typical_kernel_Lipscthiz}), which will be used critically in proving Theorem \ref{thm:main_convergence_conti}.

	\begin{theorem}[Lipschitz continuity of typical table w.r.t. margin]\label{thm:typical_lipschitz_margins}
		Let $C_\delta=\max\{|\phi(A_\delta)|,\, |\phi(B_\delta)| \}$ be as in Lemma \ref{lem:strong_dual_MLE_typical}. There  for each $m\times n$ $\delta$-tame margins  $(\r,\c)$ and $(\r',\c')$, we have
		\begin{align}\label{eq:typical_margin_lipschiz_continuity}
			\lVert Z^{\r,\c}-Z^{\r',\c'} \rVert_{F}^{2} \le 4C_{\delta}\left( \sup_{ \phi(A_{\delta})\le w \le \phi(B_{\delta}) } \psi''(w) \right) \lVert (\r,\c)-(\r',\c') \rVert_{1}.
		\end{align}
	\end{theorem}
	
	Our key insight behind the proof of the above result is to \textit{use the dual variable (MLE) to bridge the margin and the corresponding typical table.} Namely, while it is difficult to directly construct a map from a margin to the corresponding typical table, being the solution of a constrained optimization problem \eqref{eq:typical_table_opt}, it is easy to write down the margin and the typical table corresponding to a given MLE, by using Lemma \ref{lem:strong_dual_MLE_typical}. A schematic for this approach is given in the diagram below:
	\begin{align}
		&\hspace{1cm}	\xymatrix{
			\textit{\hspace{0.5cm} margin}	&&&  \textit{\hspace{2.3cm}MLE} && \textit{\hspace{1.5cm} typical table}
		}\\
		&	\xymatrix{
			(\r(1),\dots, \r(m),\c(1),\dots,\c(n))	&& \ar[ll]_{\Lambda}  (\balpha(1),\dots, \balpha(m),\bbeta(1),\dots,\bbeta(n)) \ar[rr]^{\hspace{1.8cm}\Phi}&& \VEC(Z) 
		}
	\end{align}
	Here for each matrix $X\in \R^{m\times n}$, let $\VEC(X)\in \R^{mn}$ denotes is vectorization, which is obtained by stacking the $j$th column underneath its $(j-1)$st column for $j=2,\dots,n$. 
	
	Denote $q:=m+n$ and $p:=mn$. 
	Define a map $\Lambda:\R^{q}\rightarrow \R^{q}$, $(\balpha,\bbeta)\mapsto (\r,\c)$ where 
	\begin{align}\label{eq:def_MLE2margin_map}
		\begin{cases}
			\r(i) = \sum_{j=1}^{n} \psi'(\balpha(i)+\bbeta(j)) \quad \textup{for all $i=1,\dots,m$}, \\ 
			\c(j) = \sum_{i=1}^{m} \psi'(\balpha(i)+\bbeta(j)) \quad \textup{for all $j=1,\dots,n$}.
		\end{cases}
	\end{align}
	Also define a map $\Phi:\R^{m+n}\rightarrow \R^{m\times n}$, $(\balpha,\bbeta)\mapsto Z=(z_{ij})$, where for each $(i,j)\in [m]\times [n]$, 
	\begin{align}
		z_{ij} := \psi'(\balpha(i)+\bbeta(j)). 
	\end{align}
	Both maps are well-defined and are differentiable. It is important to understand the structure of the Jacobian matrices of these maps: 
	\begin{align}\label{eq:Jacobian_MLE_margin}
		J_{\Lambda}(\balpha,\bbeta) &:= 	\left[ 	\frac{\partial (\r(1),\dots,\r(m),\c(1),\dots,\c(n)) }{\partial (\balpha(1),\dots,\balpha(m),\bbeta(1),\dots,\bbeta(n))}  \right]_{ q\times q}, \\
		J_{\Phi}(\balpha,\bbeta) &:= 	\left[ 	\frac{\partial \VEC(\Phi(\balpha,\bbeta))^{\top} }{\partial (\balpha(1),\dots,\balpha(m),\bbeta(1),\dots,\bbeta(n))}  \right]_{ p\times q}.
	\end{align}
	These matrices turn out to admit precise factorization into some important matrices, which we introduce below. For each dual variables $(\balpha,\bbeta), (\balpha',\bbeta')\in \R^{m}\times \R^{n}$, define matrices 
	\begin{align}\label{eq:def_jacobian_matrices_P}
		P^{\balpha,\bbeta}&:= \diag\left( \psi''(\balpha(1)+\bbeta(1)), \dots, \psi''(\balpha(1)+\bbeta(n)),\dots,  \psi''(\balpha(m)+\bbeta(n)) \right) \in \R^{p\times p} \\
		P^{\balpha,\bbeta; \balpha',\bbeta'} &:= \int_{0}^{1}P^{(1-t)(\balpha,\bbeta) + t(\balpha',\bbeta')  }\,dt. 
	\end{align}
	Let $h_{i}(X)$ and $h^{j}(X)$ denote the $i$th row sum and $j$th columns sum of $X$, respectively. Then $\nabla h_{i}(X)$ is the $m\times n$ matrix where the $i$th row is filled with 1s and 0s elsewhere. Similarly, $\nabla h^{j}(X)$ is the $m\times n$ matrix with $1$s on the $j$th column and $0$s elsewhere. Then define a matrix 
	\begin{align}\label{eq:def_jacobian_matrices_Q}
		Q&:= 
		\begin{bmatrix}
			\VEC(\nabla h_{1}) & \dots & \VEC(\nabla h_{m}) & \VEC(\nabla h^{1}) &\dots \VEC(\nabla h^{n}) 
		\end{bmatrix}
		\in \R^{p\times q}. 
	\end{align}

	\begin{prop}\label{prop:Jacobian_matrix_formula}
		\qquad $\displaystyle J_{\Lambda_{0}}(\balpha,\bbeta) = Q^{\top} P^{\balpha,\bbeta} Q, \qquad J_{\Phi_{0}}( \balpha,\bbeta) = P^{\balpha,\bbeta} Q.$
	\end{prop}
	
	\begin{proof}
		To see the first formula, a useful computation is the following. Fix a matrix $E=(E_{ij})_{i,j}\in \R^{m\times n}$. 
		We claim that 
		\begin{align}\label{eq:jacobian}
			Q^{\top}  \diag( \VEC(E)) Q
			= 	\left[ 
			\begin{array}{c c c | c  c c c c c }
				E_{1\bullet} &   & & E_{11} & \dots & E_{1n} \\
				& \ddots &  & & \ddots & \\
				& & E_{m\bullet} & E_{m1} & \dots & E_{mn} \\ 
				\hline
				E_{11} & \dots & E_{m1} &  E_{\bullet 1} & &  \\
				& \ddots & &  &\ddots &  \\
				E_{1n} & \dots & E_{mn} &  & & E_{\bullet n}
			\end{array}
			\right],
		\end{align}
		where $E_{i\bullet}$ and $E_{\bullet j}$ denotes the $i$th row sum and the $j$th column sum of $E$, respectively. To see this, note that the columns of $Q$ are the indicators of entries in the corresponding row/column sum for a $m\times n$ matrix.  Hence, the $(1,1)$ entry in the left-hand side of \eqref{eq:jacobian} is the sum of all entries in $E$ that appear in the first row sum, which is $E_{1\bullet}$; 
		Its $(1,2)$ entry is zero since there is no entry of $E$ shared in the first and the second row sums; Its $(1,m)$ entry is the sum of all entries in $E$ appearing in the first row sum and the first column sum, which is $E_{11}$.

		Now one can directly compute the Jacobian $J_{\Lambda;\balpha,\bbeta}$ and it has the $2\times 2$  block matrix form in the right-hand side above with $E_{ij}=\psi''(\balpha(i)+\bbeta(j))$. Hence the first formula in the assertion follows from the identity \eqref{eq:jacobian}. The second formula can be verified by a straightforward computation.
	\end{proof}

	Suppose we have two $\delta$-tame margins $(\r,\c)$ and $(\r',\c')$. Can we find a canonical path within $\mathcal{M}^{\delta}$ that interpolates between these two $\delta$-tame margins? A natural candidate would be the linear interpolation between the two margins. However, it is not necessarily true that a convex combination of two $\delta$-tame margins is again $\delta$-tame. It turns out that the right way is to linearly interpolate between the corresponding MLEs in the dual space and map it back to the space of tables and margins. 
	More precisely, let  $(\balpha,\bbeta)$ and $(\balpha',\bbeta')$ be any MLEs for margins $(\r,\c)$ and $(\r',\c')$, respectively, which exist by Lemma \ref{lem:strong_dual_MLE_typical}. For each $\lambda\in [0,1]$, let  $(\balpha_{\lambda},\bbeta_{\lambda}):=(1-\lambda) (\balpha,\bbeta) + \lambda (\balpha',\bbeta')$ be the convex combination of the MLEs. Define a margin 
	\begin{align}\label{eq:def_gamma_interpolating_curve}
		\gamma(\lambda):= \textup{the margin satisfied by $Z_{\lambda}$}, 
	\end{align}
	where the $m\times n$ table  $Z_{\lambda}=((z_{\lambda})_{ij})$  is defined by  
	\begin{align}\label{eq:def_gamma_interpolating_curve_typical}
		(z_{\lambda})_{ij}  := \psi'\left( \balpha_{\lambda}(i) + \bbeta_{\lambda}(j) \right)  \quad \textup{for all $i,j$}. 
	\end{align}
	Notice that for any $\lambda\in [0,1]$, the interpolated dual variable $(\balpha_{\lambda},\bbeta_{\lambda})$ is an MLE  for the margin $\gamma(\lambda)$ due to Lemma \ref{lem:typical_MLE_equation}.

	\begin{prop}[$C^{1}$-connectedness of $\mathcal{M}^{\delta}$]\label{prop:tame_C1_connected}
		Keep the same setting as before. Then $\gamma$ in \eqref{eq:def_gamma_interpolating_curve} defines a $C^{1}$-path $\gamma$ from $[0,1]$ into the set of all $\delta$-tame $m\times n$ margins  such that $\gamma(0)=(\r,\c)$, $\gamma(1)=(\r',\c')$. 
	\end{prop}
	
	\begin{proof}
		
		Denote $Z=(z_{ij})=Z^{\r,\c}$ and $Z'=(z'_{ij})=Z^{\r',\c'}$. By Lemma \ref{lem:strong_dual_MLE_typical}, there exists a MLEs $(\balpha,\bbeta)$ and $(\balpha',\bbeta')$ for margins $(\r,\c)$ and $(\r',\c')$, respectively,  such that for all $i,j$, 
		\begin{align}
			\phi(z_{ij}) = \balpha(i)+\bbeta(j) \quad \textup{and} \quad \phi(z_{ij})=\balpha'(i)+\bbeta(j)'. 
		\end{align}
		Since $Z$ and $Z'$ are $\delta$-tame, 
		\begin{align}
			&	\phi(A_{\delta}) 	\le  \balpha(i) + \bbeta(j) \le \phi(B_{\delta}), \qquad 	\phi(A_{\delta}) 	\le \balpha'(i) + \bbeta'(j)   \le \phi(B_{\delta}).
		\end{align}
		Let $\gamma(\lambda):=(\r_{\lambda}, \c_{\lambda})$ denote the margin satisfied by $Z_{\lambda}$. Note that 
		\begin{align}\label{eq:dual_lin_interpolation_bd}
			\balpha_{\lambda}(i)+ \bbeta_{\lambda}(j)   &= (1-\lambda) (\balpha(i)+\bbeta(j)) + \lambda(\balpha'(i)+\bbeta'(j)) \in [\phi(A_{\delta}), \phi(B_{\delta})] \quad \textup{for all $i,j$ and $\lambda\in [0,1]$}. 
		\end{align}
		Hence $Z_{\lambda}$ is $\delta$-tame for all $0\le \lambda \le 1$. Therefore, $\gamma$ defines a continuous curve in $\mathcal{M}^{\delta}$ that connects the two margins $(\r,\c)$ and $(\r',\c')$. 
	\end{proof}

	We can write the change in margin and the typical table as we move along the secant line between the two MLEs by a line integral.  Using Proposition \ref{prop:Jacobian_matrix_formula} along the way, we derive the following key lemma for the proof of Theorem \ref{thm:typical_lipschitz_margins}.

	\begin{lemma}\label{lem:key_change_margin_Z_factorization}
		Keep the same setting as before. Then denoting $v := \VEC(\balpha'-\balpha,\bbeta'-\bbeta) \in \R^{q}$, we have 
		\begin{align}
			(\r',\c')- (\r,\c)	
			&=Q^{\top} P^{\balpha,\bbeta;\balpha',\bbeta'} Q v\quad \textup{and} \quad \VEC(Z') - \VEC(Z) = P^{\balpha,\bbeta;\balpha',\bbeta'}  Q v.
		\end{align}
	\end{lemma}
	
	\begin{proof}
		We first use the chain rule to write 
		\begin{align}
			\frac{d}{dt}\gamma(t) = J_{\Lambda}(\balpha_{t},\bbeta_{t}) v.
		\end{align}
		By Proposition \ref{prop:Jacobian_matrix_formula}, we can write	 	\begin{align}\label{eq:margin_MLE_integral_eq}
			\gamma(1)-\gamma(0) 
			&= \int_{0}^{1} J_{\Lambda}(\balpha_{t},\bbeta_{t}) v \,dt  = Q^{\top}\left[ \int_{0}^{1} P^{\balpha_{t},\bbeta_{t}} \,dt \right] Q v = Q^{\top} P^{\balpha,\bbeta;\balpha',\bbeta'} Q v,
		\end{align}
		where the integral in the rightmost expression above is done entry-wise. Similarly, using Proposition \ref{prop:Jacobian_matrix_formula}, we get 
		\begin{align}
			\Phi(\balpha',\bbeta') - \Phi(\balpha',\bbeta') &= \int_{0}^{1} J_{\Phi}(\balpha_{t},\bbeta_{t}) v \,dt = \left[ \int_{0}^{1} P^{\balpha_{t},\bbeta_{t}} \,dt \right] Q v. 
		\end{align}
		This shows the assertion. 
	\end{proof}

	Now we are ready to show Theorem \ref{thm:typical_lipschitz_margins}.

	\begin{proof}[\textbf{Proof of Theorem} \ref{thm:typical_lipschitz_margins}]
		Let $(\balpha,\bbeta):=(\balpha^{\r,\c}, \bbeta^{\r,\c})$ and $(\balpha',\bbeta'):=(\balpha^{\r',\c'}, \bbeta^{\r',\c'})$ denote the standard MLEs for margins $(\r,\c)$ and $(\r',\c')$, respectively.  Write $\overline{P}:=P^{\balpha,\bbeta;\balpha',\bbeta'}$ and $R:=\overline{P}^{1/2}Q\in \R^{p\times q}$. Then by Lemma \ref{lem:key_change_margin_Z_factorization}, we have the following relations 
		\begin{align}
			\VEC(Z') - \VEC(Z) = \overline{P}^{1/2} Rv,\quad 	\gamma(1)-\gamma(0) = R^{\top}Rv, 
		\end{align}
		where $v$ is as in Lemma \ref{lem:key_change_margin_Z_factorization}.
		On the one hand, by Proposition \ref{prop:tame_C1_connected}, $A_{\delta}\le \psi'(\balpha_{t}(i)+\bbeta_{t}(j))\le B_{\delta}$ for all $t\in [0,1]$ and $i,j$. Hence $\lVert \overline{P}\rVert_{2}\le \sup_{\phi(A_{\delta})\le w\le \phi(B_{\delta})} \psi''(w)$, so we get 
		\begin{align}\label{eq:Z_Z_upper}
			\lVert Z-Z' \rVert_{F} = \lVert 	\VEC(Z') - \VEC(Z)  \rVert_{2} \le \lVert \overline{P}^{1/2} \rVert_{2} \lVert Rv\rVert_{2} \le \left( \sup_{\phi(A_{\delta})\le w\le \phi(B_{\delta})} \psi''(w) \right)^{1/2} \lVert Rv\rVert_{2}.
		\end{align}
		On the other hand, by H\"{o}lder's inequality,  
		\begin{align}\label{eq:Z_Z_upper2}
			\lVert Rv \rVert_{2}^{2} = v^{\top} R^{\top} R v \le  \lVert v^{\top} \rVert_{\infty}  \cdot \lVert R^{\top} R v \rVert_{1} \le  4C \, \lVert (\r,\c)-(\r',\c') \rVert_{1},
		\end{align}
		where $C$ is the constant in Lemma \ref{lem:strong_dual_MLE_typical} \textbf{(iii)}. Then \eqref{eq:typical_margin_lipschiz_continuity} follows from \eqref{eq:Z_Z_upper} and \eqref{eq:Z_Z_upper2}. 
	\end{proof}

	\subsection{Stability of typical kernels and continuous MLEs}
	\label{sec:proof_typical_kernels}

	In this section, we prove Theorem \ref{thm:main_convergence_conti}.

	Let $\mu$ be a measure as in the introduction and let $A,B$ be endpoints of ${\rm supp}(\mu)$ as in \eqref{eq:supp}. 
	Let $\mathcal{W}:=L^{1}([0,1]^{2})$ denote the set of all measurable functions from $[0,1]^2\mapsto \R$ which are 
	integrable. 
	We will equip $\mathcal{W}$ with the cut-metric topology defined as follows.  The  \textit{strong cut metric}  on $\mathcal{W}$ is defined by $d_{\square}(U,W)=\lVert U-W \rVert_{\square}$ using the cut norm $\lVert \cdot \rVert_{\square}$ in \eqref{eq:def_cut_norm}.	The  \textit{weak cut metric} $\delta_{\square}$ on $\mathcal{W}$ is defined as 
	\begin{align}
		\delta_{\square}(U,W)  :=  \inf_{\phi,\psi:[0,1]\rightarrow [0,1]} \lVert U( p(\cdot), q(\cdot) ) - W \rVert_{\square}, 
	\end{align}
	where the infimum above is over all Lebesgue-measure-preserving maps $p,q$ on the unit interval $[0,1]$. Note that the following gives an equivalent definition of the cut-norm in \eqref{eq:def_cut_norm}:
	\begin{align}\label{eq:def_cut_norm2}
		\lVert W \rVert_{\square} = \sup_{f,g:[0,1] \rightarrow [0,1] } \, \left| \int_{[0,1]^{2}} f(x) W(x,y) g(y) \,dx dy \right|,
	\end{align}
	where the supremum above is over all continuous functions $f,g$ on the unit interval $[0,1]$. The two definitions \eqref{eq:def_cut_norm} and \eqref{eq:def_cut_norm2} differ upto multiplicative universal constants (\cite[Lem. 8.10]{lovasz2012large}).

	Let $\mathcal{W}^{(A,B)}$ denote the set of all kernels $W$ in $\mathcal{W}$ such that $A< W < B$. Similarly, define $\mathcal{W}^{[A,B]}$ to be the set of all kernels taking values from $[A,B]$. With $f(x)=D(\mu_{\phi(x)}\Vert \mu)$ being as in \eqref{eq:typical_table_opt}, define a function
	$G:\mathcal{W}^{(A,B)}\mapsto \R$
	\begin{align}\label{eq:G}
		G(W):=\int_{[0,1]^2} f(W(x,y)) \, dx dy.
	\end{align}
	Let $W(\cdot,\bullet):=\int_{[0,1]}W(\cdot,y)dy$ and $W(\bullet,\cdot):=\int_{[0,1]}W(x,\cdot)dx$ be univariate functions which denote the row and column marginals of $W$. Let 
	$(\r,\c)$ be a continuum margin (defined above \eqref{eq:margin_function}).  
	Define 
	\begin{align}
		\mathcal{W}_{\r,\c}:=\Big\{W\in \mathcal{W}:W(\cdot,\bullet)=\r(\cdot), W(\bullet,\cdot)=\c(\cdot)\Big\},\quad \mathcal{W}_{\r,\c}^{(A,B)}:= \mathcal{W}_{\r,\c} \cap \mathcal{W}^{(A,B)}, \quad \mathcal{W}_{\r,\c}^{[A,B]}:= \mathcal{W}_{\r,\c} \cap \mathcal{W}^{[A,B]}. 
	\end{align}
	Here $\mathcal{W}_{\r,\c}$ is the set of all kernels with  continuum  margin $(\r,\c)$.

	Now for each margin $(\r,\c)$, consider the minimization problem
	\begin{align}\label{eq:typical_kernel_opt}
		\inf_{W\in \mathcal{W}_{\r,\c}^{(A,B)}}  \,\, G(W).
	\end{align}
	If $\mathcal{W}_{\r,\c}$ is non-empty, then it is a convex subset of $\mathcal{W}$. Consequently, invoking strict convexity 
	of $G$ w.r.t. the cut-metric topology (see Proposition \ref{prop:G_USC}),  there can exist at most one global minimizer of the above optimization problem. We denote it by   $W^{\r,\c}$, if it exists. In this case, we call $W^{\r,\c}$ the \textit{typical kernel} for the margin $(\r,\c)$. 
	We say a continuum  margin $(\r,\c)$ \textit{$\delta$-tame} if a typical kernel $W^{\r,\c}$ exists and satisfies $A_{\delta}\le W \le B_{\delta}$.

	Our proof of Theorem \ref{thm:main_convergence_conti} will follow by combining the following three results. First, we show that the $L^{1}$-limit of discrete $\delta$-tame margins is a $\delta$-tame continuum margin. First, we characterize the typical kernels for $\delta$-tame margins using continuous dual variables in the same way as we did in the discrete case in Thm. \ref{thm:strong_duality_simple}.

	\begin{lemma}[Characterization of typical kernels and tame margins]\label{lem:W_typical_Lagrange}
		${}$
		\begin{description}
			\item{(a)}
			Fix a $\delta$-tame continuum margin $(\r,\c)$ in $L^{1}[0,1]$. Let $C=C_\delta$ as before. Then there exists bounded measurable functions $\balpha:[0,1]\rightarrow [-2C,2C]$ and $\bbeta:[0,1]\rightarrow [-C,C]$ such that
			\begin{align}\label{eq:W_typical_lambda_mu_gen}
				W^{\r,\c}(x,y) \overset{a.s.}{=} \psi'(\balpha(x)+\bbeta(y)). 
			\end{align}
			
			\item{(b)}
			Conversely, suppose there exists bounded measurable functions $\balpha, \bbeta:[0,1]\rightarrow \R$ such that the function $W^*(x,y)=\psi'(\balpha(x)+\bbeta(y))$ satisfies the margin $(\r,\c)\in L^{1}[0,1]$. Then $W^*$ is the unique typical kernel for the margin $(\r,\c)$, and $(\r,\c)$  is $\delta$-tame for some $\delta>0$. 
		\end{description}
	\end{lemma}

	Second, we have the following continuous extension of Theorem \ref{thm:typical_lipschitz_margins}, which establishes  Lipschitz continuity of typical kernels w.r.t. tame margins. In addition, we also establish Lipschitz continuity of the dual variables 
	w.r.t. tame margins.

	\begin{theorem}[Lipschitz continuity of typical kernels and dual variables]\label{thm:typical_kernel_Lipscthiz}
		Let $(\r,\c),(\r',\c')$ be  $\delta$-tame continuum margins on $[0,1]^{2}$.
		Furthermore, let $C_\delta:=\max\{|\phi(A_\delta)|,\, |\phi(B_\delta)| \}$. Then 
		\begin{align}\label{eq:conti_typical_continuity}
			\lVert W^{\r,\c}- W^{\r',\c'} \rVert_{2}^{2} \, \le \, 2C_{\delta} \left( \sup_{ \phi(A_{\delta})\le w \le \phi(B_{\delta}) } \psi''(w)  \right) \lVert (\r,\c) - (\r',\c') \rVert_{1}. 
		\end{align}
		Furthermore, let $(\balpha,\bbeta)$ and  $(\balpha',\bbeta')$ denote the dual variables characterizing $W^{\r,\c}$ and $W^{\r',\c'}$ via \eqref{eq:W_typical_lambda_mu_gen}, respectively. Without loss of generality, assume $\int_{0}^{1}\balpha(x)\,dx=\int_{0}^{1}\balpha'(x)\,dx=0$. Then 
		\begin{align}\label{eq:conti_MLE_continuity}
			\lVert \balpha-\balpha' \rVert_{2}^{2} + \lVert \bbeta-\bbeta' \rVert_{2}^{2} \, \le\,  2C_{\delta} \left( \frac{ \sup_{ \phi(A_{\delta})\le w \le \phi(B_{\delta}) } \psi''(w) }{ \inf_{ \phi(A_{\delta})\le w \le \phi(B_{\delta}) } \psi''(w)}  \right)  \lVert (\r,\c) - (\r',\c') \rVert_{1}. 
		\end{align}
	\end{theorem}
	
	Lastly, we show that the $L^{1}$-limit of discrete $\delta$-tame margins is a $\delta$-tame continuum margin. 
	
	\begin{prop}\label{prop:typical_kernel_existence}
		Let $(\r_{k},\c_{k})_{k\ge 1}$ be a sequence of $(m_{k}\times n_{k})$ $\delta$-tame  discrete margins converging to some continuum margin $(\r,\c)$ in $L^{1}$. Then $(\r,\c)$ is $\delta$-tame and 
		\begin{align}
			W^{(\bar{\r}_{k},\bar{\c}_{k})} \rightarrow	W^{\r,\c} \quad \textup{in $L^{2}$ as $k\rightarrow\infty$}. 
		\end{align}
	\end{prop}

	Assuming these three results above, we deduce Theorem \ref{thm:main_convergence_conti} below.

	\begin{proof}[\textbf{Proof of Theorem} \ref{thm:main_convergence_conti}]
		By the hypothesis and Proposition \ref{prop:typical_kernel_existence}, the limiting continuum margin $(\r,\c)$ is $\delta$-tame. Hence the existence of continuum dual variable $(\balpha,\bbeta)$ and the claimed characterization of the typical kernel $W^{\r,\c}$ in part \textbf{(i)} follows from Lemma \ref{lem:W_typical_Lagrange}. 	Part \textbf{(ii)} follows directly from Theorem \ref{thm:typical_kernel_Lipscthiz} and Lemma \ref{lem:W_typical_Lagrange}. 
	\end{proof}

	The rest of this section is devoted to showing the three results stated above. We first prove Lemma \ref{lem:W_typical_Lagrange}.

	\begin{proof}[\textbf{Proof of Lemma} \ref{lem:W_typical_Lagrange}]
		We first show (a). By the hypothesis of part (a), the typical table $W^{\r,\c}$ exists and it is $\delta$-tame. Let  $U(\cdot,\cdot)$ be a measurable function from $[0,1]^2$ to $[-1,1]$ which satisfies $U(\cdot,\bullet)=U(\bullet,\cdot)\stackrel{a.s.}{=}0.$ Then the function $\widetilde{W}^{(\lambda)}:=W^{\r,\c}+\lambda U\in \mathcal{W}_{\r,\c}$ for all $\lambda\in (-\delta,\delta)$. Since $W^{\r,\c}$ is the typical table, the function $\lambda\mapsto G(\widetilde{W}^{(\lambda)})$ is uniquely maximized on $(-\delta,\delta)$ at $\lambda=0$, which gives
		\begin{align*}
			0=\frac{\partial}{\partial \lambda}G(\widetilde{W}^{(\lambda)})\Big|_{\lambda=0}=\int_{[0,1]^2}\phi\left(W^{\r,\c}(x,y)\right) U(x,y)dx dy.
		\end{align*}
		Since this holds for all bounded measurable $U$ which integrates to $0$ along both marginals, it follows that there exists functions $\balpha,\bbeta:[0,1]\mapsto \R$ such that
		$$\phi\left(W^{\r,\c}(x,y)\right)\stackrel{a.s.}{=}\balpha(x)+\bbeta(y),\text{ which implies } W^{\r,\c}(x,y)\stackrel{a.s.}{=}\psi'(\balpha(x)+\bbeta(y)).$$ Finally, the fact that $W^{\r,\c}$ is $\delta$-tame implies $\phi(A_\delta)\stackrel{a.s.}{\le }\balpha(x)+\bbeta(y)\stackrel{a.s.}{\le} \phi(B_\delta)$. Since changing $(\balpha,\bbeta)$ to $(\balpha+\eta,\bbeta-\eta)$ has no impact on $W^{\r,\c}$ for any $\eta\in \R$, we can assume without loss of generality that $\int_{[0,1]}\balpha(x)dx=0$. This, along with the previous display implies
		$$\phi(A_\delta)\stackrel{a.s.}{\le} \bbeta(y)\stackrel{a.s.}{\le} \phi(B_\delta),$$
		which in turn implies $$ \phi(A_\delta)-\phi(B_\delta)\stackrel{a.s}{\le}\balpha(x)\stackrel{a.s}{\le} \phi(B_\delta)-\phi(A_\delta).$$
		The desired conclusion of part (a) follows.

		Next, we show (b). Let $\widetilde{W}\in \mathcal{W}_{\r,\c}$ be arbitrary. It suffices to show that $G(W^*)\ge G(\widetilde{W})$. Since $G$ is strictly convex, it suffices to show that the function $\lambda\mapsto G\left((1-\lambda)W^*+\lambda \widetilde{W}\right)$ on $[0,1]$ has a derivative which vanishes at $\lambda=0$. This is equivalent to checking
		$$\int_{[0,1]^2}\phi\left(W^*(x,y)\right) U(x,y)dxdy =0,\text{ where }U=W^*-\widetilde{W}.$$
		But this follows on noting that $\phi\left(W^*(x,y)\right)=\balpha(x)+\bbeta(y)$, and $U(\cdot,\bullet)=U(\bullet,\cdot)\stackrel{a.s.}{=}0$.
	\end{proof}

	Our next aim is to prove Theorem \ref{thm:typical_kernel_Lipscthiz}. We will need some preparation. 
	We first recall the following standard fact about step function approximation. 
	
	\begin{lemma}[Approximation by stepfunctions in $L^{1}$]\label{lem:approximation_step_L1}
		Suppose $|A|,|B|<\infty$ and let $h:[A,B]^{d}\rightarrow \R$ be a measurable function for some integer $d\ge 1$. Then there exists a sequence of stepfunctions $(h_{n})_{n\ge 1}$ such that $\lVert h_{n}-h \rVert_{1}\rightarrow 0$ as $n\rightarrow\infty$. 
	\end{lemma}

	Typical construction of such stepfunctions in Lemma \ref{lem:approximation_step_L1} is by block-averaging over diadic partitions. The $L^{1}$ convergence can be shown by applying L\'{e}vy's upward convergence theorem. We omit the details. 
	
	Next, we establish the basic properties of the function $G$ in \eqref{eq:G}. 
	
	\begin{prop}\label{prop:G_USC}
		The function $G$ in \eqref{eq:G}  
		is well-defined and strictly convex. 
		Furthermore, for any $\delta>0$, $G$ restricted on $\mathcal{W}^{[A_{\delta}, B_{\delta}]}$ is lower semi-continuous with respect to the cut distance $\delta_\square$.
	\end{prop}
	
	\begin{proof}
		By the assumption on $\mu$, there exists some tilting parameter $\theta_{0}\in \Theta$, for which $\mu_{\theta_{0}}$  is a probability measure. Fix $\theta\in \Theta$. Noting that the KL divergence between two probability distributions is nonnegative, 
		\begin{align}
			D(\mu_{\theta} \Vert \mu) 
			&= D(\mu_{\theta} \, \Vert \, \mu_{\theta_{0}})+ \theta_{0} \psi'(\theta) - \psi(\theta_{0}) \ge \theta_{0} \psi'(\theta) - \psi(\theta_{0}). 
		\end{align}
		It follows that for the function $f(x)=	D(\mu_{\phi(x)} \Vert \mu) $ in \eqref{eq:typical_table_opt},  $f(x) \ge \theta_{0} x- \psi(\theta_{0})$. 
		This yields that $G(W)\in (-\infty, \infty]$ for all $W\in \mathcal{W}^{(A,B)}$ since  
		\begin{align}
			G(W) = \int_{[0,1]^{2}} f(W(x,y))\,dx dy \ge \theta_{0} \int_{[0,1]^{2}} W(x,y)\,dx dy -\psi(\theta_{0})>-\infty. 
		\end{align}
		For strict convexity, fix $\lambda\in (0,1)$. Since $f$ is strictly convex on $(A,B)$ (see \eqref{eq:f_derivatives2}), we have $f(\lambda y + (1-\lambda) x)<\lambda f(y) + (1-\lambda) f(x)$ for $x,y\in (A,B)$ and $x\ne y$. It follows that $G(\lambda  W' + (1-\lambda) W)<\lambda G(W') + (1-\lambda) G(W)$ for $W,W'\in \mathcal{W}^{(A,B)}$ with $W\ne W'$ almost surely.

		Next, we consider $G$ restricted on $\mathcal{W}^{[A_{\delta},B_{\delta}]}$.
		To show lower semi-continuity, let $W_k$ be a sequence of functions in $\mathcal{W}^{[A_{\delta},B_{\delta}]}$ converging to $W\in \mathcal{W}^{[A_{\delta},B_{\delta}]}$ in $\delta_{\square}$-metric. We wish to show that
		\begin{align}\label{eq:G_LSC_pf}
			\liminf_{k\rightarrow\infty} \, G(W_k) \ge   G(W).
		\end{align} 
		Noting that $G(W)=G(W(\xi(\cdot), \eta(\cdot)))$  for every measure-preserving transformations $\xi,\eta$ on $[0,1]$, without loss of generality we can assume $\lVert W_{k}-W \rVert_{\square}\rightarrow 0$ as $k\rightarrow \infty$. 
		
		Define $W^{L}$ to be the $L\times L$ block-ageraging of $W$ for every $W\in \mathcal{W}$ and $L>0$. By convexity of $G$ and Jensen's inequality,  $G(W_{k})\ge G(W_{k}^{L})$. For fixed $L$, $G(W_{k}^{L})\rightarrow G(W^{L})$ using the continuity of $G$ on $L\times L$ stepfunctions. It follows that 
		\begin{align}\label{eq:G_LSC_pf2}
			\liminf_{k\rightarrow\infty} \, G(W_k) \ge   \liminf_{L\rightarrow\infty}\,  \liminf_{k\rightarrow\infty} \, G(W_k^{L}) = \liminf_{L\rightarrow\infty}\,  G(W^{L}). 
		\end{align} 
		Hence it suffices to show that 
		\begin{align}\label{eq:G_LSC_pf3}
			\liminf_{L\rightarrow\infty}\,  G(W^{L}) \ge G(W). 
		\end{align} 
		To this end, let $U,V$ be independent uniform $[0,1]$ variables. Then $W^{L}(U,V)$ converges to $W(U,V)$ in probability. 	 For every integrable function  $W':[0,1]\rightarrow \R$, we have $				G(W') = \E[f(W'(U,V))]$. 
		Then noting that $f$ is continuous and  the range of $f\circ W^{L}$ is in $[A_{\delta},B_{\delta}]$ for all $L\ge 1$, we get 
		\begin{align}
			\lim_{L\rightarrow\infty}	\,	G(W^{L})  = 			\lim_{L\rightarrow\infty}	\,	\E[f(W^{L}(U,V))] = 	\E[f(W(U,V))]  = 			G(W),
		\end{align}
		where the equality in the middle follows from the bounded convergence theorem noting that $f\circ W$ is bounded for $W$ taking values from $[A_{\delta},B_{\delta}]$. 
	\end{proof}

	\begin{prop}\label{prop:W_typical_averaging}
		Let $(\r,\c)$ be an $m\times n$ margin with typical table $Z^{\r,\c}\in (A,B)^{m\times n}$. Then $W_{Z^{\r,\c}}$ is the unique typical kernel for the continuum margin $(\bar{\r}, \bar{\c})$. 
	\end{prop}
	
	\begin{proof}
		Denote $W^{*}:=W_{Z^{\r,\c}}$. 	Let $I_{1}\sqcup \dots \sqcup I_{m} = [0,1]$ and $J_{1}\sqcup \dots \sqcup J_{n}=[0,1]$ be the interval partitions for $\bar{\r}$ and $\bar{\c}$, respectively. Fix any kernel $W\in \mathcal{W}_{\overline{\r},\overline{\c}}^{(A,B)}$. Let $\overline{W}$ denote the block-averaged version of $W$, which takes the contant value $\frac{1}{|I_{i}| \cdot |J_{j}|}\int_{I_{i}\times J_{j}} W(x,y) \,dx\,dy$ on the rectanble $I_{i}\times J_{j}$ for all $i,j$. Note that $\overline{W}\in \mathcal{W}_{\overline{\r},\overline{\c}}^{(A,B)}$. Since $G$ is concave, by Jensen's inequality, $G(W)\le G(\overline{W})$. Also clearly $\overline{W^{*}}=W^{*}$ and $W^{*}\in \mathcal{W}_{\overline{\r},\overline{\c}}^{(A,B)}$. It follows that 
		\begin{align}
			\inf_{W\in \mathcal{W}_{\r,\c}^{(A,B)}} G(W) =    \inf_{W\in \mathcal{W}_{\r,\c}^{(A,B)}} G(\overline{W}).
		\end{align}
		That is, the function $G$ is maximized over step kernels. Maximizing $G$ over the $m\times n$ step kernels with continuum margin $(\bar{\r}, \bar{\c})$ is exactly the typical table problem that $Z^{\r,\c}$ solves for the discrete margin $(\r,\c)$. This is enough to conclude. 
	\end{proof}

	We now prove Lipschitz continuity of typical kernels and dual variables stated in Theorem \ref{thm:typical_kernel_Lipscthiz}. Our approach is to extend the discrete analog (Thm. \ref{thm:typical_lipschitz_margins}) to the continuous case by discretizing the dual variables and passing to the limit. The key ingredient for these continuous extensions is Lemma \ref{lem:W_typical_Lagrange}, which we have established above. 
	
	\begin{proof}[\textbf{Proof of Theorem \ref{thm:typical_kernel_Lipscthiz}}]
		We will first show \eqref{eq:conti_typical_continuity} by pushing the discrete result in Theorem \ref{thm:typical_lipschitz_margins} to the continuum limit by block averaging the MLEs.
		Fix an integer $L\ge 1$. Then by Lemma \ref{lem:W_typical_Lagrange}, there exists bounded and measurable functions $\balpha,\bbeta,\balpha',\bbeta':[0,1]\rightarrow [-2C_{\delta},2C_{\delta}]$ such that 
		\begin{align}\label{eq:delta_tame_dual_characterization_cond}
			W^{\r,\c}(x,y)  = \psi'(\balpha(x)+\bbeta(y)) \quad \textup{and} \quad 	W^{\r',\c'}(x,y)  = \psi'(\balpha'(x)+\bbeta'(y)),
		\end{align}
		where $W^{\r,\c}, W^{\r',\c'}\in \mathcal{W}^{[A_{\delta},B_{\delta}]}$.
		In particular, 
		\begin{align}
			\r(x) &= \int_{0}^{1} \psi'(\balpha(x)+\bbeta(y))\,dy,\quad 			\c(y) = \int_{0}^{1} \psi'(\balpha(x)+\bbeta(y))\,dx, 
		\end{align}
		and similarly for $(\r',\c')$. 
		
		For each function $h:[0,1]\rightarrow \R$ and an integer $L\ge 1$, let $h_{L}$  denote the block average of $h$ over intervals $[(i-1)2^{-L}, i2^{-L})$ for $i=1,\dots,2^{L}$. By Lemma \ref{lem:approximation_step_L1}, there exists $\eps=\eps(L)>0$ such that 
		\begin{align}
			\lVert \balpha - \balpha_{L} \rVert_{1} + \lVert \bbeta - \bbeta_{L} \rVert_{1} +  \lVert \balpha' - \balpha'_{L} \rVert_{1} +  \lVert \bbeta' - \bbeta'_{L} \rVert_{1}  \le \eps(L)  \rightarrow 0 \quad \textup{as $L\rightarrow\infty$.}
		\end{align}
		Now define block margin $(\r_{L},\c_{L})$ by 
		\begin{align}\label{eq:margin_from_dual}
			\r_{L}(x) = \int_{0}^{1} \psi'(\balpha_{L}(x)+\bbeta_{L}(y))\,dy, \quad \c_{L}(y) = \int_{0}^{1} \psi'(\balpha_{L}(x)+\bbeta_{L}(y))\,dx.
		\end{align}
		Namely, the block margin $(\r_{L},\c_{L})$ is obtained by taking block-average of the dual variable $(\balpha,\bbeta)$. Further, note that from \eqref{eq:delta_tame_dual_characterization_cond}, 
		\begin{align}
			\phi(A_{\delta})	\le \balpha_{L}(x)+\bbeta_{L}(y) \le \phi(B_{\delta}) \quad \textup{and}\quad 	\phi(A_{\delta})	\le \balpha'_{L}(x)+\bbeta'_{L}(y) \le \phi(B_{\delta}). 
		\end{align}
		Hence $(\r_{L},\c_{L})$ and $(\r'_{L},\c'_{L})$ are both $\delta$-tame margins for all $L\ge 1$.

		Denote $D_{\delta}:= \sup_{ \phi(A_{\delta})\le w \le \phi(B_{\delta}) } \psi''(w)>0$.
		Note that 
		\begin{align}\label{eq:margin_blocking1}
			\lVert \r 	 - \r_{L} \rVert_{1} + 	\lVert \c	 - \c_{L} \rVert_{1}
			& \le D_{\delta} \int_{[0,1]^{2}} |\balpha_{L}(x)-\balpha(x) | + |\bbeta_{L}(y)-\bbeta(y)| \,dy\,dx 
			\le 2D_{\delta} \eps(L).
		\end{align}
		By a similar argument, 
		\begin{align}\label{eq:margin_blocking3}
			\lVert \r'	 - \r'_{L} \rVert_{1} + 	\lVert \c'	 - \c'_{L} \rVert_{1} \le 2D_{\delta} \eps(L).
		\end{align}
		Then by a triangle inequality, 
		\begin{align}\label{eq:pf_conti_kernel_lipschitz_triangle}
			\lVert W^{\r,\c}  - W^{\r',\c'}  \rVert_{2} \le  \lVert W^{\r,\c}  - W^{\r_{L},\c_{L}}  \rVert_{2} + \lVert W^{\r',\c'}  - W^{\r'_{L},\c'_{L}}  \rVert_{2} + \lVert W^{\r_{L},\c_{L}}  - W^{\r'_{L},\c'_{L}}  \rVert_{2}.
		\end{align}
		In order to bound the last term, we can apply Theorem \ref{thm:typical_lipschitz_margins} since both the margins and the typical kernels are stepfunctions on the intervals $[(i-1)2^{-L}, i2^{-L})$, $i=1,\dots,2^{L}$ and rectangles form by them, respectively. 
		Thus by  Theorem \ref{thm:typical_lipschitz_margins} and inequalities  \eqref{eq:margin_blocking1} and \eqref{eq:margin_blocking3}, 
		\begin{align}
			\lVert W^{\r_{L},\c_{L}}  - W^{\r'_{L},\c'_{L}}  \rVert_{2}^{2} 
			&\le 2 C_{\delta} D_{\delta} \lVert (\r_{L},\c_{L}) - (\r'_{L},\c'_{L}) \rVert_{1} \le 2 C_{\delta} D_{\delta} \left( 4D_{\delta} \eps + \lVert (\r,\c) - (\r',\c') \rVert_{1}  \right). 
		\end{align}
		The first two terms on the right-hand side of \eqref{eq:pf_conti_kernel_lipschitz_triangle} vanishes as $L,L'\rightarrow\infty$ by
		Proposition \ref{prop:typical_kernel_existence}. This shows \eqref{eq:conti_typical_continuity}.

		Next, we show the Lipschitz continuity of MLEs as stated in 
		\eqref{eq:conti_MLE_continuity}. By mean value theorem and \eqref{eq:f_derivatives}, $\phi$ restricted on $[A_{\delta},B_{\delta}]$ is $L$-Lipschitz continuous for
		\begin{align}
			L = \sup_{A_{\delta}\le t \le B_{\delta}}  \frac{1}{\psi''(\phi(t))} = \frac{1}{\inf_{\phi(A_{\delta})\le w \le \phi(B_{\delta})} \psi''(w)}. 
		\end{align}
		Define a kernel $V(x,y):=\balpha(x)+\bbeta(y)=\phi(W^{\r,\c}(x,y))$ for $x,y\in [0,1]$ and similarly define $V'$ using $(\balpha',\bbeta')$. Then by a simple computation using the fact that $\int_{0}^{1}\balpha(x)\,dx=\int_{0}^{1}\balpha'(x)\,dx=0$, we have 
		\begin{align}
			\lVert \balpha-\balpha' \rVert_{2}^{2} + \lVert \bbeta-\bbeta' \rVert_{2}^{2}  = 	\lVert V-V' \rVert_{2}^{2} \le L \lVert W^{\r,\c}-W^{\r',\c'} \rVert_{2}^{2}. 
		\end{align}
		Then \eqref{eq:conti_MLE_continuity} follows from the above and \eqref{eq:conti_typical_continuity}. 
	\end{proof}

	Lastly in this section, we prove Proposition \ref{prop:typical_kernel_existence}.

	\begin{proof}[\textbf{Proof of Proposition \ref{prop:typical_kernel_existence}}.]
		Let $Z_{k}$ denote the typical table for the $(m_{k}\times n_{k})$ $\delta$-tame margin $(\r_{m_{k}},\c_{n_{k}})$. By Proposition \ref{prop:W_typical_averaging}, $W_{k}:=W_{Z_{k}}$ is the typical kernel for the corresponding continuum step margin $(\bar{\r}_{m_{k}}, \bar{\c}_{n_{k}})$ (see \eqref{eq:margin_function}).  By Theorem \ref{thm:typical_kernel_Lipscthiz}, $W_{k}$ converges to some kernel $W^{*}$ in $L^{2}$ as $k\rightarrow \infty$. Since $W_{k}\in \mathcal{W}^{[A_{\delta},B_{\delta}]}$, it follows that $W^{*}\in \mathcal{W}^{[A_{\delta},B_{\delta}]}$. It is easy to see that $W^{*}$ has continuum margin $(\r,\c)$. 
		
		By Lemma \ref{lem:W_typical_Lagrange}, there exists a constant  $C=C_{\delta}>0$ such that for each $k\ge 1$,   there exists bounded measurable functions $\balpha_{k},\bbeta_{k}:[0,1]\rightarrow [-C,C]$ for which 
		\begin{align}
			W_{k}(x,y) \overset{a.s.}{=} \psi'(\balpha_{k}(x)+\bbeta_{k}(y)). 
		\end{align}
		Without loss of generality, we may assume $\int_{0}^{1}\balpha_{k}(x)\,dx = 0$ for all $k\ge 1$. Then by \eqref{eq:conti_MLE_continuity} in Theorem \ref{thm:typical_kernel_Lipscthiz}, we have that $\balpha_{k}\rightarrow \balpha^{*}$ and $\bbeta_{k}\rightarrow \bbeta^{*}$ in $L^{2}$ for some bounded measurable functions $\balpha^{*},\bbeta^{*}:[0,1]\rightarrow [-C,C]$. Note that $\int_{0}^{1}\balpha^{*}(x)\,dx = 0$. 
		
		Now define a kernel $W^{*}(x,y)=\psi'(\balpha^{*}(x)+\bbeta^{*}(x))$.
		By mean value theorem, $\psi'$ restricted on $[\phi(A_{\delta}), \phi(B_{\delta})]$ is $L$-Lipschitz continuous for some constant $L=L(\delta)>0$. Hence  
		\begin{align}
			\lVert W_{k}-W^{*} \rVert_{2}^{2} 
			\le L(\lVert \balpha_{k}-\balpha^{*} \rVert_{2}^{2} + \lVert \bbeta_{k}-\bbeta^{*} \rVert_{2}^{2}). 
		\end{align}
		It follows that $W_{k}\rightarrow W^{*}$ in $L^{2}$. Since $W_{k}\in \mathcal{W}_{\bar{\r}_{k},\bar{\c}_{k}}^{[A_{\delta},B_{\delta}]}$ and $(\bar{\r}_{k}, \bar{\c}_{k})\rightarrow (\r,\c)$ in $L^{1}$, this yields $W^{*}\in \mathcal{W}_{\r,\c}^{[A_{\delta},B_{\delta}]}$. By definition of $W^{*}$ and Lemma \ref{lem:W_typical_Lagrange}, we deduce that $W^{*}$ is the unique typical kernel for margin $(\r,\c)$. Since $W^{*}$ takes values from $[A_{\delta},B_{\delta}]$, we conclude that $(\r,\c)$ is $\delta$-tame. 
	\end{proof}

	\section{Proof of sharp sufficient conditions for tame margins} 
	\label{sec:tameness_pf}
	
	In this section, we prove various sufficient conditions for tame margins stated in Section \ref{sec:suff_cond_tame}. Our first goal is to prove Theorem \ref{thm:tameness_non_compact}, which requires some preparation. We first reduce the problem to symmetric margins.

	\begin{lemma}[Reduction to symmetric margins]\label{lem:reduction_symm}
		Suppose $\mu$ is such that $\Theta^{\circ}$ is unbounded. 
		Fix an $m\times n$ margin $(\r,\c)$ with $\T(\r,\c)\cap (A,B)^{m\times n}$ non-empty and assume $s<\r(i)/n<t$ and $s<\c(j)/n < t$ for all $i,j$ for some $s,t\in (A,B)$. Then for all $k_{0}$ sufficiently large, there exists a $k_{0}\times k_{0}$ symmetric margin $(\tilde{\r},\tilde{\r})$ such that $(\r,\c)$ is $\delta$-tame if $(\tilde{\r},\tilde{\r})$ is so, and  $s<\tilde{\r}(i)/k_{0}<t$ for all $i$. 
	\end{lemma}
	
	\begin{proof}
		Without loss of generality we assume $(-\infty,a)\subseteq \Theta^{\circ}$ for some $a\in \R$. 
		By Lemma \ref{lem:typical}, there exists an MLE $(\balpha,\bbeta)$ and a unique typical table $Z$ for $(\r,\c)$. By permuting the rows and columns,  without loss of generality assume that the coordinates in $\balpha$ and $\bbeta$ are ascending. We will consider to symmeri cases: (1) $\balpha(m)-\balpha(1)\le \bbeta(n)-\bbeta(1)$ and (2) $\balpha(m)-\balpha(1)> \bbeta(n)-\bbeta(1)$. We will first argue for case (1). 
		
		Assume $\balpha(m)-\balpha(1)\le \bbeta(n)-\bbeta(1)$. By shifting the MLE, without loss of generality also assume $\balpha(m)=\bbeta(n)$, so $\bbeta(1)\le \balpha(1)$. 
		Since $\Theta^{\circ}$ is an open interval and it is assumed to be unbounded, This implies that $2\balpha(m)=\balpha(m)+\bbeta(n)=2\bbeta(n)<a$. Fix an integer $k\ge 1$ and define $\tilde{\balpha}_{k}:=((\mathbf{1}_{k}\otimes \balpha )^{\top},\bbeta^{\top})^{\top}$, which is obtained by concatenating $\balpha$ $k$ times followed by $\bbeta$. Let $k_{0}:=km+n$. Define the following symmetric $k_{0}\times k_{0}$ matrix with $(k+1)\times (k+1)$ block structure: 
		\begin{align}
			\tilde{Z} := \psi'(\tilde{\balpha} \oplus \tilde{\balpha}^{\top}) = 
			\begin{bmatrix}
				\mathbf{1}_{k}\mathbf{1}_{k}^{\top}\otimes \psi'(\balpha\oplus \balpha^{\top})  & \mathbf{1}_{k}\otimes \psi'(\balpha\oplus \bbeta^{\top})  \\
				\mathbf{1}_{k}^{\top}\otimes \psi'(\balpha^{\top}\oplus \bbeta) & \psi'(\bbeta \oplus \bbeta^{\top}) 
			\end{bmatrix}
			.
		\end{align}
		Let $\tilde{\r}$ denote the row sums of $\tilde{Z}$. Then by Lemma \ref{lem:strong_dual_MLE_typical},  $(\tilde{\balpha},\tilde{\balpha})$ is an MLE for $(\tilde{\r},\tilde{\r})$ and $\tilde{Z}$ is the typical table for $(\tilde{\r},\tilde{\r})$. From the construction, it follows that $(\r,\c)$ is $\delta$-tame if  the symmetric margin $(\tilde{\r},\tilde{\r})$ is so. 
		
		Next, 
		since $\balpha(1)\ge \bbeta(1)$ and $\balpha(m)=\bbeta(n)$, 
		\begin{align}
			\tilde{\r}(1) &= \r(1)+ k  \sum_{1\le i \le m} \psi'(\balpha(1)+\balpha(i))  \ge \r(1) + \sum_{1\le i \le m} \psi'(\bbeta(1)+\balpha(i)) =\r(1) + k \, \c(1), \\ 
			\tilde{\r}(m) &= \r(m) + k \sum_{1\le i \le m} \psi'(\balpha(m)+\balpha(i))   = \r(m) +  k  \sum_{1\le i \le m} \psi'(\bbeta(n)+\balpha(i))=\r(m)+  k\, \c(n), \\
			\tilde{\r}(km+1) &= k \, \c(1) + \sum_{1\le i \le n} \psi'(\bbeta(1)+\bbeta(j)) \le k\, \c(1) + \r(1), \\
			\tilde{\r}(km+n) &= k\, \c(n) +   \sum_{1\le j \le n} \psi'(\bbeta(n)+\bbeta(j)) =  k\, \c(n) +  \sum_{1\le j \le n} \psi'(\balpha(m)+\bbeta(j)) = k\, \c(n)+\r(m).
		\end{align}
		It follows that the largest row sum of $\tilde{Z}$ is $\tilde{\r}(m)=\tilde{\r}(k_{0})\le k\c(n)+\r(m)$ and its smallest row sum is $\tilde{\r}(km+1)=k\c(1)+M$, where $M:=\sum_{1\le j \le n} \psi'(\bbeta(1)+\bbeta(j))$ is a constant that does not depend on $k$. By the hypothesis, $\tilde{\r}(k_{0})< ktm+tn \le tk_{0}$ and $\tilde{\r}(km+1)\ge ks'm+M$ for some $s'>s$. It follows that for $k$ sufficiently large, $\tilde{\r}(km+1)/k_{0}\ge (s+s')/2>s$, as desired. 
		
		For case (2), we switch the roles of $\balpha$ and $\bbeta$ (e.g., $\tilde{\bbeta}_{k}:=((\mathbf{1}_{k}\otimes \bbeta )^{\top},\balpha^{\top})^{\top}$) and apply the same argument for case (1). 
	\end{proof}

		Next, we establish a technical lemma, which essentially reduces the proof of Theorem \ref{thm:tameness_non_compact} to an extreme Barvinok margin.  Recall that we say a function $h(t)$ \textit{log-convex} if $\log h(x)$ is convex. 
		
		\begin{lemma}[Reducing symmetric margin to Barvinok margin]
			\label{lem:reduction_barvinok}
			Suppose $\psi''$ is  log-convex on $\Theta$.
			Fix a symmetric $n\times n$ margin $(\tilde{\r},\tilde{\r})$ such that $0<\tilde{\r}(1)\le \dots \le \tilde{\r}(n)$. Let  $(\tilde{\balpha},\tilde{\balpha})$ be its symmetric MLE. Then there exists a symmetric margin $(\r^{*},\r^{*})$ such that 
			\begin{align}
				\tilde{\r}(1)\le \r^{*}(1) = \cdots = \r^{*}(n-1) \le \r^{*}(n) \le \tilde{\r}(n)
			\end{align}
			and the corresponding symmetric MLE $(\balpha^{*},\balpha^{*})$ satisfies $\tilde{\balpha}(1)\le \balpha^{*}(1) \le  \balpha^{*}(n)=\tilde{\balpha}(n)$. 
		\end{lemma}
		
		\begin{proof}
			Without loss of generality assume $A\le 0 < B$ and $n\ge 3$. The overall idea of the construction is to evolve the MLEs smoothly so that the first $n-1$ row sums merge together while the last coordinate of the MLE stays put.

			\textbf{1. Computing the Jacobian.}

			Next, for each $\balpha\in [\tilde{\balpha}(1),\tilde{\balpha}(n)]^{n}$, define the corresponding row margin $\r(\balpha)\in \R^{n}$ by 
			\begin{align}
				\r(\balpha)(i) = \sum_{j=1}^{n} \psi'(\balpha(i)+\balpha(j)) \qquad \textup{for $1\le i \le n$}. 
			\end{align}
			That is, $(\r(\balpha),\r(\balpha))$ is the symmetric margin corresponding to the symmetric MLE $(\balpha,\balpha)$. 
			Let $E(\balpha)=(E_{ij})_{i,j}$ denote the $n\times n$ positive symmetric matrix with $E_{ij}=\psi''(\balpha(i)+\balpha(j))$. Let $E_{i\bullet}$ and $E_{\bullet j}$ denotes the $i$th row sum and the $j$th column sum of $E$, respectively. Since $E$ is symmetric, $E_{\bullet i} = E_{i \bullet}$. Then the Jacobian of the margin $\r$ w.r.t. the MLE $\balpha$ is given by 
			\begin{align}
				J_{\r(\balpha);\balpha} = 
				\begin{bmatrix} 
					E_{11} +  E_{\bullet 1} & E_{12} & E_{13} & \cdots & E_{1n} \\
					E_{21} & E_{22} + E_{\bullet 2} & E_{23} &\cdots  & E_{2n} \\
					\vdots &&&& \vdots \\
					E_{n1} & E_{n2} & \cdots & E_{n,n-1} & E_{nn} + E_{\bullet n} 
				\end{bmatrix}
				.
			\end{align}

			\textbf{2. Evolution of the MLE to synchronize the first $n-1$ coordinates.} 
			
			Now we define a piecewise smooth evolution of the MLE $\balpha_{t}$, $t\ge 0$ with $\balpha_{0}=\tilde{\balpha}$. Our construction is so that $\balpha_{t}$ converges to $\balpha^{*}$ where 
			\begin{align}\label{eq:alpha_prime_intermediate}
				\tilde{\balpha}(1)\le	\balpha^{*}(1)=\cdots=\balpha^{*}(n-1) \le \balpha^{*}(n)=\tilde{\balpha}(n).
			\end{align}
			Furthermore, the corresponding row sum vector, say $\r^{*}$, should satisfy 
			\begin{align}\label{eq:row_sums_shrink}
				\tilde{\r}(1) \le	\r^{*}(1) \le \r^{*}(n) \le \tilde{\r}(n). 
			\end{align}
			The evolution $\balpha_{t}$ will be defined inductively.  Let $i_{1}\in \{1,\dots,n-2 \}$ be such that the gap 
			$\balpha_{0}(i_{1}+1)-\balpha_{0}(i_{1})$ is maximized.  Then we evolve $\balpha_{t}$ so that its $i_{1}$th coordinate increases and its $i_{1}+1$st coordinate decreases until some time $t_{1}$ that the gap becomes zero. During this time interval $[0,t_{1}]$, we wish the minimum and the maximum row sums move toward each other so that \eqref{eq:row_sums_shrink} holds at intermediate times. The time derivative $\dot{\balpha}_{t}$ is defined as 
			\begin{align}\label{eq:alpha_t_evolution}
				\dot{\balpha}_{t} =
				\mathbf{e}_{i_{1}} - 	a(t) \mathbf{e}_{i_{1}+1} \quad  \textup{with} \quad a(t)=\frac{E_{n,i_{1}}(\balpha_{t})}{E_{n,i_{1}+1}(\balpha_{t})} \quad \textup{as long as $\balpha_{t}(i_{1})<\balpha_{t}(i_{1}+1)$},
			\end{align}
			where  $\mathbf{e}_{i}$ denotes the $i$th standard basis vector in $\R^{n}$ 
			From now we will supress the dependence on $\balpha_{t}$  in $E_{ij}(\balpha_{t})$. Let  $t_{1}=\inf\{ t\ge 0\,:\, \balpha_{t}(i_{1})=\balpha_{t}(i_{1}+1) \}$. Note that $t_{1}\le \balpha_{0}(i_{1}+1)-\balpha_{0}(i_{1})$. 
			
			Starting from time $t_{1}$, let $i_{2}\in \{1,\dots,n-2\}$ such that the gap  $\balpha_{t_{1}}(j_{2}+1)-\balpha_{t_{1}}(i_{2})$ is maximized and we define the dynamics similary using the index $i_{2}$ until some finite time $t_{2}\ge t_{1}$ such that $\balpha_{t_{2}}(i_{2})=\balpha_{t_{2}}(i_{2}+1)$, and so on. During this process, the last coordinate $\balpha_{0}(n)$ remains the same. 
			
			Next, we will observe that the $\balpha_{t}$ converges to some $\balpha^{*}$ satisfying \eqref{eq:alpha_prime_intermediate} as $t\rightarrow\infty$. To this effect, we claim 
			\begin{align}\label{eq:alpha_contraction}
				\balpha_{t_{k}}(n-1) - \balpha_{t_{k}}(i_{k}) \le \left( 1- \frac{1}{2n} \right)	(\balpha_{t_{k-1}}(n-1) - \balpha_{t_{k-1}}(i_{k}) ).
			\end{align}
			To show the claim, write $\mathcal{A}=\balpha_{t_{k-1}}(n-1) - \balpha_{t_{k-1}}(i_{k})$ and $\mathcal{A}'=\balpha_{t_{k}}(n-1) - \balpha_{t_{k}}(i_{k})$. Denote  $\Delta_{j}:=\balpha_{t_{k-1}}(j+1)-\balpha_{t_{k-1}}(j)\ge 0$ so that we can write $\mathcal{A}=\sum_{i_{k}\le j< n-1} \Delta_{j}$. Since $a(t)\in [0,1]$ for all $t\ge 0$, for each $k\ge 1$, the two coordinates $\balpha_{t_{k-1}}(i_{k})<\balpha_{t_{k-1}}(i_{k}+1)$ (setting $t_{0}=0$) are replaced by some value between their mean and the larger one $\balpha_{t_{k-1}}(i_{k}+1)$ at time $t_{k}$. In particular,  
			\begin{align}
				\balpha_{t_{k}-1}(i_{k}+1)- \balpha_{t_{k}}(i_{k}+1) \le	
				\Delta_{i_{k}}/2. 
			\end{align}
			Writing $	\mathcal{A}' = \Delta_{n-2}+\cdots + \Delta_{i_{k}+1} + \balpha_{t_{k-1}}(i_{k}+1) - \balpha_{t_{k}}(i_{k}+1)$ (using $\balpha_{t_{k}}(i_{k})=\balpha_{t_{k}}(i_{k}+1)$), 
			\begin{align}
				\mathcal{A}' 
				&\le \Delta_{n-2}+\cdots + \Delta_{i_{k}+1} +  \frac{\Delta_{i_{k}}}{2}  = \mathcal{A} - \frac{\Delta_{i_{k}}}{2}. 
			\end{align}
			Since $\Delta_{i_{k}}$ is the largest gap among the $n-1$ ones at time $t_{k-1}$, 
			\begin{align}
				\mathcal{A}-\mathcal{A}' \ge \frac{\Delta_{i_{k}}}{2} \ge \frac{1}{2n} \sum_{i_{k}\le j < n-1} \Delta_{j}  = \frac{1}{2n}\mathcal{A}. 
			\end{align}
			Simplifying the above shows the claim \eqref{eq:alpha_contraction}.

			Now let  $\liminf_{k\rightarrow \infty} i_{k}=i^{*}\ge 1$.  Then $i_{k}\ge i^{*}$ for all $k\ge k_{0}$ for some $k_{0}\ge 1$. In turn, $\balpha_{t_{k-1}}(n-1)-\balpha_{t_{k-1}}(i^{*})$ is a decreasing function in $t$ for $t\ge t_{k_{0}}$ and it contracts by at least a constant factor due to \eqref{eq:alpha_contraction} 
			whenever $i_{k}=i^{*}$ for $k\ge k_{0}$. Thus  $\balpha_{t_{k-1}}(n-1)-\balpha_{t_{k-1}}(i^{*})\rightarrow 0$ as $k\rightarrow\infty$. In particular, $\balpha_{t_{k-1}}(i_{k}+1)-\balpha_{t_{k-1}}(i_{k}) \rightarrow 0$ as $k\rightarrow\infty$. But since this is the maximum gap at time $t_{k-1}$, we deduce $\balpha_{t_{k-1}}(n-1)-\balpha_{t_{k-1}}(1)\rightarrow 0$ as $k\rightarrow\infty$. Also note that $\balpha_{t_{k-1}}(i^{*})$ is increasing and $\balpha_{t_{k-1}}(n-1)$ is decreasing for $k\ge k_{0}$. 
			This shows that $\balpha_{t_{k-1}}\rightarrow \balpha^{*}$ as $t\rightarrow\infty$ for some $\balpha^{*}$ satisfying \eqref{eq:alpha_prime_intermediate}.

			\vspace{0.1cm} 
			\textbf{3. Contraction of the range of the row sums.}

			Lastly, denoting $\r_{t}:=\r(\balpha_{t})$, we will show that  $\dot{\r}_{t}(1)\ge 0$ and $\dot{\r}(n)\le 0$ during $(t_{k-1},t_{k})$ for each $k\ge 1$. This is enough to conclude \eqref{eq:row_sums_shrink}. Fix $k\ge $1. 
			We wish to show 
			\begin{align}\label{eq:row_sum_evolution_claim}
				\dot{\r}_{t}(1) &=  E_{1,i_{k}}+E_{1\bullet}\mathbf{1}(i_{k}=1) - E_{1,i_{k}+1} a(t)  \ge 0 , \\
				\dot{\r}_{t}(n) &=  E_{n,i_{k}} - E_{n,i_{k}+1} a(t)  \le 0,
			\end{align}			
			The expressions for the time derivates of the extreme row sums are clear by chain rule and \eqref{eq:alpha_t_evolution}. The choice $a(t)=E_{n,i_{k}}/E_{n,i_{k}+1}$ in \eqref{eq:alpha_t_evolution} is valid if $\frac{E_{1,i_{k}}}{E_{1,i_{k}+1}}\ge \frac{E_{n,i_{k}}}{E_{n,i_{k}+1}}$, equivalently, 
			\begin{align}\label{eq:log_convex_psi_double}
				\frac{\psi''(\balpha_{t}(1)+\balpha_{t}(i_{k}))}{\psi''(\balpha_{t}(1)+\balpha_{t}(i_{k}+1))} \ge 	\frac{\psi''(\balpha_{t}(n)+\balpha_{t}(i_{k}))}{\psi''(\balpha_{t}(n)+\balpha_{t}(i_{k}+1))}. 
			\end{align}
			Recall that if $h(\cdot)$ is positive and log-convex, then for $x\le y$ and $z\le w$, we have 
			\begin{align}
				\frac{h(x+z)}{h(x+w)} \ge \frac{h(y+z)}{h(y+w)}.
			\end{align}
			Since the dynamics respects the ordering $\balpha_{t}(1)\le \cdots \le \balpha_{t}(n)$, \eqref{eq:log_convex_psi_double} follows since $\psi''$ is positive and log-convexity. This finishes the proof. 
		\end{proof}

		We are now ready to establish the first part of Theorem \ref{thm:tameness_non_compact}. 
		
		\begin{proof}[\textbf{Proof of the first part of Theorem \ref{thm:tameness_non_compact}}] 
			Suppose \eqref{eq:convex_mean_tame_condition} holds. 
			By Lemma \ref{lem:reduction_symm}, without loss of generality, we assume $\r=\c$ and $m=n\ge N_{0}$ for some large constant $N_{0}$ to be determined. Let $(\balpha,\balpha)$ be the unique symmetric MLE for $(\r,\r)$. Without loss of generality, assume $n\ge 3$ and  $\r(1)\le \dots \le \r(n)$. Then $\balpha(1)\le \dots \le \balpha(n)$. Let $Z=(z_{ij})$ denote the typical table for $(\r,\r)$. By the hypothesis, we may choose $\delta>0$ small enough so that 
			\begin{align}
				\phi(A)<\phi(A_{\delta}) < 3\phi(s)-2\phi(t) \quad \textup{and} \quad    2\phi(t)-\phi(s) < \phi(B_{\delta}) < \phi(B).
			\end{align}

			We will first show $z_{n,n}\le B_{\delta}$. 
			Let $\balpha^{*}$ and $\r^{*}$ be as in Lemma \ref{lem:reduction_barvinok}. Since $ \balpha(n)= \balpha^{*}(n)$, it suffices to upper bound $\balpha^{*}(n)$. 
			By the margin condition, Lemma \ref{lem:reduction_barvinok}, and Jensen's inequality ($\psi'$ is convex since $\psi''$ is increasing by the hypothesis), 
			\begin{align}\label{eq:pf_tame_thm_1}
				s \le n^{-1} \r(1)\le n^{-1} \r^{*}(1) &= (1-n^{-1}) \psi'(2\balpha^{*}(1)) + n^{-1} \psi'(\balpha^{*}(1)+\balpha^{*}(n)), \\
				t \ge n^{-1}\r(n) \ge  n^{-1}\r^{*}(n) &= (1-n^{-1} )\psi'(\balpha^{*}(1)+\balpha^{*}(n)) + n^{-1} \psi'(2\balpha^{*}(n)) \\
				&\ge \psi'\big((1+n^{-1}) \balpha(n) + (1-n^{-1})\balpha^{*}(1)\big) \ge \psi'\big(\balpha^{*}(n) + \balpha^{*}(1)\big).
			\end{align}
			Combining the two inequalities, we get $s \le (1-n^{-1}) \psi'(2\balpha^{*}(1)) + n^{-1} t$, 
			which yields 
			\begin{align}
				\phi\big( (1-n^{-1})^{-1} s - n^{-1}t  \big) \le 2\balpha^{*}(1). 
			\end{align}
			Then combining the above with  $t\ge \psi'(\balpha^{*}(1)+\balpha^{*}(n))$ from  \eqref{eq:pf_tame_thm_1}, 
			\begin{align}\label{eq:alpha_n_upper_bd}
				2\balpha^{*}(n) \le 2\phi(t) - 2\balpha^{*}(1) \le 2\phi(t) -   \phi\big( (1-n^{-1})^{-1} s - n^{-1}t  \big). 
			\end{align}
			By continuity of $\phi$, the right-hand side above converges to $2\phi(t)-\phi(s)$ as $n\rightarrow\infty$. Hence exists a constant $N_{0}$ depending only on $\mu,s,t$ such that the right-hand side above is well-defined and it is at most $\phi(B_{\delta})$ for all $n\ge N_{0}$. Recalling $\balpha(n)= \balpha^{*}(n)$, this shows 
			\begin{align}
				z_{n,n} = \psi'(2\balpha(n)) \le \psi'(\phi(B_{\delta}))=B_{\delta} \qquad \textup{for all $n\ge N_{0}$}. 
			\end{align}

			Next, it now remains to show $z_{11}\ge A_{\delta}$. Note that 
			\begin{align}
				s \le n^{-1} \r(1) = n^{-1} \sum_{j=1}^{n} \psi'(\balpha(1)+\balpha(j)) \le \psi'(\balpha(1)+\balpha(n)).
			\end{align}
			Using the upper bound on $\balpha(n)$ in \eqref{eq:alpha_n_upper_bd}, this gives 
			\begin{align}
				2 \balpha(1) \ge2 \phi(s) - 2\balpha(n) &\ge 2\phi(s) -  2\phi(t) +  \phi\big( (1-n^{-1})^{-1} s - n^{-1}t  \big).
			\end{align}
			The last expression above converges to $3\phi(s) - 2\phi(t)$. 
			By enlarging $N_{0}$ if necessary, it follows that the last expression above is at least $\phi(A_{\delta})$ for all $n\ge N_{0}$. This yields $z_{11}=\psi'(2\balpha(1))\ge A_{\delta}$, as desired. 
		\end{proof}

		Instead of proving the second part of Theorem \ref{thm:tameness_non_compact} directly, we will establish a stronger result below on the sharp phase transition of typical tables for Barvinok margins. Essentially for Barvinok margins, the sufficient condition for tameness in Theorem \ref{thm:tameness_non_compact} is also necessary. 
		A special case of this result, specifically for contingency tables, was previously established by Dittmer, Lyu, and Pak \cite{dittmer2020phase}. The argument presented here for the general case is significantly simpler.

		\begin{prop}[Sharp phase transition in typical tables  for Barvinok margins]\label{prop:z_tame_unbounded_gen}
			Suppose $\mu$ is a mearsure on $\R_{\ge 0}$ such that $A=0$, $B=\infty$, and $\phi(B)<\infty$. Fix constants $s,t\in (A,B)$ with $s\le t$ and a $\rho\in [0,1)$. Fix two converging sequences $s_{n}\rightarrow s$ and $t_{n}\rightarrow t$ in $(A,B)$. 
			Consider the following symmetric margin $(\r,\r)$ with 
			\begin{align}\label{eq:def_barvinok_margin1}
				\r  :=  \bigl( \underbrace{t_{n}n  ,\ldots, t_{n} n}_{\lfloor n^{\rho} \rfloor} ,   \underbrace{s_{n}n  , \ldots,  s_{n}n }_{n-\lfloor n^{\rho} \rfloor}  \bigr)
				\, \in \, \R_{\ge 0}^{n}\ts.
			\end{align}
			Then the following hold: 
			\begin{description}[itemsep=0.1cm]
				\item[(i)] If $2\phi(t)-\phi(s) < \phi(B)$, then 
				\begin{align}
					&z_{n+1,n+1} = s_{n} + O(n^{\rho-1}), \quad z_{1,n+1} = t_{n} +O(n^{\rho-1}), \\
					& z_{11} = \psi'\left( 2\phi(t_{n} - O(n^{\rho-1})) -  \phi(s_{n}  - O(n^{\rho-1}))  \right) =O(1).
				\end{align}
				In particular, for some fixed $\delta>0$, $(\r,\c)$ is $\delta$-tame for all $n\ge 1$. 
				
				\item[(ii)] If $2\phi(t)-\phi(s) >\phi(B)$, then 
				\begin{align}
					&z_{n+1,n+1} = s_{n} + O(n^{\rho-1}), \quad z_{1,n+1} = \psi'\left( \frac{a+\phi(s_{n}+\eps)}{2} +o(1) \right) , \\ 
					& 	\lfloor n^{\rho-1} \rfloor \psi'(2\alpha_{1}) = s_{n} - \psi'\left(\frac{a+\phi(t_{n}+O(n^{\rho-1}))}{2} + o(1)\right) = \Omega(1).
				\end{align}
				In particular, $(\r,\c)$ is not $\delta$-tame for all $n\ge 1$ for any fixed $\delta>0$. 
			\end{description}
		\end{prop}

		\begin{proof}
			Note that for each $n$, $\mathcal{T}(\r,\c)$ contains a positive real-valued matrix since it contains the Fisher-Yates table $(\r(i)\c(j)/N)_{i,j}$ where $N$ denotes the total sum. Hence by Lemma \ref{lem:typical}, there exists a unique typical table $Z=Z^{\r,\c}$. By Lemma \ref{lem:strong_dual_MLE_typical}, there exists an MLE $(\balpha,\bbeta)$ for margin $(\r,\c)$.  By shifting the MLE, we may assume that $\balpha(1)=\bbeta(1)$. By symmetry, $Z$ is symmetric, so this implies $\balpha=\bbeta$. Denoting $\alpha_{1}:=\balpha(1)$ and $\alpha_{2}:=\balpha(n+1)$, we have the following system of equations:
			\begin{align}\label{eq:barvinok_PT_gen_main}
				s_{n} n &= \lfloor n^{\rho} \rfloor \psi'(2\alpha_{1}) + n \psi'(\alpha_{1}+\alpha_{2}) \\
				t_{n} n &= \lfloor n^{\rho} \rfloor \psi'(\alpha_{1}+\alpha_{2}) + n \psi'(2\alpha_{2}).
			\end{align}
			Since $\psi'\ge 0$, from the first equation, we deduce  $\psi'(\alpha_{1}+\alpha_{2} ) \le s_{n}= O(1)$.  From the second equation, 
			\begin{align}\label{eq:barvinok_PT_gen_a2}
				2\alpha_{2} =  \phi(t_{n}  - O(n^{\rho-1})). 
			\end{align}
			It follows that, dropping the term $\lfloor n^{\rho} \rfloor \psi'(2\alpha_{1})\ge 0$ from the first equation in \eqref{eq:barvinok_PT_gen_main}, 
			\begin{align}\label{eq:barvinok_PT_gen_a1}
				\alpha_{1} \le \phi(s_{n}) -  \phi(t_{n}- O(n^{\rho-1}))/2.
			\end{align}

			Now suppose $2\phi(t)-\phi(s) < \phi(B)$. Since $\phi$ is continuous, there exists $\eps>0$ such that 
			\begin{align}
				2\phi(t+\eps)+\eps<\phi(s+\eps) + \phi(B). 
			\end{align}
			From \eqref{eq:barvinok_PT_gen_a1}, it follows that for all sufficiently large $n\ge 1$, $2\alpha_{1} \le \phi(B)-\eps$. 
			Hence $\psi'(2\alpha_{1})=O(1)$, so from the first equation in \eqref{eq:barvinok_PT_gen_main}, $\alpha_{1}+\alpha_{2} = \phi(s_{n} - O(n^{\rho-1}))$, 
			so using \eqref{eq:barvinok_PT_gen_a2}, 
			\begin{align}
				\alpha_{1}= \phi(s_{n} - O(n^{\rho-1})) -  \phi(t_{n}  - O(n^{\rho-1}))/2 < \phi(B) -\eps/2. 
			\end{align}
			Therefore in this case, $(\r,\r)$ is $\delta$-tame for a fixed $\delta$ for all $n\ge 1$. 
			
			Next, suppose $2\phi(t)-\phi(s) > \phi(B)$. Then there exists $\eps>0$ such that for all $n$ large enough, 
			\begin{align}
				2\phi(s_{n})-\eps>\phi(t_{n}+\eps) + \phi(B).  
			\end{align}
			Using $2\alpha_{1}\le a$ and  the first equation in \eqref{eq:barvinok_PT_gen_main}, for all sufficiently large $n\ge 1$, 
			\begin{align}
				\alpha_{1}+\alpha_{2} \le \frac{\phi(B)}{2} +\alpha_{2} = \frac{\phi(B)+\phi(t_{n}+\eps)}{2} \le \phi(s_{n})-\eps.  
			\end{align}
			Thus, in the first equation in \eqref{eq:barvinok_PT_gen_main},  $n\,\psi'(\alpha_{1}+\alpha_{2})$ can be at most $n(\phi(B)+\phi(t_{n}+\eps))/2$ and the other term $\lfloor n^{\rho} \rfloor \psi'(2\alpha_{1})$ must make up for the linearly large remainder term. Namely, 
			\begin{align}
				\lfloor n^{\rho-1} \rfloor \psi'(2\alpha_{1}) = t_{n} - \psi'(\alpha_{1}+\alpha_{2}) \ge \eps
			\end{align}
			for all $n$ large enough, so we have  $2\alpha_{1} \ge \phi(\eps n^{1-\rho})$. This yields  $\alpha_{1}\nearrow \phi(B)/2$ as $n\rightarrow\infty$. From this,
			\begin{align}
				\lfloor n^{\rho-1} \rfloor \psi'(2\alpha_{1}) &= s_{n} - \psi'\left(\frac{\phi(B)+\phi(t_{n}+O(n^{\rho-1}))}{2} + o(1)\right).
			\end{align}
		\end{proof}

		\begin{remark}\label{rmk:Barvionok_Lebesgue}
			Recently in \cite{barvinok2024quick}, Barvinok and Rudelson observed that the Barvinok margin $\r=\c=(n,\dots,n,\lambda n)$ is $\delta$-tame if $\lambda<2$ and it is not if $\lambda>2$ when $\mu$ is the Lebesgue measure on $\R_{\ge 0}$. In particular, they noted that $z_{nn}=\Omega(n)$. Our Prop. \ref{prop:z_tame_unbounded_gen} implies a more precise asymptotics of the typical table in the supercritical case $\lambda>2$. For instance, $z_{nn}=(\lambda-2+o(1))n + O(n^{-1})$.
		\end{remark}

		We can now quickly deduce Corollary \ref{cor:sharp_tame_CT} from the results we established above. 
		
		\begin{proof}[\textbf{Proof of Corollary \ref{cor:sharp_tame_CT}}]
			First, we will show that $ t/ s<\lambda_{c}$ implies tameness. All base measures in the statement satisfy the hypothesis of Theorem \ref{thm:tameness_non_compact} (see Sec. \ref{sec:examples}). 
			Since $\phi(A)=\infty$, Theorem \ref{thm:tameness_non_compact} implies tameness if $2\phi(t)-\phi(s)<0=\phi(\infty)$. A simple computation using explicit forms of $\phi$ shows that this condition is simply $t/s<\lambda_{c}$, where the critical ratio $\lambda_{c}$ is given in \eqref{eq:critical_ratio}. 
			
			Lastly, note that by  Proposition \ref{prop:z_tame_unbounded_gen}, the typical table corresponding to the Barvinok margin \eqref{eq:def_barvinok_margin1} with $c_{n} =  t+o(1)$, $s_{n}= s+o(1)$, and $2\phi( t)<\phi( s)$ 
			blows up in the $(1,1)$ entry as the size of the matrix tends to infinity. Note that $2\phi( t)<\phi( s)$  if and only if $ t/ s<\lambda_{c}$ for both measures $\mu$ we consider here. This finishes the proof. 
		\end{proof}

		Lastly in this section, we prove the sharp condition for tame margins when $B<\infty$ stated in Theorem \ref{thm:z_uniform_bd_tame}. Our proof is a combination of the symmetrization technique (Lem. \ref{lem:reduction_symm}) and a minor modification of the proof of  \cite[Lem. 12.3]{barvinok2013number}.

		\begin{proof}[\textbf{Proof of Theorem \ref{thm:z_uniform_bd_tame}}]

			By shifting, we may assume $A=0\le B<\infty$. Then the inequality \eqref{eq:EG_condition_gen} reduces to $(s+t)^{2} < 4Bs$. 
			We first argue that for each $(s,t)\in (0,B)^{2}$ with $s\le t$,
			$(s,t)\notin \Omega(\mu)$ if  and $(s+t)^{2} > 4Bs$. Fix $x,y\in (0,1)$ and $c\in \R$. Consider the following $n$-dimensional dual variables $\balpha=(-n,\dots,-n, 0,\cdots,0)$ and $\bbeta=(n+c,\cdots,n+c,c,\cdots,c)$, where $\balpha$ repeats $-n$ $\lfloor x n \rfloor$ times and $\bbeta$ repeats $n+c$ $\lfloor yn \rfloor$ times. Let $Z:=\psi'(\balpha\oplus \bbeta)$ and let $(\r,\c)$ denote the margin of $Z$. By Thm. \ref{thm:strong_duality_simple}, $(\balpha,\bbeta)$ is an MLE for $(\r,\c)$ and $Z$ is the typical table for $(\r,\c)$. By construction, the margin $(\r,\c)$ is not $\delta$-tame for any constant $\delta>0$ independent of $n$ since the MLEs diverge as $n\rightarrow\infty$.

			Since $\psi'(-n)\rightarrow 0$ and $\psi'(n+c)\rightarrow B$ as $n\rightarrow\infty$, it is easy to see that the row and column sums rescaled by $n^{-1}$ as $n\rightarrow\infty$ have only the following four values: 
			\begin{align}
				y \psi'(c), \quad yB + (1-y)\psi'(c),\quad x\psi'(c) + (1-x)B,\quad  (1-x)\psi'(c). 
			\end{align}
			Suppose $y\le 1-x$. Then the minimum and the maximum among the above are $ y \psi'(c)$ and $x\psi'(c) + (1-x)B$. The proof is finished once we choose $x,y\in (0,1)$ and $c$ so that 
			$s=y \psi'(c)$, $t = x\psi'(c) + (1-x)B$, and $y\le 1-x$. Indeed, this can be done if we choose $c$ so that $0<s\le \psi'(c) \le t<B$ and $\psi'(c)$ is sufficiently close to $s$ so that $y\le 1-x$.

			Next, we show that $(s+t)^{2} < 4Bs$ implies $(s,t)\in \Omega(\mu)$. 
			By Lemma \ref{lem:reduction_symm}, without loss of generality, we assume $m=n$ and $\r=\c$. Let $(\balpha,\balpha)$ be the unique symmetric MLE for $(\r,\r)$. Without loss of generality, assume $n\ge 3$ and  $\r(1)\le \dots \le \r(n)$. Then $\balpha(1)\le \dots \le \balpha(n)$. We will denote $\balpha=(\alpha_{1},\dots,\alpha_{n})$. Since $\psi'$ is decreasing  and $\alpha_{1}\le \dots \le \alpha_{n}$ 
			for all $i,j$, it follows that 
			\begin{align}\label{eq:z_ij_basic_bd_pf0}
				z_{i,n} &=  \psi'(\alpha_{i}+\alpha_{n})  \ge z_{1,n} \ge \psi'(\alpha_{1}+\alpha_{j})  = 		z_{1,j}, \\
				z_{n,j} &=  \psi'(\alpha_{n}+\alpha_{j})  \ge z_{n,1} \ge \psi'(\alpha_{i}+\alpha_{1})  = 		z_{i,1}.
			\end{align}

			Fix $\eps>0$ sufficiently small. Note that $\phi(x)\rightarrow -\infty$ as $x\searrow 0$ and $\phi(B-x)\rightarrow \infty$ as $x\searrow 0$. Hence there are constants $c_{1},c_{2}$ (depending on $\eps$) such that
			\begin{align}\label{eq:Z_bd_thm_c_choice_pf}
				&	\phi(c_{1})+\phi(B-\eps)\le \phi( s-\eps), \qquad 	\phi(c_{1})+\phi(B-c_{2})\ge \phi( t+\eps).
			\end{align}
			We claim that 
			\begin{align}\label{eq:z_bd_dual_variables_bd_claim}
				\phi(c_{1}) & \overset{(a)}{\le} \alpha_{1} \le  \alpha_{n} \overset{(b)}{\le } \phi(B-c_{2}).
			\end{align}
			Since $\phi(z_{ij}) = \alpha_{i}+\alpha_{j}$ and $\phi$ is increasing, the assertion follows immediately from this claim. Furthermore, (b) follows from (a). Indeed, suppose (b) is not true while (a) holds.  Then 
			\begin{align}
				t  \ge	n^{-1}\r(n)	=n^{-1}	\sum_{i} z_{i,n}	= n^{-1}\sum_{i} \psi'(\alpha_{i} +\alpha_{n}) \ge  \psi'( \phi(c_{1}) + \phi(B-c_{2}) ) \ge ( t+\eps)n,
			\end{align}
			which is a contradiction. Thus it is enough to show \eqref{eq:z_bd_dual_variables_bd_claim} (a).

			Suppose for contradiction (a) does not hold, i.e.,  $\alpha_{1}<\phi(c_{1})$. Then necessarily $\alpha_{n}> \phi(B-\eps)$, since otherwise for all $1\le j \le n$, 
			\begin{align}
				z_{1,j} = \psi'(\alpha_{1}+\alpha_{j}) \le \psi'(\alpha_{1}+\alpha_{n})    \le \psi'\left( \phi(c_{1})+\phi(B- \eps)\right)  \le \psi'(\phi( s-\eps))= s-\eps.
			\end{align}
			So this implies $ s \le r_{1}/n \le  s-\eps$, a contradiction.

			Now since $\alpha_{1}<\phi(c_{1})$ and $\beta_{n}>\phi(B-\eps)$, 
			\begin{align}
				z_{i,n} &= \psi'(\alpha_{i}+\alpha_{n}) \ge  \psi'(\alpha_{n})  \ge  \psi'( \phi(B-\eps)) = B-\eps \qquad \textup{if $\alpha_{i}\ge 0$ }, \\
				z_{1,j} &=  \psi'(\alpha_{1}+\alpha_{j}) \le \psi'(\alpha_{1})  \le \psi'(\phi(c_{1}))=c_{1} \qquad \textup{if $\alpha_{j}\le 0$}.
			\end{align}
			Let $\rho:=n^{-1} \{1\le i \le n \,:\, \alpha_{i}\ge 0 \} $. Then we have 
			\begin{align}
				&	 t n \ge 	c_{n} = \sum_{i} z_{i,n} = \sum_{i; \alpha_{i}\ge  0} z_{i,n} +  \sum_{i; \alpha_{i}< 0} z_{i,n} \ge  n \rho (B-\eps) + (1-\rho) n z_{1,n}, \\
				&	 s n 	\le r_{1} = \sum_{j} z_{1,j} 
				= \sum_{j; \alpha_{j}<  0} z_{1,j} +  \sum_{j; \alpha_{j}\ge 0} z_{1,j}  \ge (1-\rho) c_{1} + \rho n z_{1,n}.
			\end{align}
			It follows that  
			\begin{align}\label{eq:tame_compact_pf1}
				& t \ge  \rho (B-\eps) + (1-\rho) z_{1,n}, \qquad  s \le (1-\rho)c_{1} + \rho z_{1,n}.
			\end{align}
			Denote $\tau=z_{1,n}$. Since $\rho\in [0,1]$, 
			\begin{align}\label{eq:pf_z_bd_1}
				t +\eps \ge \rho B + (1-\rho) \tau, \qquad  s-c_{1} \le \rho \tau. 
			\end{align}
			This yields 
			\begin{align}
				t + \eps \ge \rho B + (1-\rho) \tau \ge 2\sqrt{B \rho \tau}  - \rho \tau \ge  2\sqrt{B(s-c_{1})} - (s-c_{1}), 
			\end{align}
			where the first and the last inequalities above are from \eqref{eq:pf_z_bd_1}. The middle inequality above uses $\frac{\rho B+\tau}{2} \ge  \sqrt{\rho B \tau}$ and the fact that the function $x\mapsto 2\sqrt{ Bx} - x$  is increasing for $x\in [0,B]$. 
			Thus if 
			\begin{align}\label{eq:pf_z_bd_2}
				t +  \eps < 2\sqrt{B(s-c_{1})} - (s-c_{1}),
			\end{align}
			then this leads to a contradiction. Note that \eqref{eq:pf_z_bd_2} holds for $\eps, c_{1}$ sufficiently small if $t<2\sqrt{Bs}-s$, or  $(s+t)^{2}< 4Bs$. Thus we conclude that  \eqref{eq:z_bd_dual_variables_bd_claim} (a) hold, as desired. 
		\end{proof}

		Lastly in this section, we prove Theorem \ref{thm:tame_EG}. 
		
		\begin{proof}[\textbf{Proof of Theorem \ref{thm:tame_EG}}] 
			We first show the ``only if'' part. 
			Fix $I\subseteq [n]$. $\delta$-tameness implies that there is a symmetric MLE $(\balpha,\balpha)$ for margin $(\r,\r)$ such that $c_{1}:=\delta \le \psi'(\balpha\oplus \balpha)\le B-\delta=:c_{2}$. Denote $z_{ij}:=\psi'(\balpha\oplus \balpha)$. The typical table $ \psi'(\balpha\oplus \balpha)$ satisfies the margin $(\r,\c)$, so $c_{1} \le \r(i)/n \le c_{2}$ for all $i,j$. 
			Now writing $\r(i)=\sum_{j}z_{ij}$ and using $z_{ij}\le 1-\delta$, 
			\begin{align}
				B|I|^{2} - \sum_{i\in I}\r(i) &\ge \delta B|I|^{2} - \sum_{i\in I}\sum_{j\notin I} z_{ij}. 
			\end{align}
			Also, since $B|I|\ge \sum_{i\in I}z_{ij}$ due to $z_{ij}\in [0,B]$, 
			\begin{align}
				\sum_{j\notin I} B|I|\land \c(j) \ge \sum_{j\notin I} \sum_{i\in I} z_{ij}. 
			\end{align}
			Combining the above, we deduce 
			\begin{align}
				B|I|^{2}+  \sum_{j\notin I} B|I|\land \r(j) - \sum_{i\in I}\r(i) \ge \delta |I|^{2}. 
			\end{align}
			Hence \eqref{eq:matrix_EG} holds with $c_{3}=\delta$. 
			
			Next, we show the ``if'' part, which is the more substantial implication.  
			For the case of  $\mu=\textup{Uniform}(\{0,1\})$,  Chatterjee, Diaconis, and Sly showed this implication in \cite[Lem. 4.1]{chatterjee2011random}. The implication for the general case follows from a minor modification of their argument. Below we provide a detailed argument for completeness. 
			
			Suppose there exist constants $c_{1},c_{2},c_{3}>0$ independent of $\r$ such that  $c_{1} \le \r(i)/n \le c_{2}$ for all $i$ and  \eqref{eq:matrix_EG} hold. 
			Assume that $\r(1)\le \dots \le \r(n)$ and $\balpha(1)\le \cdots \le \balpha(m)$. 
			We will show that $\balpha(n) \le C$ for some constant $C=C(\mu,c_{1},c_{2},c_{3})>0$ independent of $\r$. 
			Given this, we also have a lower bound $\balpha(1)\ge \phi(c_{1})-C$ from 
			\begin{align}
				c_{1} n \le \r(1) \le \sum_{j=1}^{n} \psi'(\balpha(1)+\balpha(j)) \le n \psi'(\balpha(1)+\balpha(n)) \le n \psi'(\balpha(1)+C). 
			\end{align}
			Then $\psi'(2\phi(c_{1})-C)\le \psi'(\balpha\oplus \bbeta)\le \phi'(2C)$, so for $(\r,\r)$ is $\delta$-tame for $\delta>0$ small enough so that $\delta<\psi'(2\phi(c_{1})-C)$ and $\phi'(2C)\le B-\delta$.

			Without loss of generality, we can assume $\balpha(n)\ge C_{0}$ for any constant $C_{0}=C_{0}(\mu,c_{1},c_{2},c_{3})>0$. 
			We first claim that a constant fraction of coordinates of $\balpha$ is also large: 
			\begin{align}\label{eq:EG_pf_1}
				| \{i \,:\, \balpha(i) \ge \balpha(n)/4 \}| \ge c_{1}^{2}n. 
			\end{align}
			To see this, let $  M:= |\{ i \,:\, \balpha(i)>-\balpha(n)/2  \}|$. 
			Note that $M\ge 1$ since $\balpha(n)\ge C_{0}>0$. Since $\psi'$ takes values in $(0,B)$, 
			it follows that 
			\begin{align}
				c_{2}n \ge  \r(n) = \sum_{j=1}^{n}\psi'(\balpha(n)+\balpha(j))\ge M \psi'(\balpha(n)/2). 
			\end{align}
			Assuming $C_{0}>2\phi(c_{2})$, we can impose $\balpha(n)> 2\phi(c_{2})$.  Then $\psi'(\balpha(n)/2)>c_{2}$, so the above yields $M<n$. It follows that there exists an index $j$ such that $\balpha(j)\le-\balpha(n)/2$. In particular, $\balpha(1)<0$. Let 
			\begin{align}
				M_{1}:= \{ i\,:\, i\ne j,\, \balpha(i)<-\balpha(1)/2 \}. 
			\end{align}
			Then using $\balpha(1)\le -\balpha(n)/2$, note that 
			\begin{align}
				c_{1} n \le  \c(1) 
				\le M_{j}\psi'(\balpha(1)/2) + B(n-M_{1}) \le M_{j}\psi'(-\balpha(n)/4) + B(n-M_{1}).
			\end{align}
			Assuming $C_{0}>-4\phi(\frac{Bc_{1}}{1+c_{1}})$, we can further impose $\balpha(n)\ge -4\phi(\frac{Bc_{1}}{1+c_{1}})$ so that $\psi'(-\balpha(n)/4)\le B\delta/(1+c_{1})$. Hence 
			\begin{align}
				M_{1}
				\le \frac{(B-\delta)n}{B-\psi'(-\balpha(n)/4)} \le (1-c_{1}^{2})n. 
			\end{align}
			Hence $n-M_{1}\ge c_{1}^{2}n$, so there are at least $c_{1}^{2}n$ indices $i$ such that $\balpha(i)\ge -\balpha(1)/2\ge \balpha(n)/4$. This shows the claim \eqref{eq:EG_pf_1}. 
			
			Next, assume $C_{0}>17^{2}$ so that  $h:=\sqrt{\balpha(n)}\ge 17$. For each integer $k$  between $0$ and $\frac{h}{16}-1$, define 
			\begin{align}
				D_{k}:= \left\{ i\,:\, -\frac{\balpha(n)}{8}+kh \le \balpha(i) < -\frac{\balpha(n)}{8}+(k+1)h \right\}.
			\end{align}
			Since $D_{0},D_{1},\dots$ are disjoint, there exists an integer $1\le k\le \frac{h}{16}-1$ such that $ |D_{k}|\le \frac{n}{(h/16)-1}$. 
			Fix such an integer $k$ and let 
			\begin{align}
				I:=\left\{ i\,:\,   \balpha(i) \ge  \frac{\balpha(n)}{8} - \left(k+\frac{1}{2} \right)h  \right\}. 
			\end{align}
			Since $ \frac{\balpha(n)}{4} \ge \frac{\balpha(n)}{8} - \left(k+\frac{1}{2} \right)h$, we have  $|I|\ge c_{1}^{2}n$ by the claim \eqref{eq:EG_pf_1}.

			Now since $\balpha(i)\ge \balpha(n)/16$ for each $i\in I$, we have
			\begin{align}\label{eq:EG_gen_pf1}
				B  |I|^{2} - \sum_{i,j\in  I} z_{ij} \ge (B-\psi'(\balpha(n)/8))|I|^{2}. 
			\end{align}
			Next, fix $j\notin I$. We consider three cases. First, suppose $\balpha(j)\ge -\frac{\balpha(n)}{8}+(k+1)h$. Then for each $i\in I$, we have $\balpha(i)+\balpha(j)\ge h/2$, so 
			\begin{align}
				\r(j)\land B|I| - \sum_{i\in I} z_{ij} \le  B|I| - \sum_{i\in I} z_{ij} \le |I|(B - \psi'(h/2)). 
			\end{align}
			Second, suppose $\balpha(j)\ge -\frac{\balpha(n)}{8}+kh$. Then for each $i\notin I$, we have $\balpha(i)+\balpha(j)\le -h/2$, so 
			\begin{align}
				\r(j)\land B|I| - \sum_{i\in I} z_{ij} \le  \r(j) - \sum_{i\in I} z_{ij} =\sum_{i\notin I} z_{ij} \le n \, \psi'(-h/2).
			\end{align}
			Lastly, the third case is the one in which $j\notin I$ does not belong to either of the previous two cases. The set of all such $j$'s is contained in the set $D_{k}$. Hence combining the three cases above with  \eqref{eq:EG_gen_pf1}, 
			\begin{align}\label{eq:EG_gen_pf3}
				\sum_{j\notin I} \left(  \r(j)\land B|I|   \right)  + B|I|^{2} - \sum_{i\in I}\r(i) \le \left[ (B-\psi'(h/2)) +  \psi'(-h/2) + \frac{1}{(h/16)-1} + (B-\psi'(h^{2}/8))\right]n^{2}. 
			\end{align}
			According to \eqref{eq:matrix_EG}, the left-hand side of the above inequality is bounded below by $c_{3}|I|^{2}\ge c_{3}c_{1}^{2}n^{2}$. The coefficient of $n^{2}$ on the right-hand side tends to zero as $\balpha(n)=h^{2}\rightarrow\infty$, which is a constracition.  So there exists a constant $C=C(\mu,c_{1},c_{2},c_{3})\ge C_{0}>0$ such that $h> C$  implies that the right-hand side above is at most $c\delta^{2}n^{2}$. Then $\balpha(n)\le C$, as desired.
		\end{proof}

		\section{Proof of convergence of generalized Sinkhorn algorithm}
		\label{sec:sinkhorn}
		
		In this section, we establish Theorem \ref{thm:Sinkhorn_conv} on the linear convergence of the generalized Sinkhorn algorithm \eqref{eq:AM_typical_MLE2}. The key difficulty is showing that the sequence of dual variables along the trajectory of the Sinkhorn algorithm stays bounded. In the special case of entropic optimal transport, a uniform bound on the norm of the dual variables (e.g., \cite[Lem. 2.3]{marino2020optimal}) is established relying heavily on the closed form of Sinkhorn iterates \eqref{eq:AM_typical_MLE3}, which is enjoyed only for the special case of the Poisson base measure in our setting. In the lemma below, we establish that the generalized Sinkhorn algorithm is a `monotone' in the $L^{\infty}$-norm. 
		
		\begin{lemma}[$L^{\infty}$-monotonicity of the Sinkhorn iterates]\label{lem:L_infty_non_expansion_sinkhorn}
			Suppose $\mathcal{T}(\r,\c)\cap (A,B)^{m\times n}$ is non-empty. Let $(\balpha_{k},\bbeta_{k})$, $k\ge 0$ denote the iterates produced by the Sinkhorn algorithm \eqref{eq:AM_typical_MLE2}. Let $(\hat{\balpha},\hat{\bbeta})$ be an arbitrary MLE for the margin $(\r,\c)$. 
			\begin{description}
				\item[(i)] For each $k\ge 0$, $\displaystyle \lVert (\balpha_{k+1},\bbeta_{k+1}) - (\hat{\balpha},\hat{\bbeta}) \rVert_{\infty}\le \lVert (\balpha_{k},\bbeta_{k}) - (\hat{\balpha},\hat{\bbeta}) \rVert_{\infty} \le \| \balpha_{0}-\hat{\balpha} \|_{\infty}$. 
				
				\item[(ii)] Suppose $(\r,\c)$ is $\delta$-tame, $\psi''$ is increasing, and  $\balpha_{0}= \mathbf{0}$. 
				Then $(\balpha_{k},\bbeta_{k})$ is $\delta$-tame for all $k\ge 0$. 
			\end{description}

		\end{lemma}
		
		\begin{proof}
			By permuting the rows and columns if necessary, we may assume that $\r(1)\le \dots \le \r(m)$ and $\c(1)\le \dots \le \c(n)$. Since $\psi'$ is increasing, it follows that $\hat{\balpha}(1)\le \dots \le \hat{\balpha}(m)$ and $\hat{\bbeta}(1)\le \dots  \le\hat{\bbeta}(n)$. 
			Fix $(\balpha,\bbeta)\in \R^{m}\times \R^{n}$. Let $\bbeta\mapsto \xi(\bbeta)=:\balpha'$ denote the Sinkhorn update for the first dual variable given the second one $\bbeta$. Since $\psi'$ is increasing and since $\r(i)=\sum_{j}\psi'(\balpha'(i)+\bbeta(j))$ for all $i$, 
			it follows that 
			$\balpha'$ has increasing coordinates. 
			We would like to compute the Jacobian of this map. To do so, define the function $F:\R^{m}\times \R^{n}\rightarrow \R^{m}$ as 
			\begin{align} 
				F(\balpha, \bbeta):= \left( \r(1) - \sum_{j} \psi'(\balpha(1)+\bbeta(j)),\dots,\r(m) - \sum_{j} \psi'(\balpha(m)+\bbeta(j))  \right). 
			\end{align}
			Then $\balpha'=\xi(\bbeta)$ is the unique zero of the equation $F(\cdot, \bbeta)=\mathbf{0}$. Let $E=E(\balpha',\bbeta)$ be the $m\times n$ matrix whose $(i,j)$ entry is $-\psi''(\balpha'(i)+\bbeta(j))$ and  let 
			$E_{i \bullet }$ denote the $i$th row sum of $E$ for $i=1,\dots,m$. Then the Jacobian of $F$ with respect to $\balpha$ and $\bbeta$, respectively, are given by 
			\begin{align}
				[J_{F;\balpha}(\balpha',\bbeta)]_{m\times m} &= \textup{diag}(E_{1 \bullet },\dots,E_{m\bullet }),\qquad [J_{F;\bbeta}(\balpha',\bbeta)]_{m\times n}  = E.
			\end{align}
			The first Jacobian matrix above is always invertible since $\psi''>0$ on the domain. Hence by the implicit function theorem, 
			\begin{align}
				[J_{\balpha';\bbeta}]_{m\times n} = 	 - [J_{F;\balpha}(\balpha',\bbeta)]_{m\times m}^{-1}  	[J_{F;\bbeta}(\balpha',\bbeta)]_{m\times n}  &= - \left[ E(\balpha',\bbeta)_{ij}/E(\balpha',\bbeta)_{i\bullet} \right]_{m\times n}.
			\end{align}
			Importantly, we observe that $-[J_{\balpha';\bbeta}]$ is a row-stochastic matrix. 
			
			Now fix any MLE $(\hat{\balpha},\hat{\bbeta})$ for the margin $(\r,\c)$. Note that $\xi'(\hat{\bbeta})=\hat{\balpha}$. Let $\gamma(s)=(1-s)\bbeta + s \hat{\bbeta}$ denote the linear interpolation between $\bbeta$  and $\hat{\bbeta}$. Then denoting $P_{s}:=-J_{\xi(\gamma(s));\gamma(s)}$, we have 
			\begin{align}\label{eq:sinkhorn_avg_one_step}
				\balpha' - \hat{\balpha}   =  \xi(\bbeta) - \xi(\hat{\bbeta})  
				&= \underbrace{\bigg[ \int_{0}^{1} P_{s} \,ds \bigg] }_{=:\, P}(\hat{\bbeta}-\bbeta). 
			\end{align}
			The matrix $P$ defined above is row-stochastic since every intermediate negative Jacobian matrix is row-stochastic by the earlier observation. In particular, since $\| P \|_{\infty}=1$, this yields 
			\begin{align}
				\| \balpha' - \hat{\balpha} \|_{\infty} \le \| P \|_{\infty} \|\hat{\bbeta}-\bbeta \|_{\infty} = \|\hat{\bbeta}-\bbeta \|_{\infty}. 
			\end{align}
			By a symmetric argument, it also holds that 
			\begin{align}
				\left\lVert  \bbeta' - \hat{\bbeta}  \right\rVert_{\infty} \le \lVert \hat{\balpha}-\balpha \rVert_{\infty}, 
			\end{align}
			where $\bbeta'$ denotes the output of the Sinkhorn update for the second dual variable given the first dual variable $\balpha$. It then follows that, for all $k\ge 0$, 
			\begin{align}
				\left\lVert  \balpha_{k+1} - \hat{\balpha}  \right\rVert_{\infty}  \le \lVert \bbeta_{k+1} - \hat{\bbeta} \|_{\infty} \le \left\lVert  \balpha_{k} - \hat{\balpha}  \right\rVert_{\infty}\le \lVert \bbeta_{k} - \hat{\bbeta} \|_{\infty}.
			\end{align}
			By induction, from the above, we can deduce  \textbf{(i)}.

			Next, we show \textbf{(ii)}. Without loss of generality, we may assume $\psi''$ is increasing and let $\balpha_{0}= \mathbf{0}$. We may shift the MLE (which we will still denote as $(\hat{\balpha},\hat{\bbeta})$) if necessary so that $\hat{\balpha}\le \balpha_{0}$ entrywise. 
			We claim the following: 
			\begin{description}[itemsep=0.1cm]
				\item{(a)} $\hat{\bbeta}-\bbeta_{k}$ has increasing nonnegative coordinates for $k\ge 1$. 
				
				\item{(b)} $\balpha_{k}-\hat{\balpha}$ has decreasing nonneative coordinates for $k\ge 0$.
			\end{description}
			Indeed, note that (b) holds for $k=0$ by the hypothesis. Suppose $k\ge 1$. By applying 
			\eqref{eq:sinkhorn_avg_one_step} for $\bbeta=\bbeta_{k}$, 
			\begin{align}\label{eq:alpha_k_gap_monotonicity}
				\balpha_{k}(i)-\hat{\balpha}(i) = \int_{0}^{1} P_{s}[i,:] (\hat{\bbeta}-\bbeta_{k})\, ds \ge 0.
			\end{align}
			Denoting $\bbeta^{s}:=(1-s)\bbeta_{k}+s\hat{\bbeta}$, the corresponding dual variable $\balpha^{s}:=\xi(\bbeta^{s})$  has increasing coordinates by the observation in the first paragraph. Since $\psi''$ is increasing, it follows that $P_{s}$ is the row-normalization of an $m\times n$  matrix $E(\balpha^{s},\bbeta^{s})$ with coordinates increasing along both row and column indices. 
			So the rows $P_{s}[i,:]$ and $P_{s}[i',:]$ for $i<i'$ are probability distributions over $[n]$ where the former assigns smaller weights on lesser indices than the latter. More precisely, there exists an index $j^{*}$ such that $P_{s}[i,j]\le P_{s}[i',j]$ for $j\le j^{*}$ and the inequality reverses for $j>j^{*}$. Since $\hat{\bbeta}-\bbeta_{k}$ is increasing, one can easily check that $P_{s}[i,:](\hat{\bbeta}-\bbeta_{k})\ge P_{s}[i',:](\hat{\bbeta}-\bbeta_{k})$. Integrating in $s$,  \eqref{eq:alpha_k_gap_monotonicity} implies that $\balpha_{k}-\hat{\balpha}$ is decreasing. By a symmetric argument, one can deduce that $\hat{\bbeta}-\bbeta_{k}$ has increasing nonnegative coordinates. Repeating the same argument inductively, this shows (a) and (b) above.

			Next, we claim that $(\balpha_{k},\bbeta_{k})$ is $\delta$-tame. For each $i,j$, observe that 
			\begin{align}
				\balpha_{k}(m) + \bbeta_{k}(n) 
				& =  \hat{\balpha}(m)+ \left[ \balpha_{k}(m)-\hat{\balpha}(m)\right]  + \hat{\bbeta}(n) +  \left[ \bbeta_{k}(n)-\hat{\bbeta}(n) \right]  \\
				& = \left( \hat{\balpha}(m) +  \hat{\bbeta}(n)\right) + \underbrace{P[m,:](\hat{\bbeta}-\bbeta_{k}) - (\hat{\bbeta}(n)-\bbeta_{k}(n))}_{\le 0}\le \hat{\balpha}(m) +  \hat{\bbeta}(n) \le \phi(B_{\delta}),
			\end{align}
			where the first inequality follows since $\hat{\bbeta}-\bbeta_{k}$ has increasing (claim (a) above) coordinates and the second inequality follows since $(\hat{\balpha},\hat{\bbeta})$  is $\delta$-tame. Similarly, 
			\begin{align}
				\balpha_{k}(1) + \bbeta_{k}(1) 
				& =  \hat{\balpha}(1)+ \left[ \balpha_{k}(1)-\hat{\balpha}(1)\right]  + \hat{\bbeta}(1) +  \left[ \bbeta_{k}(1)-\hat{\bbeta}(1) \right]  \\
				& = \left( \hat{\balpha}(1) +  \hat{\bbeta}(1)\right) + \underbrace{P[:,1](\hat{\bbeta}-\bbeta_{k}) - (\hat{\bbeta}(1)-\bbeta_{k}(1))}_{\ge 0}\ge \hat{\balpha}(1) +  \hat{\bbeta}(1) \ge \phi(A_{\delta}).
			\end{align}
			Since both $\balpha_{k}$ and $\bbeta_{k}$ are increasing, this is enough to conclude \textbf{(ii)}.  
		\end{proof}

		\begin{proof}[\textbf{Proof of Theorem \ref{thm:Sinkhorn_conv}}]
			We first claim the following: \eqref{eq:sinkhorn_linear_conv} holds if all Sinkhorn iterates $(\balpha_{k},\bbeta_{k})$ for $k\ge k_{0}$ as well as the MLE $(\balpha^{*},\bbeta^{*})$ are $\eps$-tame.
			Before proving this claim, we will first deduce parts \textbf{(i)}-\textbf{(iii)} from this claim. First, we remark that there are some well-known results from the optimization literature that are directly applicable to the generalized Sinkhorn algorithm \eqref{eq:AM_typical_MLE2}. 
			Since each sub-problem in \eqref{eq:AM_typical_MLE2} has a unique solution due to strong concavity of the block-restricted dual objective, asymptotic convergence to the critical point of \eqref{eq:AM_typical_MLE2} follows from a general result for alternating maximization (e.g., \cite[Prop. 2.7.1]{bertsekas1997nonlinear}). Every critical point of the dual objective is an MLE, which is a global optimum by Lemma \ref{lem:strong_dual_MLE_typical}. Thus it follows that $\balpha_{k}\oplus \bbeta_{k}\rightarrow \balpha^{*}\oplus \bbeta^{*}$  as $k\rightarrow\infty$ entrywise. In particular, if we choose $\eps>0$ small enough so that $(\r,\c)$ is $\eps$-tame (i.e., $\phi(A_{\eps})\le \balpha^{*}\oplus \bbeta^{*}\le \phi(B_{\eps})$), then there exists $k_{0}\ge 1$ such that $\phi(A_{\eps/2})\le \balpha_{k} \oplus \bbeta_{k}\le \phi(B_{\eps/2})$ for all $k\ge k_{0}$. Then \textbf{(i)} follows from the claim. For \textbf{(ii)}, the hypothesis of the claim is directly justified by Lemma \ref{lem:L_infty_non_expansion_sinkhorn} \textbf{(ii)}.
			For \textbf{(iii)}, the hypothesis and Lemma \ref{lem:L_infty_non_expansion_sinkhorn} \textbf{(i)} impliy 
			\begin{align}
				\balpha_{k}\oplus \bbeta_{k} \le \phi(B_{\eps}) - 2\| \balpha_{0}-\balpha^{*} \|_{\infty} + \| \balpha_{k}-\balpha^{*} \|_{\infty} + \| \bbeta_{k}-\bbeta^{*} \|_{\infty} \le \phi(B_{\eps}).
			\end{align}
			The lower bound follows similarly. Thus $(\balpha_{k},\bbeta_{k})$ is $\eps$-tame for all $k\ge 0$.

			It now suffices to show the claim. Our analysis for this is inspired by the proof of linear convergence of the Sinkhorn algorithm for entropic optimal transport due to Carlier \cite{carlier2022linear}.  For simplicity denote $F:= -g^{\r,\c}$ (see \eqref{eq:typical_Lagrangian}). Consider the following centered Sinkhorn iterates $(\tilde{\balpha}_{k},\tilde{\bbeta}_{k})$ for $k\ge 1$  where $(\tilde{\balpha}_{0},\tilde{\bbeta}_{0})=(\balpha_{0},\bbeta_{0})$ and for $k\ge 1$, $(\tilde{\balpha}_{k},\tilde{\bbeta}_{k})$ is obtained from $\tilde{\bbeta}_{k-1}$ by the same Sinkhorn update in \eqref{eq:AM_typical_MLE2} but follow by centering (adding and subtracting the same constants to $\tilde{\balpha}_{k}$ and $\tilde{\bbeta}_{k}$, respectively) so that 
			$\sum_{i}\tilde{\balpha}_{k}(i)=0$. By an induction, it is easy to verify 
			\begin{align}\label{eq:alpha_centering_claim}
				\balpha_{k}\oplus \bbeta_{k} =  \tilde{\balpha}_{k}\oplus \tilde{\bbeta}_{k} \qquad \textup{for all $k\ge 0$}. 
			\end{align}
			We also let  $(\tilde{\balpha}^{*},\tilde{\bbeta}^{*})$ denote the standard MLE for $(\r,\c)$. Note that $\sum_{i} \tilde{\balpha}^{*}(i)=0$ and $\balpha^{*}\oplus \bbeta^{*}=\tilde{\balpha}^{*}\oplus \tilde{\bbeta}^{*}$. In particular, $F(\balpha_{k},\bbeta_{k})=F(\tilde{\balpha}_{k},\tilde{\bbeta}_{k})$ for all $k\ge 0$ and $F(\balpha^{*},\bbeta^{*})=F(\tilde{\balpha}^{*},\tilde{\bbeta}^{*})$. Thus it is enough to show the claim for the centered iterates $(\tilde{\balpha}_{k},\tilde{\bbeta}_{k})$  and standard MLE $(\tilde{\balpha}^{*},\tilde{\bbeta}^{*})$. We will omit the tilde notation in the rest of the proof. 

			Denote $\sigma_{\ell}^{2}=\sigma_{\ell}^{2}(\eps)$ for $\ell=1,2$, which are defined in the statement. Note that 
			\begin{align}\label{eq:F_hessian}
				\nabla^{2}_{\balpha} F(\balpha,\bbeta) = \diag\left( \sum_{j}\psi''(\balpha(i)+\bbeta(j); i \right)  \quad \textup{and} \quad 	\nabla_{\bbeta}^{2} F(\balpha,\bbeta) =  \diag\left( \sum_{i}\psi''(\balpha(i)+\bbeta(j); j \right).
			\end{align} 
			If $(\balpha,\bbeta)$ and $(\balpha',\bbeta')$ are both $\eps$-tame, then so is their convex combination. 
			Hence $F(\balpha, \cdot)$ is $m\sigma_{1}^{2}$-strongly convex and $m\sigma_{2}^{2}$-smooth 
			on the secant line between $(\balpha,\bbeta)$ and $(\balpha',\bbeta)$. By the first-order optimality of $\bbeta_{t+1}$, $\nabla_{\bbeta} F(\balpha_{t},\bbeta_{t+1})=\mathbf{0}$. Hence we deduce the second-order growth property 
			\begin{align}
				F(\balpha_{t},\bbeta_{t}) - F(\balpha_{t},\bbeta_{t+1})  
				\ge \frac{n \sigma_{1}^{2}}{2} \lVert \balpha_{t+1}-\balpha_{t} \rVert^{2}. 
			\end{align}
			Similarly, using $\nabla_{\balpha} F(\balpha_{t+1},\bbeta_{t+1})=\mathbf{0}$, we get 
			\begin{align}
				F(\balpha_{t},\bbeta_{t+1}) - F(\balpha_{t+1},\bbeta_{t+1})  
				&\ge \frac{m \sigma_{1}^{2}}{2} \lVert \bbeta_{t+1}-\bbeta_{t} \rVert^{2}. 
			\end{align}
			Combining the two inequalities above and recalling $\Delta_{t}=F(\balpha_{t},\bbeta_{t}) - F(\balpha^{*},\bbeta^{*})$, we obtain 
			\begin{align}\label{eq:sinkhorn_descent_lemma}
				\Delta_{t} - \Delta_{t+1}  \ge \frac{n \sigma_{1}^{2}}{2} \lVert \balpha_{t+1}-\balpha_{t} \rVert^{2} + \frac{m \sigma_{1}^{2}}{2} \lVert \bbeta_{t+1}-\bbeta_{t} \rVert^{2}.
			\end{align}

			Next, note that $\psi$ is $\sigma_{1}^{2}$-strongly convex and $\sigma_{2}^{2}$-smooth on $[\phi(A_{\eps}),\phi(B_{\eps})]$. In particular, for each $x,y$ in that interval, 
			\begin{align}\label{eq:psi_strong_convex}
				\psi(x)-\psi(y) \ge \psi'(y)(x-y) + \frac{\sigma_{1}^{2}}{2} (x-y)^{2}. 
			\end{align}
			Then $\eps$-tameness, \eqref{eq:psi_strong_convex}, and $\sum_{i}\balpha_{t}=0=\sum_{i}\balpha^{*}(i)$ give
			\begin{align}
				&    \sum_{i,j}   \left[ \psi(\balpha^{*}(i)+\bbeta^{*}(j)) - \psi(\balpha_{t}(i)+\bbeta_{t}(j)) \right] \\
				&\quad \ge \sum_{i,j} \left[ \psi'(\balpha_{t}(i)+\bbeta_{t}(j)) ( \balpha^{*}(i)+\bbeta^{*}(j) - \balpha_{t}(i)-\bbeta_{t}(j) ) \right] + \frac{\sigma_{1}^{2}}{2}  \underbrace{  \| (\balpha^{*}\oplus \bbeta^{*}) - (\balpha_{t}\oplus \bbeta_{t}) \|_{F}^{2} }_{= n \|\balpha^{*}-\balpha \|^{2}+m \|\bbeta^{*}-\bbeta \|^{2} }  \label{eq:pf_sinkhorn1}
			\end{align}
			Also, recall that 
			\begin{align}
				\nabla_{\balpha} F(\balpha,\bbeta) = \left( \r(i)-\sum_{j} \psi'(\balpha(i)+\bbeta(j)) ; i \right) \quad \textup{and} \quad 	\nabla_{\bbeta} F(\balpha,\bbeta) = \left( \c(j)-\sum_{i} \psi'(\balpha(i)+\bbeta(j)) ; j \right).
			\end{align}
			Hence we have 
			\begin{align}\label{eq:pf_sinkhorn2}
				\langle	\nabla_{\balpha} F(\balpha_{t},\bbeta_{t}),\, \balpha^{*}-\balpha_{t}  \rangle + \big\langle \nabla_{\bbeta} F(\balpha_{t},\bbeta_{t}),\, \bbeta^{*}-\bbeta_{t}  \big\rangle  = \langle (\r,\c), \, (\balpha^{*},\bbeta^{*})-(\balpha_{t},\bbeta_{t}) \rangle - 
				I_{1},
			\end{align}
			where $I_{1}$ denotes the first term in \eqref{eq:pf_sinkhorn1}.

			Then we can deduce the following strong-convexity-type
			inequality 
			\begin{align}
				-\Delta_{t}
				& \overset{(a)}{=} \langle (\balpha_{t},\bbeta_{t})-(\balpha^{*},\bbeta^{*}), (\r,\c) \rangle +\sum_{i,j} \left[ \psi(\balpha^{*}(i)+\bbeta^{*}(j)) - \psi(\balpha_{t}(i)+\bbeta_{t}(j))\right]  \\ 
				&\overset{(b)}{\ge} \langle \underbrace{\nabla_{\balpha} F(\balpha_{t},\bbeta_{t})}_{=\mathbf{0}},\, \balpha^{*}-\balpha_{t}  \rangle + \big\langle \nabla_{\bbeta} F(\balpha_{t},\bbeta_{t}),\, \bbeta^{*}-\bbeta_{t}  \big\rangle  + \frac{\sigma_{1}^{2}}{2} \left[ n \|\balpha^{*}-\balpha_{t} \|^{2}+m \|\bbeta^{*}-\bbeta_{t} \|^{2} \right] \\ 
				& \overset{(c)}{\ge} - \frac{1}{2m\sigma_{1}^{2}}   \lVert \nabla_{\bbeta} F(\balpha_{t}, \bbeta_{t}) - \nabla_{\bbeta} F(\balpha_{t},\bbeta_{t+1}) \rVert^{2} 
				\\
				&\overset{(d)}{\ge} - \frac{\sigma_{2}^{4}}{2\sigma_{1}^{2}} \left(  n\| \balpha_{t+1}-\balpha_{t} \|^{2}  + m \| \bbeta_{t+1}-\bbeta_{t} \|^{2}\right), \label{eq:pdf_sinkhorn3}
			\end{align}
			where (a) follows from the defintion of $g^{\r,\c}$ in \eqref{eq:typical_Lagrangian}, (b) follows from combiing \eqref{eq:pf_sinkhorn1} and \eqref{eq:pf_sinkhorn2}. For (c), we used $\nabla_{\bbeta} F(\balpha_{t},\bbeta_{t+1})=\mathbf{0}$  and  Young's inequality $ab \le \frac{a^{2}}{2\lambda } + \frac{\lambda b^{2}}{2}$ for $a=\| \nabla_{\bbeta} F(\balpha_{t},\bbeta_{t}) \|$, $b=\| \bbeta^{*}-\bbeta_{t} \|$,  and  $\lambda=m\sigma_{1}^{2}$, which give
			\begin{align}
				\big\langle \nabla_{\bbeta} F(\balpha_{t},\bbeta_{t}),\, \bbeta^{*}-\bbeta_{t}  \big\rangle  &\ge - \|  \nabla_{\bbeta} F(\balpha_{t},\bbeta_{t})  \| \cdot \| \bbeta^{*}-\bbeta_{t}  \|\\
				&\ge - \frac{1}{2m\sigma_{1}^{2}} \| \nabla_{\bbeta} F(\balpha_{t},\bbeta_{t}) - \nabla_{\bbeta} F(\balpha_{t},\bbeta_{t+1}) \|^{2} - \frac{m\sigma_{1}^{2}}{2} \| \bbeta^{*}-\bbeta_{t}  \|^{2}. 
			\end{align}
			
			Lastly, (d) follows from 
			the Lipschitz continuity of $\nabla_{\bbeta} F$ and including an additional nonpositive term for the lower bound. 
			Combining \eqref{eq:pdf_sinkhorn3} with \eqref{eq:sinkhorn_descent_lemma}, we get 
			\begin{align}
				\Delta_{t} 
				&\le (\sigma_{1}/\sigma_{2})^{4} (\Delta_{t}-\Delta_{t+1}). 
			\end{align}
			Rearranging, this is $\Delta_{t+1}  
			\le \left( 1-  (\sigma_{1}/\sigma_{2})^{4} \right)  \Delta_{t}$. This shows the second inequality in \eqref{eq:sinkhorn_linear_conv}.

			To show the first inequality in \eqref{eq:sinkhorn_linear_conv}, 						first switch $(\balpha^{*},\bbeta^{*})$ and $(\balpha_{t},\bbeta_{t})$ in \eqref{eq:pf_sinkhorn1} to get 
			\begin{align}
				&    \sum_{i,j}   \left[ \psi(\balpha_{t}(i)+\bbeta_{t}(j)) - \psi(\balpha^{*}(i)+\bbeta^{*}(j)) \right] \\
				&\quad \ge  \underbrace{\sum_{i,j} \left[ \psi'(\balpha^{*}(i)+\bbeta^{*}(j)) (  \balpha_{t}(i)+\bbeta_{t}(j) - \balpha^{*}(i)-\bbeta^{*}(j) ) \right]}_{=\langle (\r,\c), \, (\balpha_{t},\bbeta_{t})-(\balpha^{*},\bbeta^{*}) \rangle} + \frac{\sigma_{1}^{2}}{2}    \| (\balpha^{*}\oplus \bbeta^{*}) - (\balpha_{t}\oplus \bbeta_{t}) \|_{F}^{2}, 
				\label{eq:pf_sinkhorn3}
			\end{align}
			where for the summation above we used $\nabla F(\balpha^{*},\bbeta^{*})=\mathbf{0}$ and  \eqref{eq:pf_sinkhorn2} with $(\balpha^{*},\bbeta^{*})$ and $(\balpha_{t},\bbeta_{t})$ swithched. 
			Rearranging terms then gives 
			\begin{align}
				\Delta_{t} \ge 
				\frac{\sigma_{1}^{2}}{2} \| (\balpha^{*}\oplus \bbeta^{*}) - (\balpha_{t}\oplus \bbeta_{t}) \|_{F}^{2}.
			\end{align}
			Combining the above we obtain \eqref{eq:sinkhorn_linear_conv} as claimed.
		\end{proof}

		\section{Proof of applications}
		\label{sec:proof_applications}
		
		In this section, we prove all results stated in Section \ref{sec:statement_applications}.

		\begin{proof}[\textbf{Proof of Theorem \ref{thm:mixture_dist}}]
			
			We first claim that there exists a constant $C=C(\mu,\delta)>0$ such that  for  $\nu$-almost all margins $(\r,\c)$ with an MLE $(\balpha,\bbeta)$, the following holds for each $t\ge 0$:
			\begin{align}\label{eq:thm_mixture_distribution}
				d_{TV}\left( \widetilde{\xi}_{I,J}, \,  \widetilde{\mu}_{\balpha(I)\oplus \bbeta(J)}\right) \le t +  \left[ \exp(C\rho) \, \P\left( Y\in \mathcal{T}_{\rho}(\r,\c)  \right)^{-1}  \land \sup p_{\r,\c}(\cdot)\right] \exp \left(  - 2 t^{2} |I\times J| \right). 
			\end{align}
			To show this claim, fix a measurable set $A\subseteq \R$ and for each matrix $W\in \R^{m\times n}$, denote 
			\begin{align}
				S_{A}(W):=  \frac{1}{|I\times J|} \sum_{(i,j)\in I\times J} \mathbf{1}(W_{ij}\in A). 
			\end{align}    
			Then $\widetilde{\mu}_{I,J}(A)=\E[ S_{A}(X) ]$ and $\widetilde{\xi}_{\balpha(I)\oplus \bbeta(J)}(A)=\E[S_{A}(Y)]$, so 
			\begin{align}
				| \widetilde{\xi}_{I,J}(A) - \widetilde{\mu}_{\balpha(I)\oplus \bbeta(J)}(A) | 
				&\le \E\left[   \left|  S_{A}(X) - \E[S_{A}(Y)]  \right| \right] \\
				&\le t\,  \P\left( \left|  S_{A}(X) - \E[S_{A}(Y)]  \right| \le t\right) +   \P\left( \left|  S_{A}(X) - \E[S_{A}(Y)] \right| > t\right) \\
				&\le t + \exp(C\rho)  \, \P(Y\in \T_{\rho}(\r,\c))^{-1}   \P\left( \left|  S_{A}(Y) - \E[S_{A}(Y)] \right| > t\right),
			\end{align}
			where the last inequality follows from Theorem \ref{thm:transference}. By Hoeffding's inequality, 
			\begin{align}
				\P\left( \left|  S_{A}(Y) - \E[S_{A}(Y)] \right| > t\right) \le 2\exp\left(  -2t^{2}  |I|\cdot |J|\right). 
			\end{align}
			This shows \eqref{eq:thm_mixture_distribution}. 
			
			Now the term in the bracket in \eqref{eq:thm_mixture_distribution} is at most $\exp(F_{\mu}(\r,\c))$  due to Thm. \ref{thm:transference}, Lem. \ref{lem:second_transference_density}, and Thm. \ref{thm:transfer_counting_Leb}. Then  the bound \eqref{eq:thm_mixture_distribution_2} 
			follows from \eqref{eq:F_r_c_def} by choosing $t=O(\sqrt{F_{\mu}(\r,\c)/|I\times J|})$.
		\end{proof}

		\begin{proof}[\textbf{Proof of Theorem \ref{thm:mixture_dist_exact_conti}}]
			
			By Thm. \ref{thm:mixture_dist},  triangle inequality, and the convexity of the TV distance, it suffices to show that 
			\begin{align}\label{eq:thm_mixture_distribution_conti1}
				\E\left[ d_{TV}\left(\mu_{\bar{\balpha}_{m}(U)+\bar{\bbeta}_{n}(V)} , \, \mu_{\balpha(U)+\bbeta(V)}\right)^{4}\right] = O\left(  \lVert (\r,\c) - (\bar{\r}_{m},\bar{\c}_{n}) \rVert_{1} \right).
			\end{align}
			By Pinsker's inequality, for each $\theta,\theta'\in [\phi(A_{\delta}),\phi(B_{\delta})]$, for $C_{\delta}:=\sup_{\phi(A_{\delta}\le w \le \phi(B_{\delta}))} |\psi'(w)|$, 
			\begin{align}
				2  d_{TV}(\mu_{\theta},\mu_{\theta'})^{2} \le  D(\mu_{\theta}\,\Vert \, \mu_{\theta'}) = (\theta-\theta') \psi'(\theta) - \psi(\theta) + \psi(\theta') \le C_{\delta} |\theta-\theta'|. 
			\end{align}
			It follows that 
			\begin{align}
				4 d_{TV}\left(\mu_{\bar{\balpha}_{m}(U)+\bar{\bbeta}_{n}(V)} , \, \mu_{\balpha(U)+\bbeta(V)}\right)^{4} \le C_{\delta}^{2} | \bar{\balpha}_{m}(U)+\bar{\bbeta}_{n}(V) -  \balpha(U)+\bbeta(V) |^{2}.
			\end{align}
			Using the fact that $\E[\bar{\balpha}_{m}]=\E[\balpha]=0$ and  Thm. \ref{thm:main_convergence_conti},  the expectation of the right-hand side is at most 
			\begin{align}
				C_{\delta}^{2}  \left( \Vert \bar{\balpha}_{m}-\balpha \rVert_{2}^{2} + \Vert \bar{\bbeta}_{n}-\bbeta \rVert_{2}^{2} \right) = O( \Vert (\bar{\r}_{m},\bar{\c}_{n}) - (\r,\c) \rVert_{1}).
			\end{align}
		\end{proof}
		
		\begin{proof}[\textbf{Proof of Corollary \ref{cor:clone}}]
			Recall that the sequence of $k$-cloning of $(\r_{0},\c_{0})$ is uniformly $\delta$-tame for some $\delta$ depending only on $(\r_{0},\c_{0})$. Specifically, if $(\balpha_{0},\bbeta_{0})$ is an MLE for $(\r_{0},\c_{0})$, then its concatenation $(\balpha_{0}\otimes \mathbf{1}_{k}, \bbeta_{0}\otimes \mathbf{1}_{k})$ is an MLE for the $k$-cloning of $(\r_{0},\c_{0})$. By symmetry, the entries in the first $k\times k$ submatrix of $X$ have the same distribution. The same holds for $Y\sim \mu_{\balpha\oplus \bbeta}$ since the first $k\times k$ submatrix of $\balpha\oplus \bbeta$ has constant entries, namely, $\balpha_{0}(1)+\bbeta_{0}(1)$. Then the assertion follows from Thm. \ref{thm:mixture_dist}.
		\end{proof}

		Next, we establish the scaling limit of $X$ in the cut norm. We begin with a simple observation of the cut norm.

		\begin{prop}\label{prop:cut_norm_kernle_mx}
			Let $A\in \R^{m\times n}$ and let $W_{A}$ denote the corresponding kernel. Then 
			\begin{align}
				mn	\lVert W_{A} \rVert_{\square}  =  \sup_{\x\in \{0,1\}^{m},\, \y\in \{0,1\}^{n}} |\x^{\top}A\y|.
			\end{align}
		\end{prop}
		
		\begin{proof}
			Write $I_{i}=((i-1)/m, i/m]$ for $i=1,\dots, m$ and $J_{j}=((j-1)/n, j/n]$ for $j=1,\dots,n$. 
			We first claim that the supremum in $\lVert W_{A} \rVert_{\square}$ is attained over a measurable rectangle $S\times T\subseteq [0,1]^{2}$, where $S$ is the disjoint union of some of the intervals $I_{i}$s and $T$ is the disjoint union of some of the intervals $J_{j}$s. This claim will show the first equality in the assertion, and the second equality there follows from the definition of the cut norm of a matrix. 
			
			It remains to show the claim. Recall that $W_{A}$ takes the constant value $A_{ij}$ over the rectangle $I_{i}\times J_{j}$. Then 
			\begin{align}
				\int_{S\times T} W_{A} \,dx\,dy 
				&= \sum_{i} |S\cap I_{i}| \left(  \sum_{j} |T\cap J_{j}| A_{ij} \right). 
			\end{align}
			If the quantity in the parenthesis on the right-hand side is nonnegative, then replacing $S\cap I_{i}$ with $I_{i}$ can only increase the total value; otherwise, replace $S\cap I_{i}$ with $\emptyset$. In this way, we find a new set $\bar{S}$ that is the disjoint union of some of $I_{i}$'s such that 
			\begin{align}
				\int_{S\times T} W_{A} \,dx\,dy \le 	\int_{\bar{S}\times T} W_{A} \,dx\,dy. 
			\end{align}
			Similarly, we can find a disjoint union $\bar{T}$ of some of $J_{j}$'s such that 
			\begin{align}
				\int_{S\times T} W_{A} \,dx\,dy \le 	\int_{\bar{S}\times T} W_{A} \,dx\,dy \le \int_{\bar{S}\times \bar{T}} W_{A} \,dx\,dy.
			\end{align}
			Repeating the same argument with $-W_{A}$, we can deduce the claim. 
		\end{proof}

		We also need the following concentration inequality for the maximum likelihood tilted model for a $\delta$-tame margin.

		\begin{lemma}[Concentration inequalities for the $(\balpha,\bbeta)$-model]\label{lem:hoeffding_geo}
			Let $(\r,\c)$ be an $m\times n$ $\delta$-tame margin and $Y\sim \mu_{\balpha\oplus \bbeta}$, where $(\balpha,\bbeta)$ is an MLE for the margin $(\r,\c)$. Define positive constants $L^{-}$ and $L^{+}$ as in Lem. \ref{lem:concentration_quadratic_form}. 
			Denote $\tilde{Y}=Y-\E[Y]$.  For $s\in [0, L^{-}L^{+}]$,
			\begin{align}
				\P\left( \lVert W_{\tilde{Y}} \rVert_{\square}  \ge s  \right)  &\le 2^{m+n+1} \exp\left( -\frac{s^{2}mn}{2L^{+}}   \right), \label{eq:gen_hoeffding_pf_01} 
			\end{align}  
			
		\end{lemma}
		
		\begin{proof}
			Using Lemma \ref{lem:concentration_quadratic_form}, Proposition \ref{prop:cut_norm_kernle_mx}, and union bound, for each $s\in [0,L^{-}]$, 
			\begin{align}
				\P\left( \lVert W_{\tilde{Y}} \rVert_{\square}  \ge s\right) = \P\left( \lVert \tilde{Y} \rVert_{\square}  \ge s mn \right) 
				&\le \sum_{\x\in \{0,1\}^{m}\, \y\in \{0,1\}^{n}}  \P\left( |\x^{\top}\tilde{Y} \y|   \ge s mn \right)  \le 2^{m+n+1} \exp\left( -\frac{s^{2}mn}{2L^{+}} \right),
			\end{align}
			where the last inequality uses \eqref{eq:gen_hoeffding_pf_000}.
		\end{proof}

		\begin{proof}[\textbf{Proof of Theorem~\ref{thm:main_convergence_conti}}]
			Let $(\balpha_{m},\bbeta_{n})$ be the standard MLE for $(\r_{m},\c_{n})$ and let $Y\sim \mu_{\balpha_{m}\oplus \bbeta_{n}}$. 
			By Lemma \ref{lem:hoeffding_geo}, 
			\begin{align}
				\P( \|  W_{Y}- W_{\E[Y]} \| \ge t) \le  \exp \left(  (m+n+1)\log 2 - \frac{t^{2}mn}{C} \right)
			\end{align}
			for some constant $C=C(\mu,\delta)>0$. Then by the transference principles (Thm. \ref{thm:transference} and Thm. \ref{thm:second_transference}),
			for $\nu$-almost surely for each $t\ge 0$, 
			\begin{align}\label{eq:main_concentration_cut1}
				&	\hspace{1cm} \P_{X}\left( \lVert W_{X} - W_{Z^{\r,\c}} \rVert_{\square} \ge t 
				\right)      \le 2e   \left[ \exp(C\rho) \, \P\left( Y\in \mathcal{T}_{\rho}(\r,\c)  \right)^{-1}  \land \sup p_{\r,\c}(\cdot)\right]   \exp \left( m+n - \frac{t^{2}mn}{C} \right).
			\end{align}
			By using a similar argument as in the proof of the second part of Thm. \ref{thm:mixture_dist}, 
			\begin{align}\label{eq:main_concentration_cut2}
				\P\left( \lVert W_{X} - W_{Z^{\r_{m},\c_{n}}} \rVert_{\square} \, \ge \,  \sqrt{\frac{F_{\mu}(\r,\c)}{mn}} \right) \le 2\exp \left(  - F_{\mu}(\r,\c)  \right). 
			\end{align} 
			By triangle inequality and Thm. \ref{thm:main_convergence_conti}, 
			\begin{align}
				\Vert W_{X} - W^{\r,\c} \rVert_{\square} \le \Vert W_{X} - W_{Z^{\r_{m},\c_{n}} }\rVert_{\square} + 
				C' \| (\bar{\r}_{m},\bar{\c}_{n})-(\r,\c) \|_{1}
			\end{align}
			for some constant $C'=C'(\mu,\delta)$. This is enough to conclude \eqref{eq:exact_conditioning_cor_tight}. By Borel-Cantelli lemma, this yields that $W_{X}\rightarrow W^{\r,\c}$ in cut norm almost surely as $m,n \rightarrow\infty$. 
		\end{proof}

		Lastly in this section, we establish the limit of the ESD of $X$. On top of our transference principles, we will rely heavily on the powerful theory in the random matrix theory literature.

		\begin{proof}[\textbf{Proof of Theorem \ref{thm:ESD}}]
			We first introduce some notations. Without loss of generality, we assume the sequence of $\delta$-tame margins $(\r_{m},\c_{n})$ is indexed by $k$ so that $m=m(k)$ and $n=n(k)$ and simply denote $(\r_{m},\c_{n})=(\r_{k},\c_{k})$. 
			Let $(\balpha_{k},\bbeta_{k})$ denote the standard MLE for margin $(\r_{k},\c_{k})$ and let $(\balpha,\bbeta)$ denote the limiting MLE. Let $Y_{k}\sim \mu_{\balpha_{k}\oplus \bbeta_{k}}$ and  $\widetilde{Y}_{k}:=(s^{*}(m+n))^{-1/2}(Y_{k}-\E[Y_{k}])$. Denote $\hat{\Xi}_{k}:=\tilde{Y}_{k}\tilde{Y}_{k}^{*}$ and let $\hat{\lambda}_{k}$ denote the ESD of $\hat{\Xi}_{k}$. Similarly, let $\lambda_{k}$ denote the ESD of $\hat{\tilde{X}}_{k}\tilde{X}_{k}^{*}$. Let $R_{k}(z):= (\hat{\Xi}_{k} - zI)^{-1}$, $z\in \mathbb{C}\setminus \R$  denote the resolvent of $\hat{\Xi}_{k}$.

			Next, we will argue for the unique existence and boundedness of the solution $\tau$ for the Dyson equation \eqref{eq:wishart_QVE}. By Prop. \ref{prop:typical_kernel_existence} the limiting margin $(\r,\c)$ is $\delta$-tame so 
			$\phi(A_{\delta}) \le \balpha\oplus \bbeta \le \phi(B_{\delta})$.  It follows that the kernel $\psi''(\balpha\oplus \bbeta)$ for the integral operator $S$ is uniformly bounded away from zero and from $\infty$. According to the discussion in \cite[Sec. 3.1]{alt2017local} (with a slight modification for the operator setting),  the problem reduces to showing the same statement about the quadratic vector equation (QVE) $-\frac{1}{\sigma_{\cdot}(z)}=z + \mathcal{S}\sigma_{\cdot}(z)$ with $\mathcal{S}$ the symmetrization of $S$. The hypothesis of \cite[Thm. 2.1]{ajanki2019quadratic} is easily verified, so the solution $\sigma$ is unique and uniformly bounded for every $z\in \mathbb{H}$. 
			
			Let $S_{k}$ denote the integral operator with step-kernel $\psi''(\bar{\balpha}_{k}\oplus \bar{\bbeta}_{k}):(0,1]^{2}\rightarrow \R$, where $(\bar{\balpha}_{k},\bar{\bbeta}_{k})$ is as in Thm. \ref{thm:main_convergence_conti}. 
			By $\delta$-tameness, this kernel is uniformly bounded away from zero and from $\infty$ with the bounds depending only on $\delta$ and $\mu$. By the same argument, the solution $\sigma^{k}$ to the QVE $-\frac{1}{\sigma_{\cdot}^{k}}=\zeta + \mathcal{S}_{k}\sigma_{\cdot}^{k}$ is unique and uniformly bounded. Since $\psi''$ is Lipschitz continuous on bounded domain, by Theorem \ref{thm:main_convergence_conti} 
			\begin{align}
				\| \psi''(\balpha\oplus \bbeta) - \psi''(\bar{\balpha}_{k}\oplus \bar{\bbeta}_{k}) \|_{2}^{2}\le C  \| (\r,\c) - (\bar{\r}_{k},\bar{\c}_{k}) \|_{1} = o(1). 
			\end{align}
			for some constant $C=C(\mu,\delta)>0$. Thus the step variance kernels  $\psi''(\bar{\balpha}_{k}\oplus \bar{\bbeta}_{k})$ converge to $\psi''(\balpha\oplus \bbeta)$ in $L^{2}$. This and the stability theorem \cite[Thm.2.13]{ajanki2019quadratic} yields that $\sigma^{k}$ converges to $\sigma$. Following the discussion in \cite[Thm. 3.1]{alt2017local}, we can transform $\sigma^{k}$s to solutions of  the Dyson equation \eqref{eq:wishart_QVE} with $S_{k}$ in place of $S_{k}$, which we denote as $\tau^{k}$s. This shows $\tau^{k}$ converges to $\tau$ pointwise. 
			
			Now we deduce the convergence of the expected singular value distribution of $\widetilde{Y}_{k}$. The entries of $Y_{k}$ have uniform subexponential norms due to the $\delta$-tameness. This and the uniform boundedness of the variance kernel verifies the hypothesis in \cite{alt2017local}, so applying their Thm. 2.1, we get a unique solution $\tau^{k}$ to \eqref{eq:wishart_QVE} with $S$ replaced by $S_{k}$ and a probability measure $\xi^{k}$ that satisfy \eqref{eq:Stieltjes_measure}. Furthermore, by \cite[Thm. 2.2]{alt2017local}, the Stieltjes transform of $\hat{\lambda}_{k}$, 
			$m^{-1} \tr R_{k}(z)$, is very close to $\langle \tau^{k}(z) \rangle$  in the half-plane $\textup{Im}(z)\ge \eps^{*}$ for any fixed $\eps^{*}>0$. 
			(see the reference for a precise statement). 
			Since we know $\tau^{k}\rightarrow \tau$, it follows that the Stieltjes transforms of $\hat{\lambda}_{k}$ converge pointwise to $\tau$. It follows that $\hat{\lambda}_{k}$ converges weakly to the probability measure $\lambda^{*}$ given by the inverse Stieltjes transform of $\tau$. 
			In fact, by the convergence of local laws in \cite[Thm.2.7, 2.9]{alt2017local}, it also follows that $\E[\hat{\lambda}_{k}]\rightarrow \lambda^{*}$ weakly. At this point, we have shown \textbf{(i)}. 
			
			To show \textbf{(ii)}, we wish to show that $\lambda_{k}$ converges to $\lambda^{*}$ in probability as $k\rightarrow \infty$. Let $d_{\mathcal{W}_{1}}$ denote the Wasserstein-1 distance between probability measures. Fix $\eps>0$. Then by the transference principles that hold under (A1) (see Thm. \ref{thm:transference} and Thm. \ref{thm:second_transference_density_1}), 
			\begin{align}
				\P\left(  d_{\mathcal{W}_{1}}(\lambda_{k}, \E[\hat{\lambda}_{k}])  \ge \eps \right) \le  O( (m\lor n)^{3/2}(\log mn) ) \, \P\left(  d_{\mathcal{W}_{1}}(\hat{\lambda}_{k}, \E[\hat{\lambda}_{k}] )  \ge \eps \right). 
			\end{align}
			We will first consider the special case when $\mu$ has compact support. According to  \cite[Cor. 1.8]{guionnet2000concentration} of Guionnet and Zeitouni (also see the remark following the statement), we have the following sub-Gaussian concentration of ESD of the Wishart matrix $\tilde{Y}_{k}$: 
			\begin{align}\label{eq:subgaussian_ESD_Y}
				\P\left(  d_{\mathcal{W}_{1}}(\hat{\lambda}_{k}, \E[\hat{\lambda}_{k}])  \ge \eps \right) \le \exp(-c(\eps) (m+n)^{2}). 
			\end{align}
			Combining the above, we get 
			\begin{align}
				\P\left(  d_{\mathcal{W}_{1}}(\lambda_{k}, \E[\hat{\lambda}_{k}])  \ge \eps \right) \le \exp(-O((m+n)^{2})). 
			\end{align}
			Since $\E[\hat{\lambda}_{k}]\rightarrow \lambda^{*}$ weakly, it follows that $\lambda_{k}\rightarrow \lambda^{*}$ weakly in probability, as desired. 
			
			It remains to justify the sub-Gaussian ESD concentration for $\tilde{Y}_{k}$ \eqref{eq:subgaussian_ESD_Y} for the more general case when $\mu$ is sub-Gaussian. 
			We remark that the sub-Gaussian ESD concentration for Wishart matrices in \cite[Cor. 1.8]{guionnet2000concentration} is a direct consequence of the similar result for the Wigner matrices stated in \cite[Thm. 1.1]{guionnet2000concentration} using Girko's symmetrization trick (see the discussion above \cite[Cor. 1.8]{guionnet2000concentration}). This result for the Wigner matrices holds when the laws of the entries have common compact support or satisfy the log-Sobolev inequality with a uniform constant. Klochkov and Zhivotovskiy \cite[Lem. 1.4]{klochkov2020uniform} showed that such a result holds in a more general setting where the laws of the entries are uniformly sub-Gaussian. It is an elementary fact that the bounded exponential tilt of a sub-Gaussian distribution is sub-Gaussian with a uniform sub-Gaussian norm. It follows that the sub-Gaussian norms of the entries of $Y$ are uniformly bounded by a constant depending only on $\mu$ and $\delta$. This is enough to conclude.
		\end{proof}

		\begin{proof}[\textbf{Proof of Corollary \ref{cor:quater_circle}}]
			This follows immediately from Thm. \ref{thm:ESD} by the argument sketched above the statement of Cor. \ref{cor:quater_circle}. 
		\end{proof}
		
		\begin{proof}[\textbf{Proof of Corollary \ref{cor:Gaussian}}]
			When $\mu$ is standard Gaussian, then according to the computations in Ex. \ref{ex:typical_gaussian}, the variance matrix $\psi''(\balpha\oplus \bbeta)$ is always the all-ones matrix $\mathbf{1}_{m}\mathbf{1}_{n}^{\top}$. Hence, the Dyson equation \eqref{eq:wishart_QVE} is the same as the one that characterizes the limiting ESD of a $m\times n$ matrix with independent and unit variance entries. The limiting ESD for such matrices is known to be the Marchenko-Pastur distribution given in the statement (see, e.g., \cite{alt2017local}).
		\end{proof}

		\section{Concluding remarks } 
		\label{sec:concluding_remarks}	
		
		Finally, we provide some concluding remarks and discuss several intriguing open problems.

		\vspace{0.1cm}
		\noindent \textbf{Random measure perspective.} Our main  result stated in  Theorem \ref{thm:main_convergence_conti} assumes that the sequence of margins $(\r_{m}, \c_{n})$ converges (after rescaling) in $L^{1}$. This assumption is natural from the perspective of viewing random matrices as random functions. A natural question that we have not investigated in this work is to view the random matrices as random measures on $[0,1]^{2}$ (as done in the permutation limit theory \cite{hoppen2013limits, borga2024large}) and ask if the marginal measures converge weakly, then the joint random measure should also converge to a limiting measure, possibly a deterministic one. We conjecture that if the marginal measures converge to limiting ones with density with respect to the Lebesgue measure, then the joint random measure in the limit should be given by the deterministic measure whose density is the limiting typical kernel as in Theorem \ref{thm:main_convergence_conti}.

		\vspace{0.1cm}
		\noindent \textbf{Conditioning on other constraints.}
		In this work, we only considered conditioning a large random matrix on dense margins, in which the row and column sums are proportional to the number of columns and rows, respectively. Can one develop an analogous concentration and limit theory of random matrices conditioned on sparse margins? Also, note that row and column margin of an $m\times n$ matrix $\X$ are particular linear constraints of the form  $\X\mathbf{1}_{m} = \r$  and $\mathbf{1}_{n} \X = \c^{\top}$. Can one develop an analogous concentration and limit theory of random matrices conditioned on general linear constraints of the form $\X \mathbf{u}=\mathbf{b}$ and $\mathbf{u}' \X =\mathbf{b}'$ for vectors $\mathbf{u}, \mathbf{u}', \mathbf{b},\mathbf{b}'$ of appropriate dimensions? Moreover, can one establish similar results for symmetric random matrices under constraints? For instance, this will give a concentration and limit theory for Wigner matrices with a given row sum, which aligns well with the limit theory of random graphs with given degree sequences \cite{chatterjee2011random}.

		\vspace{0.1cm}
		\noindent \textbf{Large deviations principle.} 
		Dhara and Sen \cite{dhara2022large} established a large deviations principle for random graphs with prescribed degree sequences, building upon foundational work by Chatterjee, Diaconis, and Sly \cite{chatterjee2011random}. We anticipate a similar large deviations principle for random matrices conditioned on given margins, with the information projection \eqref{eq:RM_min_KL} likely serving as the rate function. This is currently an ongoing work by the authors.

		\vspace{0.1cm}
		\noindent \textbf{Schr\"{o}dinger bridge and optimal transport.}
		Drawing an analogy from the connection between the static Schr\"{o}dinger bridges and the entropic optimal transport (see Sec. \ref{sec:schrodinger}), we can also consider an `entropic optimal transport' version of our conditioned random matrix problem. Namely, let $\gamma:\R^{m\times n}\rightarrow \R_{\ge 0}$ be a cost function on $m\times n$ real matrices. Then replacing the base model $\mathcal{R}$ with a probability measure proportional to $e^{-\gamma/\eps} \mathcal{R}$, \eqref{eq:RM_min_KL} becomes 
		\begin{align}\label{eq:RM_min_KL_OT}
			\min_{\mathcal{H}\in \mathcal{P}^{m\times n} } \, \int_{\R^{m\times n}} \gamma(\x) \mathcal{H}(d\x)+ \eps \, D_{KL}(\mathcal{H} \,\Vert \, \mathcal{R}) \quad \textup{subject to} \quad \textup{ $\E_{X\sim \mathcal{H} }[ (r(X), c(X))] = (\r,\c)$}. 
		\end{align}
		It remains to be seen whether the `typical tables' in this generalized setting exhibit interesting structures.

		\section*{Acknowledgements} 
		HL is partially supported by NSF DMS-2206296 and DMS-2232241. SM is partially supported by NSF-2113414. 
		The authors thank Alexander Barvinok, Peter Winkler, Igor Pak, Hongchang Ji, and Rami Tabri for helpful discussions. 
		
				{\footnotesize

					\bibliographystyle{amsalpha}   
					\bibliography{mybib}

				}
				\vspace{1cm}
				\addresseshere

			\end{document}